\documentclass[12pt]{amsart}
\usepackage{amscd,amssymb,amsthm,amsmath,amssymb,mathrsfs,enumerate,enumitem}
\usepackage[matrix,arrow,curve]{xy}
\usepackage[margin=2cm]{geometry}
\usepackage{hyperref}

\newtheorem{theorem}{Theorem}[section]
\newtheorem*{theorem*}{Theorem}

\newtheorem{proposition}[theorem]{Proposition}
\newtheorem*{proposition*}{Proposition}
\newtheorem{lemma}[theorem]{Lemma}
\newtheorem{corollary}[theorem]{Corollary}
\newtheorem*{corollary*}{Corollary}

\newtheorem{question}[theorem]{Question}
\newtheorem*{question*}{Question}
\newtheorem*{theoremA*}{Theorem A}
\newtheorem*{corollaryB*}{Corollary B}
\newtheorem*{theoremC*}{Theorem C}
\newtheorem*{theoremD*}{Theorem D}
\newtheorem*{theoremE*}{Theorem E}
\newtheorem*{theoremF*}{Theorem F}
\newtheorem*{theoremG*}{Theorem G}

\theoremstyle{definition}
\newtheorem{example}[theorem]{Example}
\newtheorem*{example*}{Example}
\newtheorem{definition}[theorem]{Definition}
\newtheorem*{definition*}{Definition}

\theoremstyle{remark}
\newtheorem{remark}[theorem]{Remark}
\newtheorem*{remark*}{Remark}

\makeatletter\@addtoreset{equation}{section} \makeatother

\newcommand{\mumu}{\boldsymbol{\mu}}
\def\PP {\mathbb{P}}
\def\ZZ {\mathbb{Z}}
\def\GG {\mathbb{G}}
\def\CC {\mathbb{C}}
\def\RR {\mathbb{R}}
\def\GL {\mathrm{GL}}

\def\mumu {\boldsymbol{\mu}}

\author{Ivan Cheltsov, Fr\'ed\'eric Mangolte, Constantin Shramov}

\thanks{All varieties are assumed to be projective, normal and defined over
the field of complex numbers $\mathbb{C}$ unless stated otherwise.}

\title{On $G$-birational rigidity of projective spaces}

\address{\emph{Ivan Cheltsov}\newline\textnormal{University of Edinburgh, Edinburgh, Scotland
\newline
\texttt{i.cheltsov@ed.ac.uk}}}

\address{\emph{Fr\'ed\'eric Mangolte}\newline
\textnormal{Aix Marseille University, CNRS, I2M, Marseille, France}
\newline
\textnormal{\texttt{frederic.mangolte@univ-amu.fr}}}

\address{\emph{Constantin Shramov}
\newline
\textnormal{Steklov Mathematical Institute, Moscow, Russia}
\newline
\textnormal{Laboratory of Algebraic Geometry, HSE University, Moscow, Russia}
\newline
\textnormal{\texttt{costya.shramov@gmail.com}}}

\begin{document}

\begin{abstract}
In this paper, we study finite subgroups $G\subset\mathrm{Aut}(\mathbb{P}^n)$ such that $\mathbb{P}^n$ is $G$-birationally rigid.  For each $n\geqslant 3$, we prove that $\mathrm{Aut}(\mathbb{P}^n)$ contains at most finitely many such subgroups up to conjugation.
For $n=4,5,7$, we prove that $\mathbb{P}^n$ is $G$-birationally superrigid if $G$ is a primitive subgroup isomorphic to $\mathrm{PSp}_{4}(\mathbf{F}_3)$,
$\mathrm{PSU}_4(\mathbf{F}_3)\rtimes\mumu_2$, $\mathrm{O}_8^+(\mathbf{F}_2)\rtimes \mumu_2$, respectively.
\end{abstract}

\maketitle
\tableofcontents

\section{Introduction}
\label{section:intro}

\subsection{Birational rigidity}
\label{subsection:intro-rigidity}
The notion of birational rigidity originated in the~work of Iskovskikh and Manin on the~irrationality of smooth quartic threefolds ---
implicitly they proved that every smooth complex hypersurface of degree $4$ in $\mathbb{P}^4$ is birationally superrigid and, in particular, it is not rational.
Recall that a~Fano variety $X$ is said to be birationally rigid if
\begin{enumerate}
\item $X$ is a~Mori fiber space over a~point ($X$ has terminal singularities and $\mathrm{Cl}(X)$ is of rank $1$),
\item and $X$ is the~only Mori fiber space that is birational to $X$.
\end{enumerate}
If $X$ is birationally rigid and $\mathrm{Bir}(X)=\mathrm{Aut}(X)$, we say that $X$ is birationally superrigid.
Being birationally rigid is an obstruction to rationality: this notion is an~extreme opposite of rationality.

At the~moment, birational rigidity is proved for many complex Fano varieties.
However, birationally rigid Fano varieties are rare among all Fanos.
For instance, there are no birationally rigid two-dimensional Fano varieties (del Pezzo surfaces),
and only $3$ out of $105$ deformation families of smooth three-dimensional Fano varieties (Fano threefolds) contain birationally rigid smooth members. The latter are smooth sextic hypersurfaces in $\mathbb{P}(1,1,1,1,3)$ (all of them are birationally superrigid), smooth complete intersections in $\mathbb{P}(1,1,1,1,1,2)$ of a~quadric and a~quartic hypersurfaces (all of them are birationally rigid), and complete intersections in $\mathbb{P}^5$ of a~quadric and a~cubic hypersurfaces (a general member of this deformation family is known to be birationally rigid). Smooth complex hypersurfaces in $\mathbb{P}^n$ of degree $d\leqslant n$ are birationally rigid only for $d=n\geqslant 4$.

Irrationality problems for Fano varieties defined over non-algebraically closed fields and study of finite subgroups of Cremona groups
lead to an equivariant generalization of the~notion of birational rigidity.
To state it, we fix a~Fano variety $X$ such that $X$ has terminal singularities,
and we fix a~finite subgroup $G\subset\mathrm{Aut}(X)$.

\begin{definition}[{\cite[Definition~3.1.1]{CheltsovShramov2015}}]
\label{definition:rigidity}
The Fano variety $X$ is $G$-birationally rigid if the following two conditions are satisfied:
\begin{enumerate}
\item $X$ is a~$G$-Mori fiber space over a~point ($X$ has terminal singularities and $\mathrm{rk}\,\mathrm{Cl}(X)^G=1$),
\item if $Y$ is a $G$-Mori fiber space that is $G$-birational to $X$, then $Y$ is $G$-equivariantly isomorphic to $X$.
\end{enumerate}
The Fano variety $X$ is said to be $G$-birationally superrigid if $X$ is $G$-birationally rigid and
$$
\mathrm{Bir}^G(X)=\mathrm{Aut}^G(X),
$$
where $\mathrm{Aut}^G(X)$ is the subgroup in $\mathrm{Aut}(X)$ consisting of all $G$-automorphisms (the normalizer of the group $G$ in $\mathrm{Aut}(X)$),
and $\mathrm{Bir}^G(X)$ is the subgroup in $\mathrm{Bir}(X)$ consisting of all $G$-birational selfmaps (the normalizer of the group $G$ in $\mathrm{Bir}(X)$).
\end{definition}

\begin{remark}
Definition~\ref{definition:rigidity}
goes back to Manin--Segre theorem: if $X$ is a~smooth del Pezzo~surface such that $(-K_X)^2\leqslant 3$ and $\mathrm{rk}\,\mathrm{Pic}(X)^G=1$, then $X$ is $G$-birationally rigid. For the complete classification of all $G$-birationally rigid del Pezzo surfaces, see \cite{CheltsovTschinkelZhang2026,Yasinsky}.
\end{remark}

The simplest example of a Fano variety of dimension~$n$ is the projective space $\mathbb{P}^n$. Moreover, finite subgroups of the group $\mathrm{Aut}(\mathbb{P}^n)\simeq\mathrm{PGL}_{n+1}(\mathbb{C})$ have been extensively studied for $n\leqslant 7$, see for instance \cite{Blichfeldt1917,Br67,DolgachevIskovskikh2009,Li68,Li70,Li71,Wa69,Wa70}.
Thus, it is natural to pose the~following question:

\begin{question}
\label{question:main}
For which finite subgroups $G\subset\mathrm{Aut}(\mathbb{P}^n)$,
the projective space~$\mathbb{P}^n$ is $G$-birationally rigid?
\end{question}

For $\mathbb{P}^2$ and $\mathbb{P}^3$, the~complete answer to this question has been obtained in \cite{CheltsovShramov2012,CheltsovShramov2014,CheltsovShramov2019,Sakovics2019}, and we know very simple geometric criteria for being $G$-birationally rigid in these two cases:

\begin{theorem}[{\cite{Sakovics2019}}]
\label{theorem:Sakovich}
Let $G$ be a~finite subgroup in $\mathrm{PGL}_{3}(\mathbb{C})$.
Then $\mathbb{P}^2$ is $G$-birationally rigid if and only if the~following conditions are satisfied:
$G$ does not fix points in $\mathbb{P}^2$, $G\not\simeq\mathfrak{A}_4$, and $G\not\simeq\mathfrak{S}_4$.
\end{theorem}

\begin{theorem}[{\cite{CheltsovShramov2012,CheltsovShramov2014,CheltsovShramov2019}}]
\label{theorem:CheltsovShramov}
Let $G$ be a~finite subgroup in $\mathrm{PGL}_{4}(\mathbb{C})$.
Then $\mathbb{P}^3$ is $G$-birationally rigid if and only if the~following conditions are satisfied:
\begin{itemize}
\item[$(\mathrm{1})$] $G$ does not fix points in $\mathbb{P}^3$,
\item[$(\mathrm{2})$] $\mathbb{P}^3$ does not contain $G$-invariant pair of skew lines,
\item[$(\mathrm{3})$] there exists no $G$-orbit in $\mathbb{P}^3$ of length $4$,
\item[$(\mathrm{4})$] $G\not\simeq\mathfrak{A}_5$ and $G\not\simeq\mathfrak{S}_5$.
\end{itemize}
\end{theorem}

In particular, we see that $\mathrm{PGL}_{3}(\mathbb{C})$ contains infinitely many finite subgroups $G$ up to conjugation such that $\mathbb{P}^2$ is $G$-birationally rigid \cite{DolgachevIskovskikh2009}. Indeed, if $G$ is the~subgroup in $\mathrm{PGL}_{3}(\mathbb{C})$ generated by
$$
\begin{pmatrix}
e^{\frac{2\pi \sqrt{-1}}{m}} & 0 & 0\\
0 & 1 & 0 \\
0 & 0 & 1
\end{pmatrix},
\begin{pmatrix}
1 & 0 & 0 \\
0 & e^{\frac{2\pi \sqrt{-1}}{m}} & 0\\
0 & 0 & 1
\end{pmatrix},
\begin{pmatrix}
0 & 0 & 1 \\
1 & 0 & 0 \\
0 & 1 & 0
\end{pmatrix},
\begin{pmatrix}
0 & 1 & 0 \\
1 & 0 & 0 \\
0 & 0 & 1
\end{pmatrix},
$$
then $G\simeq \mumu_m^2\rtimes\mathfrak{S}_3$, where $\mumu_m$ denotes the cyclic group
of order~$m$. We see that $\mathbb{P}^2$ is $G$-birationally rigid for $m\geqslant 3$ by Theorem~\ref{theorem:Sakovich}.
On~the~other hand, it follows from Theorem~\ref{theorem:CheltsovShramov} that $\mathrm{PGL}_{4}(\mathbb{C})$ contains finitely many subgroups $G$ (up to conjugation) such that $\mathbb{P}^3$ is $G$-birationally rigid. Using the classification of
finite subgroups of~\mbox{$\mathrm{PGL}_{4}(\mathbb{C})$} provided in~\cite{Blichfeldt1917}
(cf.~\mbox{\cite[Appendix A]{CheltsovShramov2019}}), it is not difficult to list all of them.

The geometric conditions $(\mathrm{1})$, $(\mathrm{2})$, $(\mathrm{3})$ in Theorem~\ref{theorem:CheltsovShramov} mean that the subgroup $G$ is \textit{primitive}.
Namely, we recall the following terminology.

\begin{definition}[{cf. Definition~\ref{definition:primitive-transitive-permutation-group}}]
\label{definition:primitive-linear-group}
A finite subgroup $\widehat{G}\subset\mathrm{GL}_{n+1}(\mathbb{C})$ is said to be primitive if there exists no non-trivial decomposition
$$
\mathbb{C}^{n+1}=\bigoplus_{i=1}^{r}V_{i}
$$
such that for any $g\in\widehat{G}$ and any $i$ there is some $j=j(g)$ such that $g(V_{i})=V_{j}$.
Similarly, a finite subgroup $G\subset \mathrm{PGL}_{n+1}(\mathbb{C})$ is said to be primitive
if the subgroup $G$ is an image of some primitive finite subgroup $\widehat{G}\subset\mathrm{GL}_{n+1}(\mathbb{C})$ via the natural projection.
\end{definition}

In Section~\ref{section:toric}, we will prove the following theorem.

\begin{theoremA*}
If $G$ is a~finite subgroup in $\mathrm{PGL}_{n+1}(\mathbb{C})$ for $n\geqslant 3$ such that $\mathbb{P}^n$ is $G$-birationally rigid, then the subgroup $G$ is primitive.
\end{theoremA*}

The proof of Theorem~A uses geometric and combinatorial results about toric symmetry of~$\mathbb{P}^n$, which are presented in Section~\ref{section:toric}. It should be pointed out that, except for Theorem~A, there are other obstructions for
$\PP^n$ to be $G$-birationally rigid, see e.g. Lemmas~\ref{lemma:Pn-pencil}
and~\ref{lemma:Pn-ci}.

Recall that the group~\mbox{$\mathrm{PGL}_{n+1}(\mathbb{C})$} contains only finitely many primitive finite subgroups up to conjugation, see for
instance~\mbox{\cite[\S\S61,73,74]{Blichfeldt1917}}. Thus, Theorem~A gives

\begin{corollaryB*}
Fix  $n\geqslant 3$. Up to conjugation, the group $\mathrm{PGL}_{n+1}(\mathbb{C})$ contains finitely many finite subgroups $G$
such that $\mathbb{P}^n$ is $G$-birationally rigid.
\end{corollaryB*}

\begin{remark}
For higher-dimensional Fano varieties different from $\PP^n$, an analogue of Corollary~B does not always hold.
For instance, if $X\simeq \PP^1\times\PP^1\times\PP^1$, then it follows from \cite[Corollary 1.4]{CheltsovDuboulozKishimoto2023} that
$\mathrm{Aut}(X)$ contains infinitely many finite subgroups $G$ up to isomorphism
such that  $\mathrm{rk}\,\mathrm{Cl}(X)^G=1$ and $X$ is $G$-birationally superrigid.
\end{remark}

Using Theorem~A, we can list all possibilities for $G\subset\mathrm{PGL}_{n+1}(\mathbb{C})$
such that~$\mathbb{P}^n$ has a chance to be $G$-birationally rigid in the case when $n$ is small. For instance, in Proposition~\ref{proposition:8-groups},
we will see that up to conjugation there are at most $8$ finite
subgroups $G\subset\mathrm{PGL}_{5}(\mathbb{C})$ such that~$\PP^4$ can a priori be $G$-birationally rigid. One of them is the subgroup $G\simeq\mathrm{PSp}_{4}(\mathbf{F}_3)$ that
preserves the famous Burkhardt quartic hypersurface.
In Section~\ref{section:P4}, we will prove the following result.

\begin{theoremC*}
Let $G$ be a~finite primitive subgroup in $\mathrm{PGL}_{5}(\mathbb{C})$ such that $G\simeq \mathrm{PSp}_{4}(\mathbf{F}_3)$. Then $\mathbb{P}^4$ is $G$-birationally superrigid.
\end{theoremC*}

It would be interesting to find out whether $\mathbb{P}^4$ is $G$-birationally rigid in each of the $8$ cases listed in Proposition~\ref{proposition:8-groups}.
Unfortunately, our proof of Theorem~C cannot be adapted to show that $\mathbb{P}^4$ is $G$-birationally rigid in the remaining~$7$ cases, cf. Remark~\ref{remark:indus}.

In Section~\ref{section:P5}, we will show that, up to conjugation, $\mathrm{PGL}_6(\mathbb{C})$ has at most $12$ finite
primitive subgroups $G$ such that~$\PP^5$ can a priori be $G$-birationally rigid (they are listed in Corollary~\ref{corollary:n-5-rigid}).
Moreover, we prove that $\PP^5$ is $G$-birationally superrigid for the largest such subgroup:

\begin{theoremD*}
Let $G$ be a~finite primitive subgroup in $\mathrm{PGL}_{6}(\mathbb{C})$ such that $G\simeq\mathrm{PSU}_4(\mathbf{F}_3)\rtimes\mumu_2$. Then $\mathbb{P}^5$ is $G$-birationally superrigid.
\end{theoremD*}

At the moment, we do not know whether $\mathbb{P}^5$ is $G$-birationally rigid in the remaining~$11$ cases. We find the next question particularly interesting.

\begin{question}[{cf. \cite{CheltsovShramov2011-2}}]
\label{question:HaJ}
Let $G\subset\mathrm{PGL}_6(\mathbb{C})$ be the simple Hall--Janko group.
Is it true that $\PP^5$ is $G$-birationally rigid?
\end{question}

In Section~\ref{section:P6}, we will discuss possible candidates $G\subset \mathrm{PGL}_{7}(\mathbb{C})$ such that $\mathbb{P}^6$ is $G$-birationally rigid.
In Section~\ref{section:P7}, we will prove the following result.

\begin{theoremE*}
Let $G$ be a~finite primitive subgroup in $\mathrm{PGL}_{8}(\mathbb{C})$ such that
$G\simeq\mathrm{O}_8^+(\mathbf{F}_2)\rtimes \mumu_2$. Then $\mathbb{P}^7$ is $G$-birationally superrigid.
\end{theoremE*}

We can naturally generalize Definition~\ref{definition:rigidity} for Fano varieties that are defined over an algebraically non-closed field~$\Bbbk$,
and we can pose Question~\ref{question:main} for~$\mathbb{P}^n_{\Bbbk}$.
The next theorem gives
an analogue  of Theorem~A over the field of real numbers
(but, somewhat surprisingly, its proof is much simpler).

\begin{theoremF*}
\label{proposition:real-primitive}
Let $G\subset \mathrm{PGL}_{n+1}(\RR)$ be a finite group for $n\geqslant 2$ such that
$\mathbb{P}^n_{\mathbb{R}}$ is \mbox{$G$-birationally} rigid.
Then $G$ is a primitive subgroup in $\mathrm{PGL}_{n+1}(\mathbb{C})$.
\end{theoremF*}

In Lemma~\ref{lemma:R-dim-2}, we will show that $\mathrm{PGL}_{3}(\mathbb{R})$ contains a unique subgroup $G$ up to conjugation such that $\mathbb{P}^2_{\mathbb{R}}$ is $G$-birationally rigid. This subgroup is isomorphic to $\mathfrak{A}_5$. Similarly, in Proposition~\ref{proposition:R-dim-3}, we will describe finite subgroups $G\subset\mathrm{PGL}_{4}(\mathbb{R})$ such that~$\mathbb{P}^3_{\mathbb{R}}$ is $G$-birationally rigid. Using the classification of primitive finite subgroups in $\mathrm{PGL}_{n+1}(\mathbb{C})$ for $n\in\{4,5,6\}$, we will prove the following surprising result.

\begin{theoremG*}
Let $G$ be a finite subgroup in  $\mathrm{PGL}_{n+1}(\mathbb{R})$. Suppose that $n\in\{4,5,6\}$. Then $\mathbb{P}^n_{\mathbb{R}}$ is not \mbox{$G$-birationally} rigid.
\end{theoremG*}

On the other hand, it follows from Theorem~E that $\mathrm{PGL}_{8}(\mathbb{R})$ contains a finite subgroup $G$ that $\mathbb{P}^7_{\mathbb{R}}$ is $G$-birationally rigid.
Indeed, the Weyl group $W(\mathbb{E}_8)$ has a natural $8$-dimensional representation that is defined over $\mathbb{Z}$,
and its image in $\mathrm{PGL}_{8}(\mathbb{R})$ is a finite primitive subgroup $G\simeq\mathrm{O}_8^+(\mathbf{F}_2)\rtimes \mumu_2$,
so $\mathbb{P}^7_{\mathbb{R}}$ is $G$-birationally superrigid by Theorem~E.

\begin{remark}
\label{remark:Sonina}
If $X$ is a non-trivial $\Bbbk$-form of $\mathbb{P}^n$, where $n+1\geqslant 3$ is a prime number different from the characteristic of~$\Bbbk$,
then it follows from \cite[Theorem~A.6]{Sonina} that $X$ is not $G$-birationally rigid for any finite subgroup $G\subset\mathrm{Aut}(X)$.
\end{remark}

\medskip
\noindent
\textbf{Plan of the paper.}
In Section~\ref{section:preliminaries}, we recall some definitions, establish auxiliary facts concerning  primitive and imprimitive finite groups acting on projective spaces, and present a couple of simple obstructions for  $\PP^n$ to be $G$-birationally rigid. In Section~\ref{section:toric}, we discuss equivariant birational geometry of the projective space  $\mathbb{P}^n$ with respect to the action of an infinite group that contains a maximal torus in $\mathrm{PGL}_{n+1}(\mathbb{C})$ as a normal subgroup and, as an application, we prove Theorem~A modulo  one technical result (Theorem~\ref{theorem:toric-2}), which is proved later in Section~\ref{section:technical}. In Section~\ref{section:P4}, we discuss primitive subgroups of $\mathrm{PGL}_5(\mathbb{C})$  and prove Theorem~C.
In Section~\ref{section:P5}, we discuss primitive subgroups of $\mathrm{PGL}_6(\mathbb{C})$ and prove Theorem~D.
In Section~\ref{section:P6}, we discuss primitive subgroups of $\mathrm{PGL}_7(\mathbb{C})$, and show that $\PP^6$ is not $G$-birationally rigid for many such subgroups~$G$.
In Section~\ref{section:P7}, we prove Theorem~E. In Section~\ref{section:real}, we discuss finite groups acting on real projective spaces, and we prove Theorems~E and F. In Appendix~\ref{section:combinatorics}, we present combinatorial results that are used in the proof of Theorem~\ref{theorem:toric-2}. Finally, Appendix~\ref{app:toric-fano} contains a proof that the toric variety $\mathscr{X}$ constructed in Remark~\ref{remark:do-nor-understand-M} is a Fano variety with terminal singularities.

\medskip
\noindent
\textbf{Notation.}
Throughout the paper, we denote by $\mumu_n$ the cyclic group of order $n$, and by  $\mathfrak{D}_n$ the dihedral group of order~$2n\geqslant 6$.
Similarly, we denote by $\mathfrak{S}_n$ the symmetric group of degree $n$, and we denote by $\mathfrak{A}_n$ its alternating normal subgroup.
By~$\mathrm{SU}_n(\mathbf{F}_q)$ we denote the group of all unitary matrices with determinant~$1$
in~\mbox{$\mathrm{GL}_n(\mathbf{F}_{q^2})$}, and by~$\mathrm{PSU}_n(\mathbf{F}_q)$ we denote
its image in~\mbox{$\mathrm{PGL}_n(\mathbf{F}_{q^2})$}.
If $G_1,\ldots,G_r$ are subgroups of a group $\Gamma$, and $g_1,\ldots,g_s$ are elements of $\Gamma$, then we denote by~\mbox{$\langle G_1,\ldots,G_r,g_1,\ldots,g_s\rangle$} the subgroup of~$\Gamma$ generated by $G_1,\ldots,G_r$ together with~\mbox{$g_1,\ldots,g_s$}.

\medskip
\noindent
\textbf{Acknowledgements.}
We are very grateful to Maria Grechkoseeva, Andrea Petracci, Andrey Vasilyev, and Zhijia Zhang for many useful discussions. We also thank CIRM, Luminy, for the hospitality provided during a semester-long Morlet Chair and for creating a perfect work environment. Ivan Cheltsov has been supported by Simons Collaboration grant \emph{Moduli of varieties}. The work of Constantin Shramov was performed at the Steklov International Mathematical Center and supported by the Ministry of Science and Higher Education of the Russian Federation (agreement no.~075-15-2022-265) and by the HSE University Basic Research Program.

\medskip
\noindent
\subsection*{AI Disclosure}
In the first version of this paper (submitted to Arxiv and completely human-written), we stated an over-optimistic conjecture that $\mathbb{P}^n_{\mathbb{R}}$ is never $G$-birationally rigid for any finite subgroup $G\subset\mathrm{PGL}_{n+1}(\mathbb{R})$ with $n\geqslant 4$.
After this, we decided to test this conjecture using Chat GPT-5.6~Sol by OpenAI.
As a result, AI suggested a candidate counter-example to the conjecture, which resulted in Theorem E showing that $G$-birational rigidity is actually possible for~$\mathbb{P}^7_{\mathbb{R}}$. The proof of Theorem~E is based on Lemmas~\ref{lemma:Q-degree-bounds} and \ref{lemma:degree-bounds} whose proofs were suggested by~AI. Using the same approach, we proved Theorem~D. Moreover, AI found a gap in the original proof of Proposition~\ref{proposition:D} used in the proof of Theorem~F, and suggested a way to remove this gap. Finally, AI provided a proof that the toric variety $\mathscr{X}$ constructed in Remark~\ref{remark:do-nor-understand-M} is a terminal Fano variety, see Appendix~\ref{app:toric-fano}. However, we decided not to use this result in the proof of our theorems.

\section{Preliminaries}
\label{section:preliminaries}

In this section we recall some definitions and establish auxiliary facts concerning
primitive and imprimitive finite groups acting on projective spaces.

Let $G$ be a~finite subgroup in $\mathrm{PGL}_{n+1}(\CC)$. Then there exists
a finite subgroup $\widehat{G}\subset\mathrm{GL}_{n+1}(\CC)$ such that $\widehat{G}$ is mapped surjectively to $G$ via the~natural epimorphism $\mathrm{GL}_{n+1}(\CC)\to \mathrm{PGL}_{n+1}(\CC)$. Recall
the following terminology.

\begin{definition}[{cf. Definitions~\ref{definition:primitive-linear-group} and~\ref{definition:primitive-transitive-permutation-group}}]
\label{definition:transitive-imprimitive}
The group $G$ is said to be transitive if the embedding
$$
\widehat{G}\hookrightarrow\mathrm{GL}_{n+1}(\CC)
$$
is an irreducible representation. Otherwise, $G$ is intransitive. Similarly,  $G$ is said to be imprimitive if
there exists a~non-trivial decomposition
\begin{equation}\label{eq:imprimitive-decomposition}
\CC^{n+1}=\bigoplus_{i=1}^{r}V_{i}
\end{equation}
such that for any $g\in\widehat{G}$ and any $i$ we have $g(V_{i})=V_{j}$ for some $j=j(g)$.
If $G$ is transitive and not imprimitive, $G$ is said to be primitive.
\end{definition}

It is easy to see that the properties of the group $G$ described in Definition~\ref{definition:transitive-imprimitive} do not depend on the choice of the lifting $\widehat{G}$ of $G$ to $\mathrm{GL}_{n+1}(\CC)$. The following result is a baby version of Theorem~A.

\begin{lemma}
\label{lemma:intransitive}
Suppose that $\mathbb{P}^n$ is $G$-birationally rigid. Then the group~$G$ is transitive.
\end{lemma}

\begin{proof}
Suppose that $G$ is not transitive. Then there exists a $G$-invariant linear subspace $\Lambda\subset\mathbb{P}^n$ of dimension~\mbox{$k\leqslant n-2$}.
Let $\pi\colon Y\to \mathbb{P}^n$ be the blow up of $\Lambda$, and let $m=n-k-1$. Then there exists a $G$-Sarkisov link
$$
\xymatrix{
& Y\ar@{->}[ld]_\pi\ar@{->}[rd]^{\phi} & \\
\mathbb{P}^n && \mathbb{P}^m}
$$
where $\phi$ is a $\mathbb{P}^{k+1}$-bundle, which is a $G$-Mori fiber space with a positive-dimensional base.
This contradicts our assumption that $\mathbb{P}^n$ is $G$-birationally rigid.
\end{proof}

Thus, if $\mathbb{P}^n$ is $G$-birationally rigid, then $G$ must be transitive. Theorem~A says that $G$ must also be primitive. In order to prove this in Section~\ref{section:toric}, we need the following result that says that if $\mathbb{P}^n$ is $G$-birationally rigid and $G$ is not primitive, then $G$ must have an orbit of length $n+1$.

\begin{lemma}[{see \cite[(4.4.2)]{Kollar-Aut}}]
\label{lemma:intransitive-2}
Suppose that $G$ is transitive and imprimitive,
so that there exists a~non-trivial decomposition~\eqref{eq:imprimitive-decomposition} such that for any~\mbox{$g\in\widehat{G}$} and any $i$ we have $g(V_{i})=V_{j}$ for some~\mbox{$j=j(g)$}.
If  $\mathbb{P}^n$ is $G$-birationally rigid, then $r=n+1$.
\end{lemma}

\begin{proof}
Note that $r$ divides $n+1$. Set $m=\frac{n+1}{r}-1$.
Suppose that $r<n+1$. Then, by definition, there exists a collection of~$r$ linear subspaces $\Lambda_1,\ldots,\Lambda_r\subset \mathbb{P}^n$ of dimension $m$ such that their union
spans $\mathbb{P}^n$, each~$\Lambda_i$ is disjoint from the linear span of the union of $\Lambda_j$ with $j\neq i$, and the group $G$ permutes~\mbox{$\Lambda_1,\ldots,\Lambda_r$}.
For each $i\in\{1,\ldots,r\}$, consider the linear projection
$$
\psi_i\colon \mathbb{P}^n\dasharrow \Lambda_i
$$
from the span of the union of all $\Lambda_j$ with $j\neq i$. Set $\Lambda=\Lambda_1\times\ldots\times\Lambda_r$, and consider the rational map
$$
\psi=\psi_1\times\ldots\times \psi_r\colon \mathbb{P}^n\dasharrow \Lambda.
$$
Then $\psi$ is dominant and $G$-equivariant. Furthermore,
observe that a fiber of $\psi_i$ is a linear subspace in $\PP^n$ (recall that a fiber
of a rational map is defined as the closure of a fiber of the restriction of the map
to an open subset where it is regular). A fiber of $\psi$ is an intersection of fibers of
$\psi_i$, so that it is also a linear subspace in $\PP^n$. Now we see that
$\psi$ is a $G$-equivariant rational map over a positive-dimensional
base whose fibers are rationally connected.
This implies that $\PP^n$ is $G$-birational to a $G$-Mori fiber space with a positive-dimensional base. Hence, in particular, $\mathbb{P}^n$ is not $G$-birationally rigid.
\end{proof}

In the remaining part of this section, we present several obstructions for $\mathbb{P}^n$ to be $G$-birationally rigid, which will be used later in Sections~\ref{section:P4}, \ref{section:P5}, and \ref{section:P6}.

\begin{lemma}
\label{lemma:Pn-pencil}
Suppose that $|\mathcal{O}_{\mathbb{P}^n}(d)|$ contains a $G$-invariant pencil $\mathcal{P}$ for some $d\leqslant n$.
Then~$\mathbb{P}^n$ is not $G$-birationally rigid.
\end{lemma}

\begin{proof}
Replacing $\mathcal{P}$ by its mobile part, we may assume that $\mathcal{P}$ is mobile.
The pencil $\mathcal{P}$ gives a $G$-equivariant rational map $\psi\colon \mathbb{P}^n\dasharrow\mathbb{P}^1$.
Equivariantly resolving the~indeterminacy of this map and the singularities of the resulting variety, we obtain a commutative diagram
$$
\xymatrix{
& Y\ar@{->}[ld]_\pi\ar@{->}[rd]^{\phi} & \\
\mathbb{P}^n\ar@{-->}[rr]^\psi  && \PP^1}
$$
where $Y$ is smooth. Furthermore,
passing to the Stein factorization of $\phi$, we may assume that
a general fiber of $\phi$ is irreducible. Observe that the fibers of $\phi$ are
strict transforms of the elements of the pencil~$\mathcal{P}$, which are (possibly singular) hypersurfaces of degree $d$.
Such hypersurfaces are uniruled by \cite[Theorem~1]{MiyaokaMori}.
Therefore, applying equivariant relative
Minimal Model Program over~$\mathbb{P}^1$ to~$Y$,
we obtain a $G$-birational map from $\mathbb{P}^n$ to a $G$-Mori fiber space with positive-dimensional base, which implies that~$\mathbb{P}^n$ is not $G$-birationally rigid.
\end{proof}

For the simplest three-dimensional application of Lemma~\ref{lemma:Pn-pencil}, see \cite[Example~1.2]{CheltsovSarikyan2023}. For higher-dimensional applications, see Sections~\ref{section:P4}, \ref{section:P5}, and \ref{section:P6} below.

\begin{lemma}
\label{lemma:Pn-ci}
Suppose that~$\mathbb{P}^n$ contains a $G$-irreducible reduced complete intersection $X=F_{d_1}\cap F_{d_2}$ such that $X$ has hypersurface singularities
and $d_1<d_2\leqslant n$, where~$F_{d_1}$ and $F_{d_2}$ are hypersurfaces in $\mathbb{P}^n$ of degree $d_1$ and $d_2$, respectively. Then $\mathbb{P}^n$ is not $G$-birationally rigid.
\end{lemma}

\begin{proof}
Let $\pi\colon V\to \mathbb{P}^n$ be the blow up of the complete intersection $X$, and let $\widetilde{F}_{d_1}$ be the strict transform on $V$ of the hypersurface $F_{d_1}$. Then $V$ has only $\mathrm{cA}$-type singularities of codimension at least $3$. By \cite[1.42]{Kollar-Sing}, this implies that $V$ has terminal singularities. On the other hand, we have the following $G$-equivariant diagram:
$$
\xymatrix{
&V\ar@{->}[dl]_{\pi}\ar@{->}[dr]^{\eta}&\\
\mathbb{P}^n && Y}
$$
where $Y$ is a (singular) Fano variety, and $\eta$ is a birational morphism that contracts $\widetilde{F}_{d_1}$ to a~point of the variety~$Y$, which is singular unless $(d_1,d_2)=(1,2)$ and $X$ is smooth. Observe that
$$
K_V\sim_{\mathbb{Q}}\eta^*K_Y+\frac{n+1-d_2}{d_2-d_1}\widetilde{F}_{d_1}.
$$
This implies that $Y$ has terminal singularities. Moreover, by construction we have $\mathrm{rk}\,\mathrm{Cl}^G(Y)=1$ and $Y\not\simeq \mathbb{P}^n$, so that $\mathbb{P}^n$ is not $G$-birationally rigid.
\end{proof}

Let us present the simplest application of Lemma~\ref{lemma:Pn-ci}
(for another application, see the proof of Corollary~\ref{corollary:n-5-PSp4}).

\begin{example}
\label{example:Bring}
Consider the action of the group $\mathfrak{S}_5$ on $\mathbb{P}^3$ given by its irreducible $4$-dimensional representation. Then $\mathbb{P}^3$ contains a $\mathfrak{S}_5$-invariant smooth complete intersection $X=F_{2}\cap F_{3}$ such that $F_{2}$ and $F_{3}$ are $\mathfrak{S}_5$-invariant quadric and cubic surfaces, respectively. Note that $X$ is a smooth curve of genus $4$ known as the Bring's curve \cite{Edge}. By Lemma~\ref{lemma:Pn-ci}, the projective space $\PP^3$ is not $\mathfrak{S}_5$-birationally rigid. Moreover, arguing as in the proof of Lemma~\ref{lemma:Pn-ci}, we obtain the following $\mathfrak{S}_5$-equivariant diagram:
$$
\xymatrix{
&V\ar@{->}[dl]_{\pi}\ar@{->}[dr]^{\eta}&\\
\mathbb{P}^3 && Y}
$$
where $\pi$ is the blow up of the curve $X$, the threefold $Y$ is a cubic hypersurface in $\mathbb{P}^4$ with one isolated ordinary double point, and $\eta$ is the blow up of this point. Note that $|\mathcal{O}_{\mathbb{P}^3}(d)|$ does not contain $\mathfrak{S}_5$-invariant pencils for $d\leqslant 3$, so Lemma~\ref{lemma:Pn-pencil} is not applicable here.
On the other hand, there are other ways to see that $\PP^3$  is not $\mathfrak{S}_5$-birationally rigid.
For instance, one can show that~$\PP^3$ is $\mathfrak{S}_5$-birationally equivalent to a $\mathfrak{S}_5$-conic bundle \cite[Example~2.7]{CheltsovShramov2019}.
\end{example}

The next lemma will be used in Section~\ref{section:P5}.

\begin{lemma}
\label{lemma:P5-net}
Suppose that $n=5$, the action of $G$ on $\PP^5$ is transitive, and the linear
system~\mbox{$|\mathcal{O}_{\mathbb{P}^5}(2)|$} contains a $G$-invariant two-dimensional
linear subsystem $\mathcal{M}$. Suppose also that
the action of~$G$ on $\mathcal{M}\simeq \PP^2$ is faithful and transitive,
and $G$ cannot act faithfully on a rational curve. Then~$\PP^5$ is not $G$-birationally rigid.
\end{lemma}

\begin{proof}
Since there are no $G$-invariant hyperplanes in $\PP^5$, we see that the linear system~$\mathcal{M}$ is mobile.

Let $\chi\colon \mathbb{P}^5\dasharrow\mathbb{P}^2$ be the rational map given by the net $\mathcal{M}$. Then we obtain a $G$-equivariant commutative diagram
$$
\xymatrix{
&X\ar@{->}[dl]_{\pi}\ar@{->}[dr]^{\eta}\\%
\mathbb{P}^5\ar@{-->}[rr]^{\chi}&&\mathbb{P}^2}
$$
where $\pi$ is a $G$-equivariant resolution of indeterminacies of the rational map $\chi$ with smooth~$X$ (cf.~the~proof of Lemma~\ref{lemma:Pn-pencil}). Recall that $G$ acts faithfully on $\mathbb{P}^2$. Hence the morphism~$\eta$ is surjective, because otherwise
the image of $\eta$ would be either a point, or a rational curve. On the other hand,
by our assumptions there are no $G$-invariant points in $\PP^2$, and
$G$ cannot act faithfully on a rational curve.

Let $F$ be a general fiber of the morphism $\eta$. Then $\pi$ induces a birational morphism $F\to\pi(F)$, and
$$
\pi(F)\subset M_1\cap M_2
$$
for two distinct quadric hypersurfaces $M_1$ and $M_2$ in the net $\mathcal{M}$. Hence, every irreducible component of $\pi(F)$ is an irreducible component of the intersection $M_1\cap M_2$. This implies that all irreducible components
of $F$ are uniruled. Thus, applying $G$-equivariant Minimal Model Program to $X$ over $\mathbb{P}^2$, we obtain a $G$-birational map from $X$ to a $G$-Mori fibred space with a positive dimensional base. In particular, $\mathbb{P}^5$ is not $G$-birationally rigid.
\end{proof}

Let us conclude this section with the following lemma, which will be used in the proofs of Theorems~D and E in Sections~\ref{section:P5} and~\ref{section:P7}, respectively.
For the basic information concerning weighted projective spaces, the reader is referred
to~\cite{Dolgachev} or~\cite{CostyaVitya}.

\begin{lemma}
\label{lemma:weighted-degree}
Let $\mathbb{P}=\mathbb{P}(w_0,\ldots,w_d)=\mathrm{Proj}(\mathbb{C}[y_0,\ldots,y_d])$ be a well formed weighted projective space of dimension $d$ such that $w_0\leqslant\ldots\leqslant w_d$,
and let $Y\subset\mathbb{P}$ be an irreducible subvariety of dimension $n$. Then
$$
\big(\mathcal{O}_{\mathbb{P}}(1)\big)^n\cdot Y\geqslant \frac{1}{w_{d-n}w_{d-n+1}\ldots w_d}.
$$
Moreover, if $Y$ is not contained in the subset of $\PP$ defined by $y_0=0$, then
$$
\big(\mathcal{O}_{\mathbb{P}}(1)\big)^n\cdot Y\geqslant \frac{1}{w_0w_{d-n+1}w_{d-n+2}\ldots w_d}.
$$
\end{lemma}

\begin{proof}
Set $H=c_1(\mathcal O_{\mathbb P}(1))\in A^1(\mathbb P)_{\mathbb Q}$. For an index set $J=\{j_0,\ldots,j_n\}\subset\{0,\ldots,d\}$,
let
$$
L_J=\{y_i=0\mid i\notin J\}\subset\mathbb P
$$
and put $g_J=\gcd(w_{j_0},\ldots,w_{j_n})$. Then
$$
H^n\cdot L_J=\frac{g_J}{w_{j_0}w_{j_1}\ldots w_{j_n}}.
$$
Now, let $S=\mathbb C[y_0,\ldots,y_d]$ with $\deg(y_i)=w_i$, and let $I_Y\subset S$ be the weighted homogeneous prime ideal of the subvariety $Y$. Fix the lexicographic monomial order $y_d\succ y_{d-1}\succ\ldots\succ y_1\succ y_0$, and choose an integral weight vector in the Gr\"obner cone of this initial ideal. Then the Gr\"obner construction from  \cite[Section 15.8]{Eisenbud} gives a flat family $\mathcal Y\subset\mathbb P\times\mathbb A^1$ whose fibre over $1$ is $Y$ and whose fibre over $0$ is
$$
Y_0=\operatorname{Proj}\bigl(S/\operatorname{in}(I_Y)\bigr);
$$
Then $\mathcal Y$ is integral, and every irreducible component of its special fibre $Y_0$ has dimension $n$.

The ideal $\operatorname{in}(I_Y)$ is monomial. Its minimal prime ideals are therefore coordinate prime ideals, and the fundamental cycle of $Y_0$ has the form
\begin{equation}
\label{eq:initial-cycle}
[Y_0]=\sum_{J\in\mathcal J}m_J[L_J],
\qquad
m_J\in\mathbb Z_{>0},
\qquad
|J|=n+1.
\end{equation}
Since $[Y]=[Y_0]$ in $A_n(\mathbb P)$, it follows from  \eqref{eq:initial-cycle} that
$$
H^n\cdot Y=\sum_{J\in\mathcal J}m_J\frac{g_J}{\prod_{j\in J}w_j}.
$$
For every $J\in\mathcal J$, we have $g_J\geqslant1$ and
$$
\prod_{j\in J}w_j\leqslant w_{d-n}w_{d-n+1}\ldots w_d,
$$
because $J$ contains $n+1$ indices and the weights are non-decreasing. Since the sum above is non-empty and all its coefficients are positive, this gives us the first inequality
$$
H^n\cdot Y \geqslant \frac{1}{w_{d-n}w_{d-n+1}\ldots w_d}.
$$
The proof of the second inequality is similar, so we omit it.
\end{proof}

\section{Toric symmetry of projective spaces}
\label{section:toric}

In this section, we discuss equivariant birational geometry of the projective space  $\mathbb{P}^n$, $n\geqslant 2$, with respect to the action of a group that contains a maximal torus in $\mathrm{PGL}_{n+1}(\mathbb{C})$ as a normal subgroup.
After some preparation, we formulate the technical result
to be proved in Section~\ref{section:technical}, and use it to deduce Theorem~A.

The groups we are interested in can be described as follows. Let $\mathbb{T}$ be the~maximal torus in~\mbox{$\mathrm{PGL}_{n+1}(\mathbb{C})$} consisting of the~transformations
$$
(t_1,\ldots,t_n)\colon \big[x_1:\ldots:x_{n}:x_{n+1}\big]\mapsto\big[t_1x_1:\ldots:t_nx_{n}:x_{n+1}\big],
$$
let $\mathbb{W}$ be its~normalizer in $\mathrm{PGL}_{n+1}(\mathbb{C})$, and let $\GG$ be a subgroup in $\mathrm{PGL}_{n+1}(\mathbb{C})$ such that
$$
\mathbb{T}\subseteq \GG\subseteq\mathbb{W}.
$$
Then both $\GG$ and $\mathbb{W}$ permute the $\mathbb{T}$-invariant points:
$$
P_1=[1:0:\ldots:0], P_2=[0:1:0:\ldots:0], \ldots, P_{n+1}=[0:\ldots:0:1].
$$
The $\mathbb{W}$-action gives an~epimorphism $\Upsilon\colon\mathbb{W}\to\mathfrak{S}_{n+1}$. Set  $G=\Upsilon(\GG)$, and let $\upsilon\colon \GG\to G$ be the induced epimorphism. Then we have the~following exact sequences of groups:
$$
\xymatrix{1\ar@{->}[rr] &&  \mathbb{T}\ar@{=}[d]\ar@{->}[rr] && \GG\ar@{->}[rr]^{\upsilon}\ar@{_{(}->}[d] && G\ar@{_{(}->}[d]\ar@{->}[rr] && 1\\
1\ar@{->}[rr] && \mathbb{T}\ar@{->}[rr] && \mathbb{W}\ar@{->}[rr]^{\Upsilon} && \mathfrak{S}_{n+1}\ar@{->}[rr] && 1}
$$
Since $\Upsilon$ and $\upsilon$ have natural sections, we have $\mathbb{W}=\mathbb{T}\rtimes\mathfrak{S}_{n+1}$ and $\GG=\mathbb{T}\rtimes G$,
so we will occasionally consider $G$ and $\mathfrak{S}_{n+1}$ as subgroups in $\GG$ and $\mathbb{W}$, respectively, consisting of projective transformations given by standard permutation matrices. On the other hand, $\mathfrak{S}_{n+1}$ acts on $\mathbb{T}$ by conjugations, which gives a monomorphism
$$
\mathfrak{S}_{n+1}\hookrightarrow\mathrm{Aut}(\mathbb{T})\simeq\mathrm{GL}_{n}(\mathbb{Z}).
$$
Thus, occasionally, we may also consider $G$ and $\mathfrak{S}_{n+1}$ as subgroups in $\mathrm{GL}_{n}(\mathbb{Z})$. We will say that the
subgroup~\mbox{$G\subset \mathrm{GL}_{n}(\mathbb{Z})$}
is irreducible, if $\mathbb{Z}^{n}$ is an irreducible $G$-module.

\begin{proposition}[{\cite[Proposition~3.7]{CheltsovDuboulozKishimoto2023}}]
\label{proposition:solid}
The following two conditions are equivalent:
\begin{itemize}
\item the image of the group $G$ in $\mathrm{GL}_{n}(\mathbb{Z})$ is an irreducible subgroup;
\item $\mathbb{P}^n$ is not $\GG$-birational to a $\GG$-Mori fiber space with a positive-dimensional base.
\end{itemize}
\end{proposition}

\begin{remark}
In the notation of \cite{CheltsovDuboulozKishimoto2023}, Proposition~\ref{proposition:solid} says that $\mathbb{P}^n$ is $\GG$-birationally solid if and only if the image of the group $G$ in $\mathrm{GL}_{n}(\mathbb{Z})$ is irreducible.
\end{remark}

In particular, it follows from Proposition~\ref{proposition:solid}
that if $\mathbb{P}^n$ is not $\GG$-birational to a $\GG$-Mori fiber space with a positive-dimensional base, then $G$ is a transitive subgroup of $\mathfrak{S}_{n+1}$ in the sense of Definition~\ref{definition:primitive-transitive-permutation-group}.
In fact, we can say more.

\begin{lemma}
\label{lemma:toric-primitive}
If the group $G$ is not a primitive subgroup of $\mathfrak{S}_{n+1}$
in the sense of Definition~\ref{definition:primitive-transitive-permutation-group},
then~$\mathbb{P}^n$ is $\GG$-birational to a $\GG$-Mori fiber space with a positive-dimensional base.
\end{lemma}

\begin{proof}
The group $\mathfrak{S}_{n+1}$ and its subgroup $G$ act by permutations on the homogeneous coordinates~\mbox{$x_1,\ldots,x_{n+1}$}.
The $\mathfrak{S}_{n+1}$-module $\ZZ^n$ can be identified with the quotient of $\ZZ^{n+1}$ with the basis
vectors~\mbox{$e_1,\ldots,e_{n+1}$} which are in bijection with these coordinates, by the submodule
generated by the vector~\mbox{$e_1+\ldots+e_{n+1}$}.

First, suppose that $G$ is not transitive in the sense of Definition~\ref{definition:primitive-transitive-permutation-group}.
Then for some integer~\mbox{$1\leqslant r\leqslant\frac{n+1}{2}$} the group $G$ permutes the vectors
$e_1,\ldots,e_r$, so that their images in $\ZZ^n$ span a non-trivial submodule. Therefore, $\mathbb{P}^n$ is $\GG$-birational to a $\GG$-Mori fiber space with a positive-dimensional base by Proposition~\ref{proposition:solid}.

Now suppose that $G$ is transitive but not primitive.
Then there is a partition
$$
\{e_1,\ldots,e_{n+1}\}=\Sigma^1\sqcup\ldots\sqcup \Sigma^k
$$
such that $1<|\Sigma^i|<n+1$, and the sets $\Sigma^i$ are permuted by $G$.
We may assume that $n+1=mk$, and
$$
\Sigma^1=\{e_1,\ldots,e_m\}, \ldots,
\Sigma^k=\{e_{(k-1)m+1},\ldots,e_{n+1}\}.
$$
Thus, $G$ preserves the submodule of $\ZZ^{n+1}$ spanned by the vectors
$$
v_1=e_1+\ldots+e_m,  \ldots,  v_k=e_{(k-1)m+1}+\ldots+e_{n+1},
$$
whose images span a non-trivial submodule of $\ZZ^n$. Therefore, $\mathbb{P}^n$ is again $\GG$-birational to a $\GG$-Mori fiber space with a positive-dimensional base by Proposition~\ref{proposition:solid}.
\end{proof}

Another useful result concerning $\GG$-birational maps of toric varieties is as follows.

\begin{proposition}[{\cite[Lemma~3.1]{CheltsovDuboulozKishimoto2023}}]
\label{proposition:G-non-rigid-via-modules}
Let $V_1$ and $V_2$ be $n$-dimensional toric varieties with the action of the $n$-dimensional torus
$\mathbb{T}$. Let $G$ be a finite group acting on $V_i$, $i=1,2$, such that $G$
normalizes the torus~\mbox{$\mathbb{T}\subset\mathrm{Aut}(V_i)$} and acts faithfully on $\mathbb{T}$.
Denote by $G_i$ the image of $G$ in
$$
\mathrm{Aut}(\mathbb{T})\simeq\GL_n(\ZZ)
$$
corresponding to these
two actions on $\mathbb{T}$, and set
$\GG_i=\mathbb{T}\rtimes G_i$.
The following two conditions are equivalent:
\begin{itemize}
\item the groups $G_1$ and $G_2$ are conjugate in $\GL_n(\ZZ)$;
\item the groups $\GG_i$ are isomorphic to each other, and there exists a
birational map $V_1\dasharrow V_2$ which is equivariant with respect to $\GG_1\simeq \GG_2$.
\end{itemize}
\end{proposition}

In Definition~\ref{definition:rigidity}, we defined equivariant birational rigidity for a Fano variety with terminal singularities acted on by a finite group. The same definition applies for an algebraic group action, so we can define $\GG$-birational rigidity for $\mathbb{P}^n$ exactly as in Definition~\ref{definition:rigidity}.
Our next goal is to prove the following purely toric theorem.

\begin{theorem}
\label{theorem:toric}
Suppose that $n\geqslant 3$. Then $\mathbb{P}^n$ is not $\GG$-birationally rigid.
\end{theorem}

To prove Theorem~\ref{theorem:toric}, we have to prove that
\begin{enumerate}
\item either $\mathbb{P}^n$ is $\GG$-birational to a $\GG$-Mori fiber space with a positive-dimensional base,
\item or the projective space $\mathbb{P}^n$ is $\GG$-birational to a Fano variety $X$ with terminal singularities such that $\mathrm{rk}\,\mathrm{Cl}(X)^{\GG}=1$, but $X$ is not $\GG$-equivariantly isomorphic to $\mathbb{P}^n$.
\end{enumerate}
This immediately follows from the following slightly stronger technical result.

\begin{theorem}
\label{theorem:toric-2}
Suppose that $n\geqslant 3$. Then either $\mathbb{P}^n$ is $\GG$-birational to a $\GG$-Mori fiber space with positive-dimensional base, or $\mathbb{P}^n$ is $\GG$-birational to a Fano variety $X$ with terminal singularities such that $\mathrm{rk}\,\mathrm{Cl}(X)^{\GG}=1$ and $X\not\simeq \mathbb{P}^n$.
\end{theorem}

We will prove Theorem~\ref{theorem:toric-2} later in Section~\ref{section:technical}.

\begin{remark}\label{remark:dim-3}
In dimension $3$, both Theorems~\ref{theorem:toric} and \ref{theorem:toric-2} have been proved in \cite{CheltsovShramov2019,CheltsovDuboulozKishimoto2023}.
\end{remark}

Let us repeat that Theorem~\ref{theorem:toric-2} implies Theorem~\ref{theorem:toric}. In fact, using Theorem~\ref{theorem:toric-2}, we can also deduce Theorem~A.

\begin{proof}[Proof of Theorem A]
Suppose that $G$ is not primitive.
By Lemma~\ref{lemma:intransitive}, we may assume that $G$ is transitive.
Next, by Lemma~\ref{lemma:intransitive-2}, we may assume that
$G$ has an orbit of length $n+1$.
Choosing appropriate coordinates on $\mathbb{P}^n$, we may assume that this orbit consists of the points
$$
P_1=[1:0:\ldots:0], P_2=[0:1:0:\ldots:0], \ldots, P_{n+1}=[0:\ldots:0:1].
$$
Let $\mathbb{T}$ be the~maximal torus in $\mathrm{PGL}_{n+1}(\mathbb{C})$ consisting of automorphisms
$$
\big[x_1:\ldots:x_{n}:x_{n+1}\big]\mapsto\big[t_1x_1:\ldots:t_nx_{n}:x_{n+1}\big],
$$
and let $\mathbb{G}$ be the subgroup in $\mathrm{PGL}_{n+1}(\mathbb{C})$ generated by $G$ and $\mathbb{T}$. Since $n\geqslant 3$, it follows from Theorem~\ref{theorem:toric-2}
that there exists a~$\mathbb{G}$-equivariant birational map
$\chi\colon\mathbb{P}^n\dasharrow X$ such that $X\not\simeq \mathbb{P}^n$ is a toric
$\mathbb{G}$-Mori fiber space over some base $Z$ (which may be a point or have positive dimension); in particular,
$X$ is $\mathbb{GQ}$-factorial, i.e.
$$
\mathrm{rk}\,\mathrm{Cl}(X)^{\mathbb{G}}=\mathrm{rk}\,\mathrm{Pic}(X)^{\mathbb{G}},
$$
and one has $\mathrm{rk}\,\mathrm{Pic}(X/Z)^{\mathbb{G}}=1$.
Note that $\chi$ is also $G$-equivariant, and the torus $\mathbb{T}$ acts trivially on the lattices~\mbox{$\mathrm{Cl}(X)$} and~\mbox{$\mathrm{Pic}(X)$}. Therefore, one has
$$
\mathrm{rk}\,\mathrm{Cl}(X)^{G}=\mathrm{rk}\,\mathrm{Cl}(X)^{\mathbb{G}}=\mathrm{rk}\,\mathrm{Pic}(X)^{\mathbb{G}}=\mathrm{rk}\,\mathrm{Pic}(X)^{G},
$$
so that $X$ is $G\mathbb{Q}$-factorial. Similarly, we see that
$$
\mathrm{rk}\,\mathrm{Pic}(X/Z)^{G}=\mathrm{rk}\,\mathrm{Pic}(X/Z)^{\mathbb{G}}=1.
$$
Thus, $X$ is a $G$-Mori fiber space over $Z$.
Since $X\not\simeq \mathbb{P}^n$, this means that $\mathbb{P}^n$ is not $G$-birationally rigid.
\end{proof}

\section{Proof of the main technical result}
\label{section:technical}

In this section we prove Theorem~\ref{theorem:toric-2}.
Let us use the notation of Section~\ref{section:toric}.
We may assume that the image of the~group $G$ in $\mathrm{GL}_{n}(\mathbb{Z})$ is an irreducible subgroup, since otherwise we are done by Proposition~\ref{proposition:solid}.
Moreover, by Lemma~\ref{lemma:toric-primitive} we may assume that
$G$ is a primitive subgroup of~$\mathfrak{S}_{n+1}$ in the sense of Definition~\ref{definition:primitive-transitive-permutation-group}.
To prove Theorem~\ref{theorem:toric-2}, we are going to construct a $\GG$-birational map
$\mathbb{P}^n\dasharrow X$
such that the following conditions hold:
\begin{itemize}
\item $X$ has terminal singularities,
\item $X$ is a toric Fano variety,
\item $\mathrm{rk}\,\mathrm{Cl}(X)^{\GG}=1$,
\item $X\not\simeq \mathbb{P}^n$.
\end{itemize}

\begin{remark}
Observe that $\mathbb{T}$ acts trivially on $\mathrm{Cl}(X)$, so that
$$
\mathrm{rk}\,\mathrm{Cl}(X)^{\GG}=\mathrm{rk}\,\mathrm{Cl}(X)^{G}.
$$
Thus, in fact we have to check
whether $\mathrm{rk}\,\mathrm{Cl}(X)^{G}=1$ or not.
\end{remark}

The plan of our construction is as follows. First, we blow up $\PP^n$ in a certain configuration of linear subspaces and run an $\mathbb{W}$-equivariant Minimal Model Program to construct a $\mathbb{W}$-equivariant birational map $\PP^n\dasharrow X$ to a particular terminal Fano variety $X$ which is not isomorphic to $\PP^n$. It turns out that if the group $G$ is large enough, then $\mathrm{rk}\,\mathrm{Cl}(X)^{\GG}=1$; this includes the cases when~$G$ is isomorphic to $\mathfrak{S}_{n+1}$ or $\mathfrak{A}_{n+1}$ and, more generally, when $G$ acts transitively on
the non-ordered pairs of indices from $\{1,\ldots,n+1\}$
(cf. Example~\ref{example:2-transitive}). This step of the construction has an alternative more explicit description, see Remark~\ref{remark:do-nor-understand-M}.
In a more general case, when $\mathrm{rk}\,\mathrm{Cl}(X)^{\GG}\neq 1$,
we start with the latter variety $X$ and proceed with a $\GG\mathbb{Q}$-factorialization followed by a further $\GG$-equivariant Minimal Model Program to obtain another $\GG\mathbb{Q}$-Fano
variety $X'$ which is $\GG$-birational to $\PP^n$. If the group $G$ is not
too small, then it turns out that the variety $X'$ is not isomorphic to $\PP^n$ (see Lemma~\ref{lemma:toric-big-orbits}). This works unless the group~$G$ is either a cyclic group $\mumu_{n+1}$ or a dihedral group $\mathfrak{D}_{n+1}$, with certain particular action on~$\PP^n$. In these two cases we construct two special $\GG\mathbb{Q}$-Fano varieties $V$ and $U$, and check that there exists a $\GG$-equivariant birational map from $\PP^n$ to $V$ or $U$, respectively (see Corollary~\ref{corollary:toric-final} and
Lemma~\ref{lemma:toric-final-final}).

\medskip
To start with, let $\mathcal{M}$ be the linear subsystem in $|\mathcal{O}_{\mathbb{P}^n}(2n)|$ that consists of all hypersurfaces
$$
x_1x_2\ldots x_{n+1}f(x_1,\ldots,x_{n+1})+\sum_{i=1}^{n+1}\lambda_i\frac{\big(x_1x_2\ldots x_{n+1}\big)^2}{x_i^2}=0,
$$
where $f$ is any polynomial of degree $n-1$, and $\lambda_1,\ldots,\lambda_{n+1}\in\mathbb{C}$. Set $N=\mathrm{dim}(\mathcal{M})$. Then
$$
N=n+\binom{2n-1}{n},
$$
the linear system $\mathcal{M}$ is mobile, its base locus is the union of all $(n-2)$-dimensional $\mathbb{T}$-invariant subvarieties, and a general member of $\mathcal{M}$ is singular along the base locus.

\begin{remark}\label{remark:do-nor-understand-M}
Since $\mathcal{M}$ is $\mathbb{W}$-invariant, it gives an~$\mathbb{W}$-equivariant birational map
$\psi\colon \mathbb{P}^n\dasharrow \mathscr{X}$ such that $\mathscr{X}$ is a toric $n$-fold in $\mathbb{P}^N$.
One can show that $\mathscr{X}$ is a Fano variety with terminal singularities, $\mathrm{rk}\,\mathrm{Cl}(X)^{\mathbb{W}}=1$, and $\mathscr{X}\not\simeq\mathbb{P}^n$; see Appendix~\ref{app:toric-fano}. So we can set $X=\mathscr{X}$, at least when~\mbox{$G\simeq \mathfrak{S}_{n+1}$}. 
However, we decided to provide an alternative construction of the toric Fano variety $X$, which does not rely on results of Appendix~\ref{app:toric-fano}. 
\end{remark}

The toric variety $\mathcal{X}$ introduced in Remark~\ref{remark:do-nor-understand-M} have been previously described for $n=3$ and $4$.

\begin{example}[{cf. \cite{Bayle,CheltsovDuboulozKishimoto2023,CheltsovSarikyan2023,CheltsovShramov2019,Martello,Sarikyan}}]
\label{example:X24}
Suppose that $n=3$.
Then $\mathcal{M}$ consists of all sextic surfaces that are singular along six $\mathbb{T}$-invariant lines forming a toric tetrahedron in $\mathbb{P}^3$, and
$$
\mathscr{X}\simeq\mathbb{P}^1\times\mathbb{P}^1\times\mathbb{P}^1/\mumu_2,
$$
where $\mumu_2$ acts on $\mathbb{P}^1\times\mathbb{P}^1\times\mathbb{P}^1$ as
$$
\big([x_1:y_1],[x_2:y_2],[x_3:y_3]\big)\mapsto\big([x_1:-y_1],[x_2:-y_2],[x_3:-y_3]\big).
$$
Therefore, the threefold $\mathscr{X}$ has $8$ cyclic quotient singularities of type $\frac{1}{2}(1,1,1)$. Moreover, we have the~following $\mathbb{W}$-equivariant commutative diagram:
\begin{equation*}
%\label{equation:toric}
\xymatrix{
&\widetilde{X}\ar@{->}[ld]_{\pi_1}\ar@{-->}[rr]^{\rho}&& \widetilde{X}^\prime\ar@{->}[dr]^{\varpi}&\\%
X_1\ar@{->}[d]_{\pi_0}&&&&\widehat{X}\ar@{->}[d]^{\phi}\ar@{->}[lllld]_{\varphi}\\%
\mathbb{P}^3\ar@{-->}[rrrr]^{\psi}&&&& \mathscr{X}}
\end{equation*}
where
\begin{itemize}
\item $\pi_0$ blows up four $\mathbb{T}$-invariant points;
\item $\pi_1$ blows up the strict transforms of the six $\mathbb{T}$-invariant lines;
\item $\rho$ flops strict transforms of $\mathbb{T}$-invariant curves contained in $\pi_0$-exceptional surfaces;
\item $\varpi$ contracts the strict transforms of $\pi_0$-exceptional surfaces;
\item $\phi$ contracts the strict transforms of four $\mathbb{T}$-invariant planes in $\mathbb{P}^3$;
\item $\varphi$ symbolically blows up the ideal sheaf of the union of six $\mathbb{T}$-invariant lines~\mbox{\cite[\S6.1]{ProkhorovReid}}.
\end{itemize}
Moreover, if $G\simeq\mathfrak{S}_4$
or $G\simeq\mathfrak{A}_4$, then one has
$\mathrm{rk}\,\mathrm{Cl}(\mathcal{X})^{\GG}=\mathrm{rk}\,\mathrm{Cl}(\mathcal{X})^{G}=1$.
\end{example}

\begin{example}
\label{example:Petracci}
For $n=4$, Andrea Petracci informed us that the toric Fano 4-fold $\mathscr{X}$ has terminal (non-isolated) singularities such that  $\mathrm{rk}\,\mathrm{Cl}(\mathscr{X})=6$, $\mathrm{rk}\,\mathrm{Pic}(\mathscr{X})=1$, and~\mbox{$(-K_\mathscr{X})^4=\frac{625}{6}$}. Moreover, if $G\simeq\mathfrak{S}_5$
or $G\simeq\mathfrak{A}_5$, then one has
$\mathrm{rk}\,\mathrm{Cl}(\mathcal{X})^{\GG}=\mathrm{rk}\,\mathrm{Cl}(\mathcal{X})^{G}=1$.
The singular locus of $\mathscr{X}$ consists of $5$ isolated quotient singularities of type $\frac{1}{3}(1,1,1,1)$, and $10$ curves such that~$\mathscr{X}$ has a quotient singularity of type $\frac{1}{2}(1,1,1)$ at their general points. In particular, $\mathscr{X}$ is not isomorphic to~$\mathbb{P}^4$. In this case, the linear system $\mathcal{M}$ has been studied in \cite{Godeaux,Stagnaro}.
\end{example}

Now we construct a $\mathbb{W}$-birational map $\mathbb{P}^n\dasharrow X$ using equivariant toric Minimal Model Program. Keeping in mind Remark~\ref{remark:dim-3}, we will further assume that~\mbox{$n\geqslant 4$}.
To start with, consider the following sequence of $n-1\geqslant 3$ toric $\mathbb{W}$-equivariant blowups
$$
\xymatrix{
\widetilde{X}\ar@{->}[rr]^{\pi_{n-2}} && X_{n-2}\ar@{->}[rr]^{\pi_{n-3}}&& \ldots\ar@{->}[rr]^{\pi_2} && X_{2}\ar@{->}[rr]^{\pi_1} && X_{1}\ar@{->}[rr]^{\pi_0} && \mathbb{P}^n,}
$$
where
\begin{itemize}
\item $\pi_0$ is the blow up of all $\mathbb{T}$-invariant points in $\PP^n$,
\item $\pi_1$ is the blow up of the strict transforms of all $\mathbb{T}$-invariant lines in $\mathbb{P}^n$,
\item $\pi_2$ is the blow up of the strict transforms of all $\mathbb{T}$-invariant planes in $\mathbb{P}^n$, etc.
\end{itemize}
Let $\pi\colon\widetilde{X}\to \mathbb{P}^n$ be the composition $\pi_0\circ\pi_1\circ\ldots\circ\pi_{n-2}$.
Then $\widetilde{X}$ is toric, and $\pi$ is $\mathbb{W}$-equivariant.

\begin{remark}
\label{remark:permutohedron}
The variety $\widetilde{X}$ is known as permutohedron \cite{Huh,Karp}.
\end{remark}

For every $k\in\{0,\ldots,n-2\}$, let $\widetilde{E}_k$ be the $\mathbb{W}$-irreducible $\pi$-exceptional divisor whose irreducible components are mapped to $\mathbb{T}$-invariant linear subspaces of $\mathbb{P}^n$ that have dimension~$k$. Similarly, we let $\widetilde{E}_{n-1}$ be the $\mathbb{W}$-irreducible divisor in $\widetilde{X}$ whose irreducible components are strict transforms of $\mathbb{T}$-invariant hyperplanes in $\mathbb{P}^n$. Denote by
$\Xi\subset\PP^n$ the union of the latter hyperplanes, i.e. the divisor in $\PP^n$
given by equation $x_1x_2\ldots x_{n+1}=0$.
Then $\pi$ induces an isomorphism
$$
\widetilde{X}\setminus\mathrm{Supp}\Bigg(\sum_{i=0}^{n-1}\widetilde{E}_i\Bigg)\simeq\mathbb{P}^{n}\setminus\Xi
\simeq\big(\mathbb{C}^*\big)^n,
$$
so the union $\widetilde{E}_0\cup \widetilde{E}_1\cup \ldots \cup\widetilde{E}_{n-1}$ forms the toric boundary in $\widetilde{X}$. Note also that
\begin{equation}
\label{equation:T}
\widetilde{E}_{n-1}\sim\pi^*\big(-K_{\mathbb{P}^n}\big)-\sum_{i=0}^{n-2}(n-i)\widetilde{E}_i.
\end{equation}

\begin{remark}
\label{remark:permutohedron-Cremona} Let $\tau_{\mathbb{P}^n}$ be the standard Cremona birational involution in $\mathrm{Bir}^{\mathbb{W}}(\mathbb{P}^n)$. Then we have the following $\mathbb{W}$-equivariant commutative diagram:
$$
\xymatrix{
\widetilde{X}\ar[d]_{\pi}\ar@{->}[rr]^{\tau_{\widetilde{X}}}&&\widetilde{X}\ar[d]^{\pi} \\
\mathbb{P}^n\ar@{-->}[rr]^{\tau_{\mathbb{P}^n}} && \mathbb{P}^n}
$$
where $\tau_{\widetilde{X}}$ is a biregular involution that swaps $\widetilde{E}_{k}$ and $\widetilde{E}_{n-1-k}$ for every $k\in\{0,\ldots,n-1\}$.
\end{remark}

Let $\widetilde{\mathcal{M}}$ be the strict transform on $\widetilde{X}$ of the linear system $\mathcal{M}$. Then
\begin{equation}
\label{equation:M}
\widetilde{\mathcal{M}}\sim \pi^*(\mathcal{M})-\sum_{i=0}^{n-2}(n-i)\widetilde{E}_i.
\end{equation}
This gives
$$
K_{\widetilde{X}}+\frac{n+1}{2n}\widetilde{\mathcal{M}}\sim_{\mathbb{Q}} \pi^*\Big(K_{\mathbb{P}^n}+\frac{n+1}{2n}\mathcal{M}\Big)+\sum_{i=0}^{n-2}a_i\widetilde{E}_i\sim_{\mathbb{Q}}\sum_{i=0}^{n-2}a_i\widetilde{E}_i,
$$
where
\begin{equation}
\label{equation:discrepancy}
\left\{\aligned
&a_0=\frac{n-3}2, \\
&a_1=\frac{n^2-4n+1}{2n},\\
&\vdots\\
&a_i=\frac{n^2-(i+3)n+i}{2n},\\
&\vdots\\
&a_{n-3}=\frac{n-3}{2n},\\
&a_{n-2}=-\frac1n.
\endaligned
\right.
\end{equation}
Observe that $a_i>0$ for all $0\leqslant i\leqslant n-3$, because we assume that $n\geqslant 4$. On the other hand, we have~\mbox{$a_{n-2}=-\frac{1}{n}<0$}. In particular, we see that the singularities of the log pair $(\mathbb{P}^n,\frac{n+1}{2n}\mathcal{M})$ are not canonical.

\begin{lemma}
\label{lemma:M-is-free}
The singularities of the log pair $(\widetilde{X},\frac{n+1}{2n}\widetilde{\mathcal{M}})$ are terminal.
\end{lemma}

\begin{proof}
Since $\frac{n+1}{2n}<1$, it is enough to show that $\mathrm{mult}_P(\widetilde{\mathcal{M}})\leqslant 1$ for every point $P$ contained in the~base locus of the linear system $\widetilde{\mathcal{M}}$. On the other hand, since $\widetilde{\mathcal{M}}$ is $\mathbb{W}$-invariant and mobile, its base locus is contained in $\widetilde{E}_{0}\cup\ldots\cup\widetilde{E}_{n-2}$.  Moreover, it follows from the~local computations that the~base locus of the~linear system $\widetilde{\mathcal{M}}$ is contained in $\widetilde{E}_{0}\cup\ldots\cup\widetilde{E}_{n-3}$. So, let us fix a point~\mbox{$P\in \widetilde{E}_{0}\cup\ldots\cup\widetilde{E}_{n-3}$} and show that $\mathrm{mult}_P(\widetilde{\mathcal{M}})\leqslant 1$.

Suppose that $P\in\widetilde{E}_0$. Let $\widetilde{E}$ be the irreducible component of the divisor $\widetilde{E}_0$ that contains $P$, and let $E$ be its strict transform on $X_1$. Then $E\simeq\mathbb{P}^{n-1}$, the variety $\widetilde{E}$ is a permutohedron, and the~induced morphism $\widetilde{E}\to E$ is an analogue of the constructed birational morphism $\pi\colon\widetilde{X}\to \mathbb{P}^{n}$. Moreover, it follows from \eqref{equation:M} and \eqref{equation:T} that the restriction $\widetilde{\mathcal{M}}\vert_{\widetilde{E}}$ is the disjoint union of the~strict transforms on $\widetilde{E}$ of torus invariant hyperplanes in $E$. This gives $\mathrm{mult}_P(\widetilde{\mathcal{M}})\leqslant 1$.

Now, we suppose that $P\in\widetilde{E}_1\setminus\widetilde{E}_0$. Let $\widetilde{E}$ be the irreducible component of $\widetilde{E}_1$ containing $P$, let $E$ be its strict transform on $X_2$, let $Z$ be its image on $X_1$, let $\widetilde{F}$ be the fiber of the induced morphism $\widetilde{E}\to Z$ such that $P\in\widetilde{F}$, and let $F$ be its strict transform on $X_2$. Then $F\simeq\mathbb{P}^{n-2}$, and~$\widetilde{F}$ is a permutohedron that is obtained from $F$ via the induced birational morphism $\widetilde{F}\to F$. Now, using \eqref{equation:M} and \eqref{equation:T}, we conclude that the restriction $\widetilde{\mathcal{M}}\vert_{\widetilde{F}}$ is the disjoint union of the strict transforms on $\widetilde{F}$ of all torus invariant hyperplanes in $F$. As above, we see that $\mathrm{mult}_P(\widetilde{\mathcal{M}})\leqslant 1$.

Hence, if $P\in\widetilde{E}_0\cup \widetilde{E}_1$, we have $\mathrm{mult}_P(\widetilde{\mathcal{M}})\leqslant 1$. Thus, we may assume that  $n\geqslant 5$. Now, arguing exactly as in the previous case, we see that $\mathrm{mult}_P(\widetilde{\mathcal{M}})\leqslant 1$ if $P\in\widetilde{E}_2\setminus(\widetilde{E}_0\cup \widetilde{E}_1)$. Continuing in this way, we see that $\mathrm{mult}_P(\widetilde{\mathcal{M}})\leqslant 1$ for every $P\in \widetilde{E}_{0}\cup\ldots\cup\widetilde{E}_{n-3}$.
\end{proof}

Applying the relative $\mathbb{W}$-equivariant Minimal Model Program to the log pair $(\widetilde{X},\frac{n+1}{2n}\widetilde{\mathcal{M}})$ over~$\mathbb{P}^n$, we obtain a toric $\mathbb{W}$-equivariant commutative diagram
$$
\xymatrix{
& \widetilde{X}\ar[ld]_{\pi}\ar@{-->}[rd]^{\eta} \\
\mathbb{P}^n&&\widehat{X}\ar@{->}[ll]_{\varphi}}
$$
such that $\eta$ is a composition of divisorial contractions and log flips, $\varphi$ is a birational morphism, and one has
$$
\mathrm{rk}\,\mathrm{Cl}\big(\widehat{X}\big)^{\mathbb{W}}=\mathrm{rk}\,\mathrm{Pic}\big(\widehat{X}\big)^{\mathbb{W}}.
$$
Furthermore, by Lemma~\ref{lemma:M-is-free}
the singularities of the log pair $(\widetilde{X},\frac{n+1}{2n}\widetilde{\mathcal{M}})$
are terminal, which implies that the~pair $(\widehat{X},\frac{n+1}{2n}\widehat{\mathcal{M}})$ also has terminal singularities, and $K_{\widehat{X}}+\frac{n+1}{2n}\widehat{\mathcal{M}}$ is $\varphi$-nef, where~$\widehat{\mathcal{M}}$ is the~strict transform on $\widehat{X}$ of the~linear system~$\mathcal{M}$. Moreover, since $K_{\widehat{X}}+\frac{n+1}{2n}\widehat{\mathcal{M}}$ is $\varphi$-nef, the map~$\eta$ contracts every $\mathbb{W}$-irreducible $\pi$-exceptional divisor $\widetilde{E}_k$ with $a_k>0$. Therefore, using~\eqref{equation:discrepancy},
we see that the constructed $\mathbb{W}$-birational map $\eta\colon\widetilde{X}\dasharrow\widehat{X}$ contracts all $\mathbb{W}$-irreducible $\pi$-exceptional divisors except for $\widetilde{E}_{n-2}$. Hence, we see that $\mathrm{rk}\,\mathrm{Cl}(\widehat{X})^{\mathbb{W}}=2$ and
$$
K_{\widehat{X}}+\frac{n+1}{2n}\widehat{\mathcal{M}}\sim_{\mathbb{Q}} \pi^*\Big(K_{\mathbb{P}^n}+\frac{n+1}{2n}\mathcal{M}\Big)-\frac{1}{n}\widehat{E}_{n-2}\sim_{\mathbb{Q}}-\frac{1}{n}\widehat{E}_{n-2},
$$
where $\widehat{E}_{n-2}$ is the strict transform on $\widehat{X}$ of the divisor $\widetilde{E}_{n-2}$.

Now, applying the~absolute $\mathbb{W}$-equivariant Minimal Model Program to the log pair $(\widehat{X},\frac{n+1}{2n}\widehat{\mathcal{M}})$, and using Proposition~\ref{proposition:solid}, we obtain the following $\mathbb{W}$-equivariant Sarkisov link:
$$
\xymatrix{
\widehat{X}\ar@{->}[d]_{\varphi}\ar@{-->}[rr]^{\rho}&&\overline{X}\ar[d]^{\phi}\\
\mathbb{P}^n\ar@{-->}[rr]^{\chi}&& X}
$$
Here $\rho$ is a small $\mathbb{W}$-birational map, $\phi$ is an $\mathbb{W}$-equivariant divisorial contraction that contracts the strict transform of the divisor $\widetilde{E}_{n-1}$, and $X$ is a toric Fano variety such that
$$
\mathrm{rk}\mathrm{Cl}(X)^{\mathbb{W}}=\mathrm{rk}\mathrm{Cl}(X)^{\mathfrak{S}_{n+1}}=1.
$$
By construction, the toric boundary of the toric variety $X$ is the strict transform of the divisor~$\widetilde{E}_{n-2}$, and its irreducible components generate the class group $\mathrm{Cl}(X)$. Thus, since $\phi$ contracts the~strict transform of the divisor $\widetilde{E}_{n-1}$, we see that
$$
\mathrm{rk}\,\mathrm{Cl}(X)=\mathrm{rk}\,\mathrm{Cl}(\widehat{X})-n-1=\binom{n+1}{2}-n.
$$
In particular, one has $X\not\simeq\mathbb{P}^n$.

\begin{remark}
As pointed out in Appendix~\ref{app:toric-fano}, $X$ is isomorphic to the variety $\mathscr{X}$ introduced
in Remark~\ref{remark:do-nor-understand-M}, but we decided not to use this fact.
\end{remark}

The above construction proves Theorem~\ref{theorem:toric-2} in the case when
\begin{equation}\label{eq:G-rkCl-rkCl}
\mathrm{rk}\mathrm{Cl}(X)^{\GG}=\mathrm{rk}\mathrm{Cl}(X)^{G}=1.
\end{equation}
This often happens when $G$ is large enough,
although there are cases when the group does not satisfy these properties.

\begin{example}\label{example:2-transitive}
If the group $G$ acts transitively on the set of all non-ordered pairs~\mbox{$\{P_i,P_j\}$}  consisting of distinct points in
$$
\big\{P_1,\ldots,P_{n+1}\big\}\subset \mathbb{P}^{n},
$$
then it acts transitively on irreducible components of $\widetilde{E}_{n-2}$, and so  condition~\eqref{eq:G-rkCl-rkCl} holds.
For instance, this is always the case when $G\simeq\mathfrak{S}_{n+1}$ or $G\simeq\mathfrak{A}_{n+1}$.
Furthermore, if $n=4$, then $G$ is isomorphic to one of the following groups:
\begin{center}
$\mumu_5$, $\mathfrak{D}_{5}$, $\mumu_5\rtimes\mumu_4$, $\mathfrak{A}_5$, $\mathfrak{S}_5$,
\end{center}
because $G$ is a primitive subgroup in $\mathfrak{S}_5$ in the sense of Definition~\ref{definition:primitive-transitive-permutation-group}.
If $G$ is isomorphic to $\mumu_5\rtimes\mumu_4$, $\mathfrak{A}_5$, or $\mathfrak{S}_5$, then condition~\eqref{eq:G-rkCl-rkCl} holds.
However, if $G\simeq\mumu_5$ or~$G\simeq\mathfrak{D}_{5}$, then
explicit computations show that
$\mathrm{rk}\mathrm{Cl}(X)^{\GG}=\mathrm{rk}\,\mathrm{Cl}(X)^{G}=2$.
\end{example}

Thus, to complete the proof of Theorem~\ref{theorem:toric-2}, we may assume that
condition~\eqref{eq:G-rkCl-rkCl} fails, i.e. one has
$$
\mathrm{rk}\mathrm{Cl}(X)^{\GG}=\mathrm{rk}\mathrm{Cl}(X)^{G}\ne 1.
$$
Let $\gamma\colon X^{\sharp}\to X$ be a $\GG\mathbb{Q}$-factorialization, i.e. $\gamma$ is a small $\GG$-equivariant birational morphism (possibly an isomorphism) such that $U$ has terminal $\GG\mathbb{Q}$-factorial singularities. One has
$$
\mathrm{rk}\mathrm{Cl}(X^{\sharp})^{G}=\mathrm{rk}\mathrm{Cl}(X^{\sharp})^{\GG}=\mathrm{rk}\mathrm{Pic}(X^{\sharp})^{\GG}=\mathrm{rk}\mathrm{Pic}(X^{\sharp})^{G}.
$$
Then, applying $\GG$-equivariant Minimal Model Program, and using Proposition~\ref{proposition:solid} again, we obtain a $\GG$-equivariant birational map $\theta\colon X^{\sharp}\dasharrow X^\prime$ such that $\theta$ is a composition of divisorial contractions and flips, and $X^\prime$ is a toric Fano variety with terminal singularities such that
$$
\mathrm{rk}\,\mathrm{Cl}\big(X^\prime\big)^{\GG}=\mathrm{rk}\,\mathrm{Cl}\big(X^\prime\big)^{G}=1.
$$
If $X^\prime\not\simeq\mathbb{P}^{n}$, we are done. However, we do not know whether $X^\prime$
is isomorphic to $\PP^n$ or not. On the other hand, there exists a very simple combinatorial condition on $G$ which guarantees that~\mbox{$X^\prime\not\simeq\mathbb{P}^{n}$}.

\begin{lemma}
\label{lemma:toric-big-orbits}
Let $\Sigma$ be the set of non-ordered pairs $\{P_i,P_j\}$ consisting of distinct points in
$$
\big\{P_1,\ldots,P_{n+1}\big\}\subset \mathbb{P}^{n}.
$$
Then $G$ naturally acts on $\Sigma$. Suppose that $\Sigma$ contains no $G$-orbits of length at most $n+1$. Then~\mbox{$X^\prime\not\simeq\mathbb{P}^{n}$}.
\end{lemma}

\begin{proof}
The constructed $\GG$-equivariant birational map $\theta$ contracts some (but not all) irreducible components of the strict transform of the $\mathbb{T}$-invariant divisor $\widetilde{E}_{n-2}$. The remaining irreducible components are mapped to $\mathbb{T}$-invariant divisors in $X^\prime$, and form its toric boundary. By assumption, the number of these irreducible components is at least $n+2$, so $X^\prime$ cannot be isomorphic to $\mathbb{P}^{n}$, since $\mathbb{P}^{n}$ has exactly $n+1$ torus invariant irreducible divisors.
\end{proof}

If the subgroup $G\subset\mathfrak{S}_{n+1}$ satisfies the combinatorial condition of Lemma~\ref{lemma:toric-big-orbits}, we are done. Hence, to complete the proof of Theorem~\ref{theorem:toric-2}, we may assume that it does not satisfy this condition. Recall that $G$ is a primitive subgroup in $\mathfrak{S}_{n+1}$ in the sense of Definition~\ref{definition:primitive-transitive-permutation-group}.
Thus, it follows from Corollary~\ref{corollary:Sr-orbit-r}
that either $G\simeq\mumu_{n+1}$ or $G\simeq\mathfrak{D}_{n+1}$, and $G$ is conjugate in $\mathfrak{S}_{n+1}$ to one of the following two subgroups:
\begin{enumerate}
\item the cyclic group generated by the cycle $(1\ 2\ \ldots\ n+1)$;
\item the dihedral group generated by the cycle $(1\ 2\ \ldots\ n+1)$ and the permutation
$$
\big(1\ n+1\big)\big(2\ n\big)\ldots \left(\left\lfloor\frac{n+1}{2}\right\rfloor\ \left\lceil\frac{n+1}{2}\right\rceil\right).
$$
\end{enumerate}
Hence, either $\GG$ is conjugate in $\mathbb{W}$ to the subgroup generated by $\mathbb{T}$ and $\sigma$ such that
$$
\sigma\big([x_1:x_2:x_{3}:\ldots:x_{n-1}:x_{n}:x_{n+1}]\big)=[x_{n+1}:x_1:x_2:\ldots:x_{n-2}:x_{n-1}:x_{n}],
$$
or $\GG$ is conjugate in $\mathbb{W}$ to the subgroup generated by $\mathbb{T}$, $\sigma$, and $\iota$ such that
$$
\iota\big([x_1:x_2:x_3:\ldots:x_{n-1}:x_{n}:x_{n+1}]\big)=[x_{n+1}:x_{n}:x_{n-1}:\ldots:x_{3}:x_{2}:x_{1}].
$$
Hence, to complete the proof of Theorem~\ref{theorem:toric-2}, we may assume that
\begin{enumerate}
\item either $\GG\simeq\mathbb{T}\rtimes\mumu_{n+1}$ and $\GG=\langle \mathbb{T},\sigma\rangle$,
\item or $\GG\simeq\mathbb{T}\rtimes\mathfrak{D}_{n+1}$ and $\GG=\langle\mathbb{T},\sigma,\iota\rangle$.
\end{enumerate}

\begin{lemma}
\label{lemma:prime}
The number $n+1$ is prime.
 \end{lemma}

\begin{proof}
If $n+1$ is not prime, then it follows from Remark~\ref{remark:prime} that $G$ is not a primitive subgroup in $\mathfrak{S}_{n+1}$ in the sense of Definition~\ref{definition:primitive-transitive-permutation-group}, which contradicts our assumptions.
\end{proof}

To deal with the case $\GG=\langle\mathbb{T},\sigma\rangle$, let $V=\mathbb{P}^n/\mumu_{n+1}$ such that $\mumu_{n+1}$ acts on $\mathbb{P}^n$ as
\begin{equation}\label{eq:cyclic-group-action}
[x_1:\ldots:x_{n}:x_{n+1}]\mapsto [\zeta x_1:\ldots:\zeta^{n}x_{n}:x_{n+1}],
\end{equation}
where $\zeta$ is a primitive $(n+1)$-th root of unity. Then $V$ is a toric Fano variety.

\begin{lemma}\label{lemma:V-terminal}
The variety $V$ has terminal singularities.
\end{lemma}

\begin{proof}
By construction, $V$ has $n+1$ cyclic quotient singularities of type
$$
\frac{1}{n+1}(1,\ldots,n).
$$
Hence, according to Reid--Tai criterion \cite[Theorem~4.11]{YPG}, the variety $V$ is terminal if and only if for each $r\in\{1,2,\ldots,n\}$, one has
\begin{equation}
\label{eq:Reid-Tai-V}
\sum\limits_{i=1}^n \overline{ri} > n+1,
\end{equation}
where $\overline{a}$ denotes the integer such that $0\leqslant \overline{a}\leqslant n$ and $a\equiv \overline{a}\ \mathrm{mod}\ {n+1}$.
Since $n+1$ is a prime number by Lemma~\ref{lemma:prime}, we have
$$
\{\overline{r},\overline{2r},\ldots,\overline{nr}\}=\{1,2,\ldots,n\}.
$$
Thus, we compute
$$
\sum\limits_{i=1}^n \overline{ri} =\sum\limits_{i=1}^n i=\frac{n(n+1)}{2}>n+1.
$$
Therefore, inequality~\eqref{eq:Reid-Tai-V} holds, so that $V$
has terminal singularities.
\end{proof}

Note that the quotient map $\mathbb{P}^n\to V$ is $\GG$-equivariant, and the~action of the~group $\GG$ on the~toric variety $V$ gives a~group homomorphism $\GG\to\mathrm{Aut}(V)$, whose kernel is isomorphic to $\mumu_{n+1}$. Let $\GG_V$ be its image. Then
$$
\GG_V=\mathbb{T}_V\rtimes G_V,
$$
where $\mathbb{T}_V$ is the maximal torus in $\mathrm{Aut}(V)$, and $G_V$ is the image of the group $G$. We have \mbox{$G_V\simeq G$}. Recall that we have a natural monomorphism $G\hookrightarrow\mathrm{GL}_{n}(\mathbb{Z})$. Similarly, the $\GG_V$-action on $V$ gives a~monomorphism
$$
G_V\hookrightarrow\mathrm{GL}_{n}(\mathbb{Z})\simeq\mathrm{Aut}(\mathbb{T}_V).
$$
By Proposition~\ref{proposition:G-non-rigid-via-modules}, the following are equivalent:
\begin{itemize}
\item there exists a $\GG$-birational map $\mathbb{P}^n\dasharrow V$, where the $\GG$-action on $V$ is given by $\GG_V$;
\item the images of the groups $G$ and $G_V$ in $\mathrm{GL}_{n}(\mathbb{Z})$ are conjugate.
\end{itemize}

\begin{lemma}
\label{lemma:toric-final}
If $\GG=\langle\mathbb{T},\sigma\rangle$, then the images of the groups $G$ and $G_V$ in $\mathrm{GL}_{n}(\mathbb{Z})$ are conjugate.
\end{lemma}

\begin{proof}
Recall that $\mathbb{P}^n$ is a toric variety whose fan in the standard lattice $\mathbb{Z}^{n}$ has $n+1$ one-dimensional rays with primitive vectors
\begin{align*}
v_1&=(1,0,\ldots,0,0),\\
&\ldots\\
v_{n}&=(0,0,\ldots,0,1),\\
v_{n+1}&=(-1,-1,\ldots,-1,-1).
\end{align*}
Then the action of $\sigma$ on $\mathbb{P}^n$ corresponds to the action of the linear operator
$$
A=\left(
\begin{array}{cccccc}
0 & 0 & 0 & \ldots & 0 & -1\\
1 & 0 & 0 & \ldots & 0 & -1\\
0 & 1 & 0 & \ldots & 0 & -1\\
\vdots&\vdots&\vdots& \ddots &\vdots&\vdots\\
0 & 0 & 0 & \ldots & 1 & -1
\end{array}
\right)
$$
on the lattice $\mathbb{Z}^{n}$. In other words, $A$ acts as
$$
v_1\mapsto v_2\mapsto\ldots\mapsto v_n\mapsto v_{n+1}\mapsto v_1.
$$
On the other hand, $V$ is a toric variety whose fan has $n+1$ one-dimensional rays with primitive vectors
$v_1,\ldots,v_{n+1}$ in the lattice $\Lambda$ generated by the standard lattice $\mathbb{Z}^{n}$ and the vector
$$
v=\frac{1}{n+1}(1,2,\ldots,n)=\frac{1}{n+1}(v_1+2v_2+\ldots+nv_n).
$$
The vectors $v, v_2,\ldots, v_{n}$ form a basis in the lattice $\Lambda\simeq\mathbb{Z}^{n}$; one has
$$
v_1=(n+1)v-2v_2-3v_3-\ldots -nv_{n}
$$
so that
$$
A(v)=-nv+v_2+2v_3+\ldots+(n-1)v_n,
$$
and
$$
A(v_n)=v_{n+1}=-(n+1)v+v_2+2v_3+\ldots+(n-1)v_n.
$$
This means that the action of $\sigma$ on $V$ corresponds to the action of the linear operator
$$
B=\left(
\begin{array}{cccccc}
-n & 0 & 0 & \ldots & 0 & -n-1\\
1 & 0 & 0 & \ldots & 0 & 1\\
2 & 1 & 0 & \ldots & 0 & 2\\
3 & 0 & 1 & \ldots & 0 & 3\\
\vdots&\vdots&\vdots& \ddots &\vdots&\vdots\\
n-1 & 0 & 0 & \ldots & 1 & n-1
\end{array}
\right)
$$
on the lattice $\Lambda$. Consider the matrix
$$
C=\left(
\begin{array}{cccccc}
1 & 1 & 1 & \ldots & 1 & 1\\
-1 & 0 & 0 & \ldots & 0 & 0\\
-1 & -1 & 0 & \ldots & 0 & 0\\
\vdots&\vdots&\vdots& \ddots &\vdots&\vdots\\
-1 & -1 & -1 & \ldots & -1 &0
\end{array}
\right)\in\mathrm{GL}_n(\mathbb{Z}).
$$
It is straightforward to check that
$$
CA=\left(
\begin{array}{cccccc}
1 & 1 & 1 & \ldots & 1 & -n\\
0 & 0 & 0 & \ldots & 0 & 1\\
-1 & 0 & 0 & \ldots & 0 & 2\\
-1 & -1 & 0 & \ldots & 0 & 3\\
\vdots&\vdots&\vdots& \ddots &\vdots&\vdots\\
-1 & -1 & -1 & \ldots & 0 & n-1
\end{array}
\right)=BC.
$$
Thus, if $\GG=\langle\mathbb{T},\sigma\rangle$, then the images of the subgroups $G$ and $G_V$ in $\mathrm{GL}_n(\mathbb{Z})$ are conjugate.
\end{proof}

\begin{corollary}
\label{corollary:toric-final}
Suppose that $\GG=\langle \mathbb{T},\sigma\rangle$.
Then there exists a $\GG$-birational map~\mbox{$\mathbb{P}^n\dasharrow V$}, where the $\GG$-action on $V$ is given by $\GG_V$.
\end{corollary}

\begin{proof}
Follows from Lemma~\ref{lemma:toric-final} and
Proposition~\ref{proposition:G-non-rigid-via-modules}.
\end{proof}

Recall that $V$ has terminal singularities, $\mathrm{rk}\,\mathrm{Cl}(V)=1$, and $V\not\simeq \mathbb{P}^n$, because $V$ is singular. Hence in the case when $\GG=\langle \mathbb{T},\sigma\rangle$ we see that Theorem~\ref{theorem:toric-2} follows from
Corollary~\ref{corollary:toric-final}.
Therefore, to complete the proof of Theorem~\ref{theorem:toric-2}, we may assume that  $\GG=\langle\mathbb{T},\sigma,\iota\rangle$, and thus $G\simeq\mathfrak{D}_{n+1}$.

\begin{remark}
\label{remark:toric-final}
Note that  $\mathbb{P}^n$ and the constructed toric Fano variety $V$ are not always $\GG$-birational in the case when $\GG=\langle\mathbb{T},\sigma,\iota\rangle$. Indeed, in the notation of the proof of Lemma~\ref{lemma:toric-final}, the action of the involution $\iota$ on the projective space $\mathbb{P}^n$ corresponds to the~action of the linear operator
$$
S=\left(
\begin{array}{cccccc}
-1 & 0 & 0 & \ldots & 0 & 0 \\
-1 & 0 & 0 & \ldots & 0 & 1 \\
-1 & 0 & 0 & \ldots & 1 & 0 \\
\vdots&\vdots&\vdots& \ddots &\vdots&\vdots\\
-1 & 0 & 1 & \ldots & 0 & 0\\
-1 & 1 & 0 & \ldots & 0 & 0
\end{array}
\right)
$$
on the lattice $\mathbb{Z}^{n}$, and the action of $\iota$ on $V$ corresponds to the action of the linear operator
$$
T=\left(
\begin{array}{cccccc}
-1 & 0 & 0 & \ldots & 0 & 0 \\
1 & 0 & 0 & \ldots & 0 & 1 \\
1 & 0 & 0 & \ldots & 1 & 0 \\
\vdots&\vdots&\vdots& \ddots &\vdots&\vdots\\
1 & 0 & 1 & \ldots & 0 & 0\\
1 & 1 & 0 & \ldots & 0 & 0
\end{array}
\right)
$$
on the lattice $\Lambda$. Using this, one can show that the images of the groups $G$ and $G_V$ in $\mathrm{GL}_{n}(\mathbb{Z})$ are not always conjugate (it is highly likely that they are never conjugate). For instance, if $n=4$, this can be verified by the following Magma code:
\begin{verbatim}
    Z := IntegerRing();
    G := GL(4,Z);
    A := G ! Matrix([[0,0,0,-1],[1,0,0,-1],[0,1,0,-1],[0,0,1,-1]]);
    B := G ! Matrix([[-4,0,0,-5],[1,0,0,1],[2,1,0,2],[3,0,1,3]]);
    S := G ! Matrix([[-1,0,0,0],[-1,0,0,1],[-1,0,1,0],[-1,1,0,0]]);
    T := G ! Matrix([[-1,0,0,0],[1,0,0,1],[1,0,1,0],[1,1,0,0]]);
    G1 := sub<G | {A,S}>;
    G2 := sub<G | {B,T}>;
    IsGLZConjugate(G1,G2);
\end{verbatim}
Now, using Proposition~\ref{proposition:G-non-rigid-via-modules}, we see that there exists no $\GG$-birational map $\mathbb{P}^n\dasharrow V$.
\end{remark}

By Remark~\ref{remark:permutohedron-Cremona}, the action of the standard Cremona involution $\tau_{\mathbb{P}^n}$ can be regularized on the permutohedron~$\widetilde{X}$, i.e. the involution $$
\tau_{\widetilde{X}}=\pi^{-1}\circ\tau_{\mathbb{P}^n}\circ\pi
$$
is biregular. Using this and the fact that $\pi$ is $\mathbb{W}$-equivariant, one can show that
$$
\mathrm{Aut}(\widetilde{X})\simeq\mathbb{T}\rtimes\big(\mathfrak{S}_{n+1}\times\mumu_2\big).
$$
Now, applying equivariant Minimal Model Program to $\widetilde{X}$, we get an $\mathrm{Aut}(\widetilde{X})$-equivariant birational map $\nu\colon\widetilde{X}\dasharrow Y$ such that $\nu$ contracts toric divisors $\widetilde{E}_1,\ldots,\widetilde{E}_{n-2}$ (but $\nu$ does not contract $\widetilde{E}_0$), and $Y$ is a smooth toric Fano variety, which is known as centrally symmetric $n$-dimensional toric del Pezzo variety \cite{Casagrande,Ewald,KlyachkoVoskresenskii}, cf. also \cite[\S4.4]{BDM}. The toric Fano variety $Y$ can be obtained from $X_1$ by antiflipping the strict transforms of the torus invariant subspaces in $\mathbb{P}^n$ of dimension $1,\ldots,\frac{n}{2}-1$ in this order (recall that $n+1\geqslant 5$ is prime by Lemma~\ref{lemma:prime}, so that in particular $n$ is even). To be precise, we have the following $\mathbb{W}$-equivariant commutative diagram
$$
\xymatrix{
X_1\ar[d]_{\pi_0}\ar@{-->}[rr]&&Y\ar@{->}[rr]^{\tau_{Y}}&&Y&&X_1\ar[d]^{\pi_0}\ar@{-->}[ll] \\
\mathbb{P}^n\ar@{-->}[rrrrrr]^{\tau_{\mathbb{P}^n}} &&&&&& \mathbb{P}^n}
$$
where $X_1\dasharrow Y$ is the small birational map described above, and $\tau_Y$ is a biregular involution. Thus, we may identify $\mathbb{W}$ with a subgroup in $\mathrm{Aut}(Y)$. Then
$$
\mathrm{Aut}(Y)=\langle\mathbb{W},\tau_Y\rangle\simeq\mathbb{T}\rtimes\big(\mathfrak{S}_{n+1}\times\mumu_2\big).
$$
By construction, the torus invariant divisors in $Y$ are strict transforms on $Y$ of the irreducible components of the divisors $\widetilde{E}_0$ and $\widetilde{E}_{n-1}$ described earlier, which form one $\mathrm{Aut}(Y)$-irreducible divisor. In particular, $Y$ is a toric $\mathrm{Aut}(Y)\mathbb{Q}$-Fano variety. Note that
$$
\langle \GG,\tau_Y\rangle=\mathbb{T}\rtimes\langle G,\tau_Y\rangle\simeq\mathbb{T}\rtimes\big(\mathfrak{D}_{n+1}\times\mumu_{2}\big)\simeq\mathbb{T}\rtimes\mathfrak{D}_{2n+2},
$$
and
$\langle G,\tau_Y\rangle\simeq \mathfrak{D}_{n+1}\times\mumu_{2}\simeq \mathfrak{D}_{2n+2}$.

\begin{remark}
\label{remark:Y}
In the standard lattice $\mathbb{Z}^{n}$, the fan of the toric variety $Y$ has $2n+2$ one-dimensional rays with primitive vectors
\begin{align*}
v_1&=(1,0,\ldots,0,0),\\
u_1&=(-1,0,\ldots,0,0),\\
&\ldots\\
v_{n}&=(0,0,\ldots,0,1),\\
u_{n}&=(0,0,\ldots,0,-1),\\
v_{n+1}&=(-1,-1,\ldots,-1,-1),\\
u_{n+1}&=(1,1,\ldots,1,1).
\end{align*}
If $n=4$, then $Y$ is the smooth toric Fano 4-fold \textnumero 118 in \cite{Batyrev}.
\end{remark}

The constructed birational map $\mathbb{P}^n\dasharrow Y$ is equivariant with respect to the action of the group~$\mumu_{n+1}$ on $\mathbb{P}^n$ given by~\eqref{eq:cyclic-group-action}, which we used earlier to construct the toric variety~\mbox{$V=\mathbb{P}^n/\mumu_{n+1}$}. Now, we let $U=Y/\mumu_{n+1}$.

\begin{lemma}
\label{lemma:U-terminal}
The variety $U$ has terminal singularities.
\end{lemma}

\begin{proof}
Let $v_1,\ldots,v_{n+1}, u_1,\ldots, u_{n+1}$ be the vectors listed in
Remark~\ref{remark:Y}, and let
$$
w=\frac{1}{n+1}(1,\ldots,n).
$$
Let $\mathcal{N}$ be the lattice in $\mathbb{R}^n$ spanned by~the vectors $v_1,\ldots,v_n$ and~$w$. Thus, $U$ is a $\mathbb{Q}$-factorial toric variety whose fan sits in the lattice
$\mathcal{N}$, and the one-dimensional cones of this fan are generated by the vectors~\mbox{$v_1,\ldots,v_{n+1}, u_1,\ldots, u_{n+1}$}.
We have to show that $U$ has terminal singularities.
In fact, we are going to show a stronger assertion: if $\breve{U}$ is a toric variety defined by a simplicial fan $\mathcal{F}$ in the lattice~$\mathcal{N}$, such that $\mathcal{F}$ has $2n+2$ one-dimensional cones generated by the vectors~\mbox{$v_1,\ldots,v_{n+1}, u_1,\ldots, u_{n+1}$}, then $\breve{U}$ has terminal singularities.

Consider an $n$-dimensional cone of the fan $\mathcal{F}$, and let $w_1,\ldots,w_n$
be the generators of its edges; thus, $w_1,\ldots,w_n$ are $n$ vectors among
$v_1,\ldots,v_{n+1}, u_1,\ldots, u_{n+1}$.
Note that the vectors $v_i$ and~\mbox{$u_i=-v_i$} cannot appear simultaneously among
$w_1,\ldots,w_n$. Next, observe that the lattice $\mathcal{N}$ and the collection of the vectors $v_1,\ldots,v_{n+1}, u_1,\ldots, u_{n+1}$ are invariant with respect to
the action~\eqref{eq:cyclic-group-action} of the group $\mumu_{n+1}$ (this group acts by cyclic permutation of the vectors $v_1,\ldots,v_{n+1}$ and the vectors $u_1,\ldots, u_{n+1}$).
Thus, replacing $w_1,\ldots,w_n$ by their images under a suitable element of~$\mumu_{n+1}$, we may assume that none of the vectors $v_{n+1}$ and $u_{n+1}$ appears
among~\mbox{$w_1,\ldots,w_n$}. Furthermore,
replacing $w_1,\ldots,w_n$ by $-w_1,\ldots,-w_n$ if necessary, we may assume that at least
one of the vectors $v_1,\ldots,v_n$ is in $\{w_1,\ldots,w_n\}$. We note also
that one cannot have
$$
\{w_1,\ldots,w_n\}=\{v_1,\ldots,v_n\},
$$
because the cone generated by $v_1,\ldots,v_n$ cannot be contained in any fan which has a one-dimensional cone generated by $u_{n+1}=v_1+\ldots+v_n$.
Therefore, it remains to consider the case when
$w_1,\ldots,w_n$ is a collection of $n$ vectors among $v_1,\ldots,v_n, u_1,\ldots,u_n$
where $v_i$ and $u_i$ do not appear simultaneously.
In other words, for some splitting of the set of indices
$$
\{1,\ldots,n\}=I_v\sqcup I_u
$$
one has $w_i=v_i$ when $i\in I_v$ and $w_i=u_i$ when $i\in I_u$.

Observe that the interior of the cone generated by $w_1,\ldots,w_n$
contains the vector
$$
w'=w+\sum\limits_{i\in I_u} u_i=\frac{1}{n+1}\left(\sum\limits_{i\in I_v} iv_i+\sum\limits_{i\in I_u} (n+1-i)u_i\right)\in\mathcal{N}.
$$
Moreover, the vectors $w_1,\ldots,w_n$ and $w'$ generate the lattice $\mathcal{N}$.
It follows that our cone describes an affine toric variety $U^\circ$
with a cyclic quotient singularity of type
$$
\frac{1}{n+1}\left(\overline{\varsigma_1 1}, \ldots, \overline{\varsigma_n n}\right),
$$
where $\varsigma_i=1$ for $i\in I_v$ and $\varsigma_i=-1$ for $i\in I_u$.
As in the proof of Lemma~\ref{lemma:V-terminal}, by $\overline{a}$ we denote the integer such that $0\leqslant \overline{a}\leqslant n$ and $a\equiv \overline{a}\ \mathrm{mod}\ {n+1}$.

Now, according to Reid--Tai criterion \cite[Theorem~4.11]{YPG}, we see that the variety $U^\circ$ is terminal if and only if for each $r\in\{1,2,\ldots,n\}$, one has
\begin{equation}
\label{eq:Reid-Tai}
\sum \limits_{i=1}^n \overline{\varsigma_i ri} > n+1.
\end{equation}
Since $n+1\geqslant 5$ is a prime number by Lemma~\ref{lemma:prime}, we have
$$
\{\overline{r},\overline{2r},\ldots,\overline{nr}\}=\{1,2,\ldots,n\}.
$$
Furthermore, $n=2k$ is an even number, and one has $k\geqslant 2$.
Thus, we compute
\begin{multline*}
\sum\limits_{i=1}^n \overline{\varsigma_i ri}\geqslant
\sum\limits_{i=1}^k i+\sum\limits_{i=k+1}^{2k} \overline{-i}=
2\sum\limits_{i=1}^k i=2\cdot\frac{k(k+1)}{2}=\\=
(k^2-2k+1)+(k-2)+(2k+1)=(k-1)^2+(k-2)+(2k+1)>2k+1=n+1.
\end{multline*}
Therefore, inequality~\eqref{eq:Reid-Tai} holds, so $U^\circ$ (and hence also $\breve{U}$,
and in particular our initial variety~$U$)
has terminal singularities.
\end{proof}

The quotient map $Y\to U$ is $\langle \GG,\tau_Y\rangle$-equivariant, and the action of the group $\langle \GG,\tau_Y\rangle$ on $U$ gives a~group homomorphism $\langle \GG,\tau_Y\rangle\to\mathrm{Aut}(U)$, whose kernel is $\mumu_{n+1}$. Let $\GG_U$ be its image. Then
$$
\GG_U=\mathbb{T}_U\rtimes\langle G_U,\tau_U\rangle,
$$
where $\mathbb{T}_U$ is the maximal torus in $\mathrm{Aut}(U)$, the group $G_U$ is the image of the group $G$, and $\tau_U$ is the image of the involution $\tau_Y$. We have
\begin{align*}
G_U&\simeq\mathfrak{D}_{n+1},\\
\langle G_U,\tau_U\rangle&\simeq\mathfrak{D}_{n+1}\times\mumu_2\simeq \mathfrak{D}_{2n+2}.
\end{align*}
By construction, $U$ is a toric $\langle \GG_U,\tau_U\rangle\mathbb{Q}$-Fano variety, i.e., we have
$$
\mathrm{rk}\,\mathrm{Cl}(U)^{\langle \GG_U,\tau_U\rangle}=\mathrm{rk}\,\mathrm{Cl}(U)^{\langle G_U,\tau_U\rangle}=1.
$$
However, $U$ is not a $\GG_U\mathbb{Q}$-Fano variety, since $\mathrm{rk}\,\mathrm{Cl}(U)^{\GG_U}=\mathrm{rk}\,\mathrm{Cl}(U)^{G_U}=2$.

\begin{remark}
\label{remark:UV}
Recall that $V=\mathbb{P}^n/\mumu_{n+1}$. Thus, there exists a $\GG_U$-birational map $U\dasharrow V$, where the $\GG_U$-action on $V$ is given by $\GG_V\simeq \GG_U$.
\end{remark}

Note that the group $\langle G_U,\tau_U\rangle\simeq \mathfrak{D}_{2n+2}$ contains two (normal) subgroups isomorphic to $\mathfrak{D}_{n+1}$. One of them is our group $G_U$. Let $G_U^\prime$ be the other one. Set $\GG_U^\prime=\langle\mathbb{T}_U, G_U^\prime\rangle$. Then
$$
\mathrm{rk}\,\mathrm{Cl}(U)^{\GG_U^\prime}=\mathrm{rk}\,\mathrm{Cl}(U)^{G_U^\prime}=1,
$$
because the group $G_U^\prime$ acts transitively on irreducible torus invariant divisors in $U$.

\begin{lemma}
\label{lemma:toric-final-final}
There exists a $\GG$-birational map $\mathbb{P}^n\dasharrow U$, where the $\GG$-action on $U$ is given by the action of the group $\GG_U^\prime\simeq \GG$.
\end{lemma}

\begin{proof}
The action of the group $\GG$ on $\mathbb{P}^n$, gives a natural monomorphism $G\hookrightarrow\mathrm{GL}_{n}(\mathbb{Z})\simeq\mathrm{Aut}(\mathbb{T})$. Similarly, the action of the group $\langle \GG_U,\tau_U\rangle$ on  $U$ gives a~monomorphism $\langle G_U,\tau_U\rangle\hookrightarrow\mathrm{GL}_{n}(\mathbb{Z})$ such that the image of the involution $\tau_U$ is the scalar matrix $-I_n$. It follows from
Proposition~\ref{proposition:G-non-rigid-via-modules} that the following two conditions are equivalent:
\begin{itemize}
\item there exists a $\GG$-birational map $\mathbb{P}^n\dasharrow U$, where $\GG$-action on $U$ is given by $\GG_U^\prime$;
\item the images of the groups $G$ and $G_U^\prime$ in $\mathrm{GL}_{n}(\mathbb{Z})$ are conjugate.
\end{itemize}
On the other hand, arguing as in the proof of Lemma~\ref{lemma:toric-final} and using Remark~\ref{remark:toric-final}, we see that the~image of the group $G$ in $\mathrm{GL}_{n}(\mathbb{Z})$ is generated by the following two $n\times n$ matrices:
$$
A=\left(
\begin{array}{ccccccc}
0 & 0 & 0 & \ldots & 0 & 0 & -1\\
1 & 0 & 0 & \ldots & 0 & 0 & -1\\
0 & 1 & 0 & \ldots & 0 & 0 & -1\\
\vdots&\vdots&\vdots& \ddots &\vdots&\vdots&\vdots\\
0 & 0 & 0 & \ldots & 1& 0 & -1\\
0 & 0 & 0 & \ldots & 0& 1 & -1
\end{array}
\right),\quad
S=\left(
\begin{array}{ccccccc}
-1 & 0 & 0 & \ldots & 0& 0 & 0 \\
-1 & 0 & 0 & \ldots & 0& 0 & 1 \\
-1 & 0 & 0 & \ldots & 0& 1 & 0 \\
\vdots&\vdots&\vdots& \ddots &\vdots&\vdots&\vdots\\
-1 & 0 & 1 & \ldots & 0& 0 & 0\\
-1 & 1 & 0 & \ldots & 0& 0 & 0
\end{array}
\right).
$$
Furthermore, the image of the group $G_U^\prime$ in $\mathrm{GL}_{n}(\mathbb{Z})$ is generated by the following two $n\times n$ matrices:
$$
B=\left(
\begin{array}{ccccccc}
-n & 0 & 0 & \ldots & 0& 0 & -n-1\\
1 & 0 & 0 & \ldots & 0& 0 & 1\\
2 & 1 & 0 & \ldots & 0& 0 & 2\\
3 & 0 & 1 & \ldots & 0& 0 & 3\\
\vdots&\vdots&\vdots& \ddots &\vdots&\vdots&\vdots\\
n-2 & 0 & 0 & \ldots & 1& 0 & n-2\\
n-1 & 0 & 0 & \ldots & 0& 1 & n-1
\end{array}
\right),\quad
T=\left(
\begin{array}{ccccccc}
1 & 0 & 0 & \ldots & 0& 0 & 0 \\
-1 & 0 & 0 & \ldots & 0& 0 & -1 \\
-1 & 0 & 0 & \ldots& 0 & -1 & 0 \\
-1 & 0 & 0 & \ldots& -1 & 0 & 0 \\
\vdots&\vdots&\vdots& \ddots &\vdots&\vdots&\vdots\\
-1 & 0 & -1 & \ldots & 0& 0 & 0\\
-1 & -1 & 0 & \ldots & 0& 0 & 0
\end{array}
\right).
$$
Consider the following $n\times n$ matrix:
$$
C=\left(
\begin{array}{cccccccc}
-1 & -1& 0 & 0 &\ldots& 0 & 0 & -1\\
-1 & 0 & -1 & 0 &\ldots& 0 & 0 & -1\\
-1 & 1 & 0 & -1 &\ldots& 0 & 0 & -1\\
-1 & 0 & 1 & 0 &\ldots& 0 & 0 & -1\\
\vdots&\vdots&\vdots& \ldots &\ddots&\vdots &\vdots&\vdots\\
-1 & 0 & 0 & 0 &\ldots& 0 & -1 & -1\\
-1 & 0 & 0 & 0 &\ldots& 1 & 0 & -2\\
0 & 0 & 0 & 0 &\ldots& 0 & 1 & -1
\end{array}
\right).
$$
Then $C\in\mathrm{GL}_{n}(\mathbb{Z})$. Moreover, we have $C^{-1}AC=B$ and $C^{-1}SC=T$.
\end{proof}

Since $U\not\simeq\mathbb{P}^n$, an application of Lemma~\ref{lemma:toric-final-final}
completes the proof
of Theorem~\ref{theorem:toric-2}.

\section{Four-dimensional projective space}
\label{section:P4}

The goal of this section is to discuss
primitive subgroups of $\mathrm{PGL}_5(\mathbb{C})$
and prove Theorem~C.

Two finite primitive subgroups of the group $\mathrm{PGL}_{5}(\mathbb{C})$ are conjugate if and only if they are isomorphic.
Moreover, up to conjugation, the group $\mathrm{PGL}_{5}(\mathbb{C})$ contains $11$ primitive finite subgroups, see~\mbox{\cite{Br67,Fe71}}. These subgroups can be described as follows.

Up to conjugation, $\mathrm{PGL}_{5}(\mathbb{C})$ contains a unique simple subgroup isomorphic to $\mathrm{PSp}_{4}(\mathbf{F}_3)$,
which can be described as the~automorphism group of the~quartic 3-fold in $\mathbb{P}^4$ given by
\begin{equation}
\label{equation:Burkhardt}
x_1\big(x_1^3+x_2^3+x_3^3+x_4^3+x_5^3\big)+3x_2x_3x_4x_5=0,
\end{equation}
and  $\mathrm{PGL}_{5}(\mathbb{C})$ contains a unique simple subgroup isomorphic to $\mathrm{PSL}_2(\mathbf{F}_{11})$,
which can be described as the~automorphism group of the~cubic 3-fold in $\mathbb{P}^4$ given by
\begin{equation}
\label{equation:Klein}
x_1x_2^2+x_2x_3^2+x_3x_4^2+x_4x_5^2+x_5x_1^2=0.
\end{equation}
These two simple subgroups of the group $\mathrm{PGL}_{5}(\mathbb{C})$ are primitive. Note that the quartic \eqref{equation:Burkhardt} is known as the Burkhardt quartic 3-fold,
and the cubic \eqref{equation:Klein} is known as the Klein cubic 3-fold.

Recall that $\mathfrak{S}_6$ faithfully acts on $\mathbb{P}^4$ via its standard five-dimensional irreducible representation.
This gives a primitive subgroup in $\mathrm{PGL}_{5}(\mathbb{C})$, which is isomorphic to $\mathfrak{S}_6$.
This subgroup contains a unique subgroup isomorphic to $\mathfrak{A}_6$, which is also primitive.
Up to conjugation, it also contains two subgroups isomorphic to $\mathfrak{S}_5$,
which are called standard and non-standard --- the non-standard subgroup is primitive, and the standard subgroup is intransitive.

\begin{lemma}[{cf. \cite[Corollary~2.5(ii)]{CheltsovShramov-S6} and \cite{CheltsovKuznetsovShramov}}]
\label{lemma:S6-A6-S5nst}
Suppose that $G\subset \mathrm{PGL}_5(\CC)$ is a primitive subgroup isomorphic to
$\mathfrak{S}_6$, $\mathfrak{A}_6$, or~$\mathfrak{S}_5$.
Then the linear system $|\mathcal{O}_{\mathbb{P}^4}(4)|$ contains a unique $G$-invariant pencil.
\end{lemma}

\begin{proof}
This follows from standard results on symmetric functions
(or from a direct computation of characters of the fourth symmetric power of the corresponding
representations).
\end{proof}

To describe the remaining $6$ primitive finite subgroups of the group $\mathrm{PGL}_{5}(\mathbb{C})$, let $\mathbb{H}_5$ be the~subgroup in $\mathrm{PGL}_{5}(\mathbb{C})$
generated by the following projective transformations:
\begin{align*}
[x_1:x_2:x_3:x_4:x_5]&\mapsto [x_2:x_3:x_4:x_5,x_1],\\
[x_1:x_2:x_3:x_4:x_5]&\mapsto [e^{\frac{2\pi \sqrt{-1}}{5}} x_1:e^{\frac{4\pi \sqrt{-1}}{5}} x_2:e^{\frac{6\pi \sqrt{-1}}{5}} x_3:e^{\frac{8\pi \sqrt{-1}}{5}} x_4:x_5],
\end{align*}
and let $\mathbb{N}_5$ be the~normalizer of the~subgroup $\mathbb{H}_5$ in $\mathrm{PGL}_5(\mathbb{C})$. Then $\mathbb{H}_5$ is transitive and imprimitive,
one has~\mbox{$\mathbb{H}_5\simeq\mumu_5^2$}, and it follows from \cite{HoMu73} that
$$
\mathbb{N}_5/\mathbb{H}_5\simeq\mathrm{SL}_2(\mathbf{F}_{5}).
$$
Moreover, we have $\mathbb{N}_5\simeq \mathbb{H}_5\rtimes\mathrm{SL}_2(\mathbf{F}_{5})$, the subgroup $\mathbb{N}_5$ is primitive~\cite{Br67},
and it contains $5$ more primitive subgroups $G_1$, $G_2$, $G_3$, $G_4$, $G_5$ such that all of them contain $\mathbb{H}_5$,
and
\begin{align*}
G_1/\mathbb{H}_5&\simeq \mumu_3,\\
G_2/\mathbb{H}_5&\simeq \mumu_6,\\
G_3/\mathbb{H}_5&\simeq \mathfrak{Q}_8,\\
G_4/\mathbb{H}_5&\simeq \mumu_3\rtimes\mumu_4,\\
G_5/\mathbb{H}_5&\simeq \mathrm{SL}_2(\mathbf{F}_3),
\end{align*}
where $\mathfrak{Q}_8$ denotes the quaternion group of order $8$. Explicit generators of $\mathbb{N}_5$ and its five primitive subgroups $G_1$, $G_2$, $G_3$, $G_4$, $G_5$ can be found in \cite{KangZhangSjiYuYau}.

The described $11$ primitive subgroups in $\mathrm{PGL}_{5}(\mathbb{C})$ are all primitive finite subgroups in $\mathrm{PGL}_{5}(\mathbb{C})$.
Note that the original classification in \cite{Br67} also listed a primitive subgroup isomorphic to $\mathfrak{A}_5$, but later it has been
determined that this subgroup is transitive and imprimitive.

\begin{proposition}
\label{proposition:8-groups}
If $\mathbb{P}^4$ is $G$-birationally rigid for some finite subgroup $G\subset\mathrm{PGL}_{5}(\mathbb{C})$, then $G$ is conjugate to one of the following subgroups:
$\mathrm{PSp}_{4}(\mathbf{F}_3)$, $\mathrm{PSL}_2(\mathbf{F}_{11})$, $\mathbb{N}_5$, $G_1$, $G_2$, $G_3$, $G_4$,~$G_5$.
\end{proposition}

\begin{proof}
We know from Theorem~A that $G$ is primitive. Thus,
it is conjugate to one of the~$11$ subgroups described above.
On the other hand, if $G$ is a primitive subgroup of $\mathrm{PGL}_{5}(\mathbb{C})$ isomorphic to $\mathfrak{S}_6$, $\mathfrak{A}_6$, or $\mathfrak{S}_5$,
then $\mathbb{P}^4$ is not $G$-birationally rigid by Lemma~\ref{lemma:Pn-pencil}, because $\mathbb{P}^4$ contains a \mbox{$G$-invariant} pencil of quartic hypersurfaces,
see Lemma~\ref{lemma:S6-A6-S5nst}.
\end{proof}

\begin{remark}
\label{remark:indus}
We do not know whether $\mathbb{P}^4$ is $G$-birationally rigid in the case when $G$ is one of the groups  $\mathrm{PSL}_2(\mathbf{F}_{11})$, $\mathbb{N}_5$, $G_1$, $G_2$, $G_3$, $G_4$, $G_5$. Observe that $\mathbb{P}^4$ is not $G$-birationally superrigid if~$G=G_1$ or $G=G_2$. Indeed, in these cases, $G$ leaves invariant a smooth quintic elliptic curve in~$\mathbb{P}^4$. Therefore, there exists a $G$-equivariant quadro-cubic Cremona transformation of $\mathbb{P}^4$, which was described in \cite{Semple}, cf.~\cite{Polishchuk,Rubtsov,Indus}.
\end{remark}

\begin{remark}
\label{remark:P4-no-quadrics}
If $G\subset\mathrm{PGL}_{5}(\mathbb{C})$ is one of the groups
$\mathrm{PSp}_{4}(\mathbf{F}_3)$, $\mathrm{PSL}_2(\mathbf{F}_{11})$, $\mathbb{N}_5$, $G_1$, $G_2$, $G_3$, $G_4$,~$G_5$, then there are no $G$-invariant quadrics in $\PP^4$.
Indeed, let $\widehat{\mathbb{H}}_5$ denote
the preimage of the group~$\mathbb{H}_5$ in $\mathrm{SL}_5(\CC)$. Then
$\widehat{\mathbb{H}}_5$ is a transitive group of order $125$ with center $Z(\widehat{\mathbb{H}}_5)\simeq \mumu_5$.
Moreover, every irreducible representation of $\widehat{\mathbb{H}}_5$ where $Z(\widehat{\mathbb{H}}_5)$ acts non-trivially has dimension $5$. Hence there are no one-dimensional subrepresentations in the space of quadratic polynomials on $\CC^5$. This implies that there are no quadric hypersurfaces in $\PP^4$ which are invariant with respect to $\mathbb{H}_5$ or with respect to any group containing $\mathbb{H}_5$. As for the groups $\mathrm{PSp}_{4}(\mathbf{F}_3)$ and  $\mathrm{PSL}_2(\mathbf{F}_{11})$, the absence of invariant quadrics can be deduced either from a direct computation of the character of the second symmetric power of the corresponding representation, or from the fact that these groups cannot be embedded into the birational automorphism group of $\PP^3$, see~\cite{ProkhorovSimple}.
\end{remark}

\medskip
In the remaining part of this section, we prove Theorem~C.
Let $X_4$ be the quartic 3-fold in~$\mathbb{P}^4$ that is given by
equation~\eqref{equation:Burkhardt}.
Then
$$
\mathrm{Aut}(X_4)\simeq\mathrm{PSp}_4\big(\mathbf{F}_3\big).
$$
Set $G=\mathrm{Aut}(X_4)$. We have to show that $\mathbb{P}^4$ is $G$-birationally superrigid. Before doing this, let us present some facts about the action of the group $G$ on $X_4$ and $\mathbb{P}^4$.

First, we recall that the singular locus $\mathrm{Sing}(X_4)$ consists of $45$ isolated ordinary double points (nodes), which form a $G$-orbit.
Recall also that
\begin{itemize}
\item a line in $\mathbb{P}^4$ can contain $1$, $2$, or $3$ nodes,

\item a plane in $\mathbb{P}^4$ can contain $1$, $2$, $3$, $4$, $6$, or $9$ nodes,

\item a hyperplane in $\mathbb{P}^4$ can contain $1$, $2$, $3$, $4$, $7$, $10$, $12$, or $18$ nodes.
\end{itemize}
The planes in $\mathbb{P}^4$ containing $9$ nodes of $X_4$ are called Jacobi planes. These planes are contained in~$X_4$; conversely, every plane contained in $X_4$ is a Jacobi plane.
In total, the quartic $X_4$ contains~$40$ Jacobi planes, and there are exactly $8$ Jacobi planes in $X_4$ that pass through a given node.
The union of all Jacobi planes in $X_4$ is a divisor in $|-10K_{X_4}|$, which we denote by~$\mathbf{J}$.
Two distinct Jacobi planes either intersect by a point or by a line.
If they intersect by a line, then this line contains three nodes of $X_4$,
and there are no more Jacobi planes passing through this line. Furthermore,
any line containing three nodes of $X_4$ is an intersection of two Jacobi planes.
The number of such lines is $240$, and they
form one $G$-irreducible curve.
We refer the reader to~\cite{Baker} and~\cite{Hunt} for more details on
configurations of nodes, lines, and planes on~$X_4$.

The following facts are well known to experts and are easy to check.

\begin{lemma}[{see e.g. \cite[p.~26]{Atlas}}]
\label{lemma:PSp43-subgroups}
The indices of proper maximal subgroups of $G$ equal $27$, $36$, $40$, and $45$.
The subgroups of indices $27$ and $36$ are isomorphic to $\mumu_2^4\rtimes\mathfrak{A}_5$
and $\mathfrak{S}_6$, respectively.
\end{lemma}

\begin{lemma}
\label{lemma:Burkhardt-40}
There are no $G$-orbits of length less than $40$ in $\PP^4$.
\end{lemma}

\begin{proof}
From Lemma~\ref{lemma:PSp43-subgroups}, we see that
a stabilizer of a point in an orbit of length less than $40$ must be isomorphic to
$\mumu_2^4\rtimes\mathfrak{A}_5$
and $\mathfrak{S}_6$. However, the restrictions of the corresponding representation
to these subgroups are irreducible.
\end{proof}

\begin{lemma}
\label{lemma:Burkhardt-computation}
Then there are no $G$-invariant hypersurfaces
of degree  of degree $1$, $2$, $3$, and $5$ in $\PP^4$.
There exists a unique $G$-invariant hypersurface of
degree $4$, and a unique $G$-invariant hypersurface of
degree $6$ in $\PP^4$.
\end{lemma}

\begin{proof}
A direct computation of characters of the symmetric powers of the corresponding representation.
\end{proof}

The unique $G$-invariant hypersurface of
degree $4$ is our Burkhardt quartic $X_4$.
Let us denote by~$X_6$ the unique $G$-invariant hypersurface of
degree $6$. The next fact is probably well-known,
but we could not find an appropriate reference.

\begin{lemma}
\label{lemma:Burkhardt-surfaces}
Let $S$ be a $G$-invariant surface in $\mathbb{P}^4$ such that $\mathrm{deg}(S)\leqslant 24$. Then $\mathrm{deg}(S)=24$, the surface $S$ is irreducible, and~\mbox{$S=X_4\cap X_6$}.
\end{lemma}

\begin{proof}
One has $\mathrm{Cl}^G(X_4)=\mathbb{Z}[-K_{X_{4}}]$, see e.g. \cite[\S2]{CPS} or \cite[Theorem~2(1)]{CheltsovTschinkelZhang}.
Thus, if~\mbox{$S\subset X_4$}, then~\mbox{$d=4k$} and $S$ is cut out on $X_4$ by a $G$-invariant hypersurface of degree $k$.
On the other hand, by Lemma~\ref{lemma:Burkhardt-computation} there are no $G$-invariant hypersurfaces of degree $1$, $2$, $3$, and~$5$ in~$\PP^4$,
and the only $G$-invariant hypersurface of degree $4$ in $\mathbb{P}^4$ is the quartic $X_4$.
Hence, if~\mbox{$S\subset X_4$} and~\mbox{$\mathrm{deg}(S)\leqslant 24$}, then~\mbox{$\mathrm{deg}(S)=24$} and~\mbox{$S=X_4\cap X_6$}.
In particular, we see that $S$ is $G$-irreducible, because otherwise
there would exist a $G$-invariant surface of smaller degree in $X_4$.
Furthermore, by Lemma~\ref{lemma:PSp43-subgroups}
there are no subgroups of index at most $24$ in $G$,
and thus there are no non-trivial homomorphisms from~$G$ to~$\mathfrak{S}_m$ with $m\leqslant 24$.
Therefore, we see that $S$ is actually irreducible.

To complete the proof, we may assume that $S\not\subset X_4$. Moreover, we can assume that the surface~$S$ is $G$-irreducible. Let us seek for a contradiction.
Set $Z=S\cdot X_4$. Then $Z$ is a possibly $G$-reducible and possibly non-reduced $G$-invariant curve of degree at most $96$. Let $Z_1$ be one of its irreducible components, and let $C$
be a $G$-irreducible curve such that $Z_1$ is an irreducible component of $C$.
Note that $C$ is not contained in a hyperplane in $\mathbb{P}^4$, since otherwise the linear span of $C$ would be a $G$-invariant linear subspace of $\PP^4$.
Hence, if $C$ is irreducible, then $G$ acts faithfully on $C$.

We claim that $C\not\subset\mathbf{J}$.
Indeed, recall that two distinct Jacobi planes either intersect by a point or by a line,
and the lines contained in at least two Jacobi planes form one $G$-irreducible curve of degree~$240$.
Thus, if $C\subset\mathbf{J}$, then each irreducible component of $C$ is contained in exactly one Jacobi plane.
Now, we let $\Pi$ be a Jacobi plane, and let $\Gamma_{648}$ be its stabilizer in $G$. Then
$$
\Gamma_{648}\simeq\mathrm{SU}_{3}(\mathbf{F}_2).\mumu_3,
$$
its GAP ID is [648,533],
and its action on $\Pi\simeq\mathbb{P}^2$ gives a homomorphism
$$
\Gamma_{648}\to\mathrm{Aut}(\Pi)\simeq\mathrm{PGL}_{3}(\mathbb{C}).
$$
The image of this homomorphism
is the so-called the Hessian group of order $216$ isomorphic
to~\mbox{$\mathrm{PSU}_{3}(\mathbf{F}_2).\mumu_3$},
its GAP~ID is [216,153], and the kernel is isomorphic to $\mumu_3$.
It is straightforward to see that this group does not act on rational and elliptic curve,
and does not act on a union of three lines in $\PP^2$. Thus, the plane $\Pi$ does not contain $\Gamma_{648}$-invariant curves of degree at most~$3$. This shows that~\mbox{$C\not\subset\mathbf{J}$}, because the degree of $C$ is at most $96$.

We claim that $\mathrm{Sing}(X_4)\not\subset C$. Indeed, suppose that $\mathrm{Sing}(X_4)\subset C$.
Let $f\colon\widetilde{X}_4\to X_4$ be the blow up of all singular points of $X_4$,
let $\widetilde{\mathbf{J}}$ be the strict transform on $\widetilde{X}_4$ of the divisor~$\mathbf{J}$,
let $E$ be the union of all $f$-exceptional prime divisors,
and let $\widetilde{C}$ be the strict transform on $\widetilde{X}_4$ of the curve~$C$.
Then $f$ is $G$-equivariant and
$$
0\leqslant\widetilde{\mathbf{J}}\cdot\widetilde{C}=\Big(f^*\big(-10K_{X_4}\big)-4E\Big)\cdot\widetilde{C}=10\mathrm{deg}(C)-4E\cdot\widetilde{C}\leqslant 960-4E\cdot\widetilde{C}\leqslant 960-4|E\cap\widetilde{C}|,
$$
which gives $|E\cap\widetilde{C}|\leqslant 240$.
Let $E_1$ be an irreducible component of the divisor $E$, and let $\Gamma_{576}$ be the stabilizer in $G$ of the point $f(E_1)\in\mathrm{Sing}(X_4)$.
Then
$$
\Gamma_{576}\simeq\mumu_2.\mathfrak{A}_4\wr\mumu_2,
$$
the GAP ID of $\Gamma_{576}$ is [576,8277], and the kernel of the $\Gamma_{576}$-action on $E_1$ is the center of the group~$\Gamma_{576}$,
which is isomorphic to~$\mumu_2$. Hence, the image in $\mathrm{Aut}(E_1)$ of the group $\Gamma_{576}$ is isomorphic to $\mathfrak{A}_4\wr\mumu_2$, and its GAP ID is [288,1025]. The latter group contains a subgroup isomorphic to~\mbox{$\mathfrak{A}_4\times\mathfrak{A}_4$},
which acts on $E_1\simeq\mathbb{P}^1\times\mathbb{P}^1$ via the product action.
This implies that $E_1$ does not contain $\Gamma_{576}$-orbits of length at most~$15$, because the minimal length of a $\mathfrak{A}_4$-orbit in $\PP^1$ is $4$.
On the other hand, it follows from the inequality $|E\cap\widetilde{C}|\leqslant 240$
that~\mbox{$|E_1\cap\widetilde{C}|\leqslant 5$}, which is a contradiction.
Hence, we see that $\mathrm{Sing}(X_4)\not\subset C$. Since the action of
$G$ on $\mathrm{Sing}(X_4)$ is transitive, this means that every point of the intersection $\mathbf{J}\cap C$ is a smooth point of the quartic~$X_4$.

Now, let us revisit our Jacobi plane $\Pi$, and its stabilizer $\Gamma_{648}$.
Note that $\mathrm{Sing}(X_4)\cap\Pi$ is a $\Gamma_{648}$-orbit of length $9$, this is the only $\Gamma_{648}$-orbit in $\Pi$ of length less than~$24$,
the plane $\Pi$ contains a $\Gamma_{648}$-orbit of length~$24$, and the stabilizer of a point in this orbit is not cyclic.
Thus, since~\mbox{$\mathrm{Sing}(X_4)\cap C=\varnothing$}, we have
$$
\Pi\cdot C\geqslant |C\cap\Pi|\geqslant 24,
$$
which gives
$$
960\geqslant 10\mathrm{deg}(Z)\geqslant 10\mathrm{deg}(C)=-10K_{X_4}\cdot C=\mathbf{J}\cdot C\geqslant 40\cdot 24=960.
$$
This implies that $C=Z$, the equality $\mathrm{deg}(C)=96$ holds, the intersection $C\cap\Pi$ form one $\Gamma_{648}$-orbit of length $24$, and the curve $C$ is smooth at every point of the intersection $C\cap\Pi$.
If $C$ is irreducible, this immediately leads to a contradiction, because the stabilizer in $\mathrm{Aut}(C)$ of a smooth point is cyclic.
Hence, we conclude that $C$ is a reducible $G$-irreducible curve of degree $96$.
On the other hand, $G$ does not contain proper subgroups whose index divides $96$ by Lemma~\ref{lemma:PSp43-subgroups}.
The obtained contradiction completes the proof of the lemma.
\end{proof}

\begin{remark}
\label{remark:Burkhardt-surfaces}
Let $S\subset\mathbb{P}^4$ be a $G$-irreducible surface. If $S\not\subset X_4$, then $\mathrm{deg}(S)>24$ by Lemma~\ref{lemma:Burkhardt-surfaces}.
This also follows from Lemma~\ref{lemma:weighted-degree}.
Indeed, by Lemma~\ref{lemma:PSp43-subgroups}, we may assume that $S$ is irreducible.
Let $U=\mathbb{P}^4/G$,
and let $\pi\colon \mathbb{P}^4\to U$ be the quotient map. Then $U\simeq \mathbb{P}(2,3,5,6,9)$ and
$$
\mathcal{O}_{\mathbb{P}^4}(2)\sim_{\mathbb{Q}}\pi^*(\mathcal{O}_U(1)).
$$
On the other hand, as cycles in $U$, we have $\pi_*[S]=|G|[\pi(S)]$, so it follows from Lemma~\ref{lemma:weighted-degree} that
$$
4\deg(S)=\big(\mathcal{O}_{\mathbb{P}^4}(2)\big)^2\cdot Z=\big(\mathcal{O}_U(1)\big)^2\cdot\pi_*[S]= |G|\,\big(\mathcal{O}_U(1)\big)^2\cdot\pi(S)\geqslant\frac{|G|}{108}=240,
$$
so $\mathrm{deg}(S)\geqslant 60>24$.
\end{remark}

Now, we are ready to prove that $\mathbb{P}^4$ is $G$-birationally superrigid.
Suppose it is not. Then it follows from \cite[Corollary 3.3.3]{CheltsovShramov2015}
that there exists a non-empty $G$-invariant mobile linear
subsystem~\mbox{$\mathcal{M}\subset |\mathcal{O}_{\mathbb{P}^4}(n)|$}, for some positive integer $n$, such that
the singularities of the log pair $(\mathbb{P}^4,\frac{5}{n}\mathcal{M})$ are not canonical.
If this log pair is not canonical at a general point of a $G$-irreducible surface $S$, then
$$
\mathrm{mult}_{S}\big(\mathcal{M}\big)>\frac{n}{5},
$$
which implies that $\mathrm{deg}(S)<25$.
So, we have $S=X_4\cap X_6$ by Lemma~\ref{lemma:Burkhardt-surfaces}.
Therefore, the restriction $\mathcal{M}\vert_{X_4}$ is a linear subsystem
of $|nH|$, where $H$ is a hyperplane section of $X_4$, whose fixed part contains
a divisor
$$
\mu S\sim_{\mathbb{Q}} 6\mu H
$$
with $\mu>\frac{n}{5}$,
which gives a contradiction.

Hence, the singularities of the log pair $(\mathbb{P}^4,\frac{5}{n}\mathcal{M})$ are canonical away from a codimension $3$ subset in $\mathbb{P}^4$.
Then it follows from \cite[Remark 3.6]{CheltsovSarikyanZhuang2024} that the singularities of the log pair $(\mathbb{P}^4,\frac{15}{2n}\mathcal{M})$ are not log canonical.
If this log pair is not log canonical at a general point of a $G$-irreducible surface~$S$, then it follows from \cite[Theorem 3.1]{corti} that
$$
\mathrm{mult}_{S}\big(M_1\cdot M_2\big)>4\cdot\left(\frac{2n}{15}\right)^2
$$
for two general members $M_1$ and $M_2$ of the linear system $\mathcal{M}$. This
implies that
$$
4\cdot\left(\frac{2n}{15}\right)^2\cdot\mathrm{deg}(S)< n^2,
$$
so $\mathrm{deg}(S)<15$, which is impossible by Lemma~\ref{lemma:Burkhardt-surfaces}.
Thus, the singularities of $(\mathbb{P}^4,\frac{15}{2n}\mathcal{M})$ are log canonical away from a codimension $3$ subset in $\mathbb{P}^4$.

Now, let $\lambda<\frac{15}{2n}$ be the log canonical threshold of the log pair $(\mathbb{P}^4,\mathcal{M})$,
let $Z_1$ be a minimal center of log canonical singularities of the log pair $(\mathbb{P}^4,\lambda\mathcal{M})$,
let $Z$ be the $G$-irreducible subvariety in $\mathbb{P}^4$ such that $Z_1$ is its irreducible component,
and let $\epsilon$ be a very small positive rational number such that $(1+\epsilon)\lambda<\frac{15}{2n}$.
Then it follows from \cite[Lemma~2.4.10]{CheltsovShramov2015} that
there exists a mobile $G$-invariant linear system $\mathcal{D}\subset|\mathcal{O}_{\mathbb{P}^4}(k)|$ for some $k\gg 0$
and two very small positive rational numbers~$\epsilon_{1}$ and~$\epsilon_{2}$ such that
$$
(1-\epsilon_{1})\lambda\mathcal{M}+\epsilon_{2}\mathcal{D}\sim_{\mathbb{Q}} (1+\epsilon)\lambda\mathcal{M},%
$$
the~log pair $(\mathbb{P}^4,(1-\epsilon_{1})\lambda\mathcal{M}+\epsilon_{2}\mathcal{D})$ is log canonical,
and the irreducible components of $Z$ are the~only centers of log canonical singularities of this log pair. By \cite[Proposition 1.5]{Kawamata1997},
the irreducible components of $Z$ are disjoint, since $Z$ is a union of minimal centers of log canonical singularities of
the log pair~\mbox{$(\mathbb{P}^4,(1-\epsilon_{1})\lambda\mathcal{M}+\epsilon_{2}\mathcal{D})$}.

Observe that either $Z_1$ is a point and $Z$ is a $G$-orbit, or $Z_1$ is a curve and $Z$ is a $G$-irreducible curve.
In the latter case, the curve $Z_1$ is smooth by \cite[Theorem 1]{Kawamata1998}. Moreover, in both cases, it follow from the Nadel's vanishing theorem \cite[Theorem 9.4.8]{Lazarsfeld}  that
$$
h^1\big(\mathbb{P}^4,\mathcal{O}_{\mathbb{P}^4}(3)\otimes\mathcal{I}_Z\big)=0,
$$
where $\mathcal{I}_Z$ is the ideal sheaf of $Z$ on $\PP^4$.
Thus, we have the following exact sequence of $G$-representations:
\begin{equation}
\label{equation:long-exact-sequence}
0\longrightarrow H^0\big(\mathbb{P}^4,\mathcal{O}_{\mathbb{P}^4}(3)\otimes\mathcal{I}_Z\big)\longrightarrow
H^0\big(\mathbb{P}^4,\mathcal{O}_{\mathbb{P}^4}(3)\big)\longrightarrow
H^0\big(\mathcal{O}_Z(\mathcal{O}_{\mathbb{P}^4}(3)\big\vert_{Z})\big)\longrightarrow 0.
\end{equation}
Thus, if $Z$ is a $G$-orbit, then
$$
|Z|\leqslant h^0\big(\mathbb{P}^4,\mathcal{O}_{\mathbb{P}^4}(3)\big)=35.
$$
This is a contradiction, because $\mathbb{P}^4$ does not contain $G$-orbits of length less than~$40$ by Lemma~\ref{lemma:Burkhardt-40}.

Thus, we see that $Z$ is a smooth curve. Let $r$ be the number of irreducible components of the curve $Z$, let $d$ be the degree of the curve $Z_1$,
let $g$ be the genus of the curve $Z_1$, let $H$ be a general hyperplane in $\mathbb{P}^4$, and let $\varepsilon$ be sufficiently small positive rational number. Then it follows from Kawamata's subadjunction theorem \cite[Theorem~1]{Kawamata1998} that
$$
\big(K_{\mathbb{P}^4}+(1-\epsilon_{1})\lambda\mathcal{M}+\epsilon_{2}\mathcal{D}
+\varepsilon H\big)\vert_{Z_1}\sim_{\mathbb{Q}} K_{Z_1}+\Delta_{Z_1}
$$
for some effective $\mathbb{Q}$-divisor $\Delta_{Z_1}$ on the curve $Z_1$. Now, comparing the degrees of these divisors
on~$Z_1$, we get
\begin{equation}\label{eq:2g-2-52}
2g-2<\frac{5}{2}d.
\end{equation}
In particular, the divisor $3H\vert_{Z_1}$ of degree $3d$ is not special, so we have
$$
h^0\big(\mathcal{O}_Z(\mathcal{O}_{\mathbb{P}^4}(3)\big\vert_{Z})\big)=rh^0\big(\mathcal{O}_{Z_1}(3H\big\vert_{Z_1})\big)=r(3d-g+1),
$$
so \eqref{equation:long-exact-sequence} gives $35\geqslant r(3d-g+1)$. Hence, using inequality~\eqref{eq:2g-2-52}, we get $35>\frac{7rd}{4}$, so that $rd<20$. This gives $r=1$, because $G$ does not have proper subgroups of index less than~$20$ by Lemma~\ref{lemma:PSp43-subgroups}.
Therefore, again using~\eqref{eq:2g-2-52}, we obtain
$$
35\geqslant 3d-g+1>\frac{12(g-1)}{5}-g+1=\frac{7g-7}{5},
$$
which implies that $g<32$. On the other hand, since $Z=Z_1$ is irreducible, $G$ acts faithfully on it, because $Z$ is not contained in a hyperplane.
Thus, using the Hurwitz bound $\mathrm{Aut}(Z)\leqslant 84(g-1)$,
we get a contradiction in the case when $g\geqslant 2$.
Finally, observe that $G$ cannot act faithfully on a rational or elliptic curve.
The obtained contradiction completes the proof of Theorem~C.

\section{Five-dimensional projective space}
\label{section:P5}

The goal of this section is to discuss primitive subgroups $G\subset\mathrm{PGL}_6(\mathbb{C})$,
show that $\mathbb{P}^5$ is not $G$-birationally rigid for many of them, and prove
that $\mathbb{P}^5$ is $G$-birationally superrigid for one of them (Theorem~D). First, we recall the classification of these subgroups~\mbox{\cite[\S3]{Li71}}.
Namely, let $G$ be a~finite primitive subgroup in $\mathrm{PGL}_6(\mathbb{C})$. Then either
\begin{itemize}
\item[(I)] $G$ leaves invariant a Segre cubic scroll $\mathbb{P}^1\times\mathbb{P}^2\simeq Y\subset \mathbb{P}^5$,
\end{itemize}
or there~exists a~finite primitive subgroup $\widehat{G}\subset\mathrm{SL}_6(\mathbb{C})$ such that $\widehat{G}$ is mapped to $G$ via the natural projection $\mathrm{SL}_6(\mathbb{C})\to\mathrm{PGL}_6(\mathbb{C})$, and $\widehat{G}$ is isomorphic to one of the following groups:
\begin{itemize}
\item[(II)] $\mathrm{SL}_2(\mathbf{F}_5)$, %
\item[(III)] $2.\mathfrak{S}_5$, %
\item[(IV)]
\begin{itemize}
\item[(i)] $3.\mathfrak{A}_6$, %
\item[(ii)]$3.\mathfrak{A}_6\rtimes\mumu_2$,%
\end{itemize}
\item[(V)] $6.\mathfrak{A}_6$,%
\item[(VI)] $\mathfrak{A}_7$ or $\mathfrak{S}_7$,%
\item[(VII)] $3.\mathfrak{A}_7$,%
\item[(VIII)] $6.\mathfrak{A}_7$,%
\item[(IX)]%
\begin{itemize}
\item[(i)] $\mathrm{PSL}_2(\mathbf{F}_7)$, %
\item[(ii)] $\mathrm{PGL}_2(\mathbf{F}_7)$,%
\end{itemize}
\item[(X)]
\begin{itemize}
\item[(i)] $\mathrm{SL}_2(\mathbf{F}_7)$, %
\item[(ii)] $\mathrm{SL}_2(\mathbf{F}_7)\rtimes\mumu_2$,%
\end{itemize}
\item[(XI)] $\mathrm{SL}_2(\mathbf{F}_{11})$,%
\item[(XII)] $\mathrm{SL}_2(\mathbf{F}_{13})$,%
\item[(XIII)]
\begin{itemize}
\item[(i)] $\mathrm{PSp}_4(\mathbf{F}_3)$,%
\item[(ii)] $\mathrm{PSp}_4(\mathbf{F}_3)\rtimes\mumu_2$,%
\end{itemize}
\item[(XIV)]
\begin{itemize}
\item[(i)] $\mathrm{SU}_3(\mathbf{F}_3)$, %
\item[(ii)] $\mathrm{SU}_3(\mathbf{F}_3)\rtimes\mumu_2$,%
\end{itemize}
\item[(XV)]
\begin{itemize}
\item[(i)] $6.\mathrm{PSU}_4(\mathbf{F}_3)$, %
\item[(ii)] $6.\mathrm{PSU}_4(\mathbf{F}_3)\rtimes\mumu_2$,%
\end{itemize}
\item[(XVI)] $2.\mathrm{HaJ}$, where $\mathrm{HaJ}$ is the~Hall--Janko simple group,%
\item[(XVII)]
\begin{itemize}
\item[(i)] $6.\mathrm{PSL}_3(\mathbf{F}_4)$, %
\item[(ii)] $6.\mathrm{PSL}_3(\mathbf{F}_4)\rtimes\mumu_2$.%
\end{itemize}
\end{itemize}
In each of the cases (II)--(XVII), the subgroup $\widehat{G}$ is uniquely determined by its isomorphism class up to conjugation in $\mathrm{SL}_6(\mathbb{C})$;
note however that in some cases $\widehat{G}$ has a more than one six-dimensional representation
such that $\widehat{G}$ is a primitive subgroup of $\mathrm{GL}_6(\CC)$.

\begin{remark}
\label{remark:n-5-Segre}
Suppose that $\mathbb{P}^5$ contains a $G$-invariant Segre cubic scroll $Y\simeq\mathbb{P}^1\times\mathbb{P}^2$. Then the linear
system which consists of all quadrics in $\mathbb{P}^5$ passing through $Y$ provides a
rational map~\mbox{$\chi\colon \PP^5\dasharrow \PP^2$}. This map fits into
the following $G$-equivariant commutative diagram:
$$
\xymatrix{
&X\ar@{->}[dl]_{\pi}\ar@{->}[dr]^{\eta}\\%
\mathbb{P}^5\ar@{-->}[rr]^{\chi}&&\mathbb{P}^2}
$$
Here $\pi$ is the blow up of the scroll $Y$, and $\eta$ is a $\mathbb{P}^3$-bundle. In particular, $\mathbb{P}^5$ is not $G$-birationally rigid.
\end{remark}

\begin{lemma}
\label{lemma:n-5-SL25}
Suppose that $\widehat{G}$ is isomorphic to $\mathrm{SL}_2(\mathbf{F}_5)$. Then there are no $G$-invariant quadric hypersurfaces in $\PP^5$,
and the linear system $|\mathcal{O}_{\mathbb{P}^5}(3)|$ contains a $G$-invariant pencil.
\end{lemma}

\begin{proof}
The GAP ID of $\widehat{G}$ is [120,5], and $\widehat{G}$ has a unique irreducible $6$-dimensional representation.
A direct computation of characters of the symmetric powers shows that
$|\mathcal{O}_{\mathbb{P}^5}(2)|$ does not contain $G$-invariant hypersurfaces, and  $|\mathcal{O}_{\mathbb{P}^5}(3)|$ contains a $G$-invariant pencil.
For instance, we can use the following GAP script to do this:
\begin{verbatim}
    G:=SmallGroup(120,5);
    T:=CharacterTable(G);
    Ir:=Irr(T);
    U:=Ir[9];
    S:=SymmetricParts(T,[U],2);
    Print(MatScalarProducts(Ir,S));
    S:=SymmetricParts(T,[U],3);
    Print(MatScalarProducts(Ir,S));
\end{verbatim}
\end{proof}

\begin{lemma}
\label{lemma:n-5-2S5}
Suppose that $\widehat{G}$ is isomorphic to $2.\mathfrak{S}_5$. Then there are no $G$-invariant quadric hypersurfaces in $\PP^5$,
and the linear system $|\mathcal{O}_{\mathbb{P}^5}(4)|$ contains a unique $G$-invariant
pencil such that the hypersurfaces in this pencil are given by
$\widehat{G}$-invariant quartic polynomials.
\end{lemma}

\begin{proof}
The GAP ID of $\widehat{G}$ is [240,90], and the action of the group $G$ on $\mathbb{P}^5$ is given by one of two (complex conjugate) faithful
irreducible $6$-dimensional representations of the group $\widehat{G}$.
A direct computation of characters of the symmetric powers shows that
$|\mathcal{O}_{\mathbb{P}^5}(2)|$ does not contain $G$-invariant hypersurfaces, and  $|\mathcal{O}_{\mathbb{P}^5}(4)|$ contains a unique $G$-invariant pencil
such that the hypersurfaces in this pencil are given by
$\widehat{G}$-invariant quartic polynomials.
\end{proof}

\begin{lemma}
\label{lemma:n-5-3A6}
Suppose that $\widehat{G}$ is isomorphic to $3.\mathfrak{A}_6$ or to $3.\mathfrak{A}_6\rtimes\mumu_2$.
Then there are no $G$-invariant quadric hypersurfaces in $\PP^5$,
and the linear system $|\mathcal{O}_{\mathbb{P}^5}(3)|$ contains a unique $G$-invariant
pencil.
\end{lemma}

\begin{proof}
The GAP ID of $3.\mathfrak{A}_6$ is [1080,260], and this group has exactly two irreducible  $6$-dimensional representations, which are complex conjugate.
A direct computation of characters of the symmetric powers shows that
$|\mathcal{O}_{\mathbb{P}^5}(2)|$ does not contain $\mathfrak{A}_6$-invariant hypersurfaces, and  $|\mathcal{O}_{\mathbb{P}^5}(3)|$ contains a unique $\mathfrak{A}_6$-invariant pencil.
This proves the lemma in the case when $\widehat{G}\simeq 3.\mathfrak{A}_6$.
If $\widehat{G}\simeq 3.\mathfrak{A}_6\rtimes\mumu_2$,
then the required assertions follow from the case $\widehat{G}\simeq 3.\mathfrak{A}_6$.
\end{proof}

\begin{lemma}
\label{lemma:n-5-A7-S7}
Suppose that $\widehat{G}$ is isomorphic to $\mathfrak{A}_7$ or $\mathfrak{S}_7$.
Then the linear system $|\mathcal{O}_{\mathbb{P}^5}(4)|$ contains a unique $G$-invariant
pencil.
\end{lemma}

\begin{proof}
Each of these groups has a unique irreducible $6$-dimensional representation,
which is a summand of the $7$-dimensional permutation representation.
Now the well-known results about symmetric functions imply that $|\mathcal{O}_{\mathbb{P}^5}(4)|$ contains a unique $G$-invariant pencil.
\end{proof}

\begin{lemma}
\label{lemma:n-5-PSL27}
Suppose that $\widehat{G}$ is isomorphic to $\mathrm{PSL}_2(\mathbf{F}_7)$ or $\mathrm{PGL}_2(\mathbf{F}_7)$. Then the linear system~\mbox{$|\mathcal{O}_{\mathbb{P}^5}(3)|$}  contains a unique $G$-invariant pencil.
\end{lemma}

\begin{proof}
The GAP ID of $\mathrm{PSL}_2(\mathbf{F}_7)$ is $[168,42]$, and this group
has a unique irreducible $6$-dimensional representation.
A direct computation of the character of the third symmetric power shows that~\mbox{$|\mathcal{O}_{\mathbb{P}^5}(3)|$} contains a unique
$\mathrm{PSL}_2(\mathbf{F}_7)$-invariant pencil.
If $\widehat{G}\simeq \mathrm{PGL}_2(\mathbf{F}_7)$,
then the required assertion follows from the case $\widehat{G}\simeq \mathrm{PSL}_2(\mathbf{F}_7)$.
\end{proof}

\begin{corollary}
\label{corollary:P5-pencil-cases}
Suppose that $\widehat{G}$ is isomorphic to one of the groups
$\mathrm{SL}_2(\mathbf{F}_5)$,
$2.\mathfrak{S}_5$,
$3.\mathfrak{A}_6$, $3.\mathfrak{A}_6\rtimes \mumu_2$,
$\mathfrak{A}_7$, $\mathfrak{S}_7$,
$\mathrm{PSL}_2(\mathbf{F}_7)$, or $\mathrm{PGL}_2(\mathbf{F}_7)$.
Then $\PP^5$ is not $G$-birationally rigid.
\end{corollary}

\begin{proof}
According to Lemmas~\ref{lemma:n-5-SL25}, \ref{lemma:n-5-2S5}, \ref{lemma:n-5-3A6},  \ref{lemma:n-5-A7-S7}, and~\ref{lemma:n-5-PSL27},
the linear system $|\mathcal{O}_{\mathbb{P}^5}(n)|$ contains a $G$-invariant pencil for $n=3$ or $n=4$. Therefore, $\PP^5$ is not $G$-birationally rigid
by Lemma~\ref{lemma:Pn-pencil}.
\end{proof}

\begin{lemma}
\label{lemma:n-5-6A6}
Suppose that $\widehat{G}\simeq 6.\mathfrak{A}_6$. Then there are no $G$-invariant quadric hypersurfaces in~$\PP^5$, and the linear system~$|\mathcal{O}_{\mathbb{P}^5}(2)|$
contains a two-dimensional $G$-invariant linear subsystem.
\end{lemma}

\begin{proof}
The group $6.\mathfrak{A}_6$ can be described as the Schur cover of $\mathfrak{A}_6$, and it has four faithful irreducible $6$-dimensional representations.
The required assertions follow from a direct computation
of the character of the second symmetric power.
\end{proof}

\begin{lemma}
\label{lemma:n-5-SL27}
Suppose that $\widehat{G}\simeq\mathrm{SL}_2(\mathbf{F}_7)$. Then there are no $G$-invariant hypersurfaces of degree $1$, $2$, $3$, and $5$ in $\PP^5$, there exist a unique
$G$-invariant hypersurface of degree~$4$, and
the linear system~$|\mathcal{O}_{\mathbb{P}^5}(2)|$
contains a two-dimensional $G$-invariant linear subsystem.
\end{lemma}

\begin{proof}
The GAP ID of $\mathrm{SL}_2(\mathbf{F}_7)$ is~\mbox{$[336,114]$},
and this group has two faithful
irreducible $6$-dimensional representations.
The required assertions follow from a direct computation
of the characters of the symmetric powers.
\end{proof}

\begin{corollary}
\label{corollary:n-5-6A6-SL27}
Suppose that $\widehat{G}\simeq 6.\mathfrak{A}_6$ or $\widehat{G}\simeq\mathrm{SL}_2(\mathbf{F}_7)$. Then $\mathbb{P}^5$ is not $G$-birationally rigid.
\end{corollary}

\begin{proof}
By Lemmas~\ref{lemma:n-5-6A6} and~\ref{lemma:n-5-SL27},
the linear system $|\mathcal{O}_{\mathbb{P}^5}(2)|$
contains a two-dimensional linear subsystem $\mathcal{M}$.
Moreover, there are no $G$-invariant quadric hypersurfaces in $\PP^5$, which means that
the action of $G$ on $\mathcal{M}\simeq\PP^2$ is transitive.
In particular, $G$ acts on $\mathcal{M}$ faithfully, because $G$ is a simple group.
Furthermore, by the classification of finite subgroups of $\mathrm{PGL}_2(\CC)$
the group $G$ does not admit a faithful action
on a rational curve. So, $\PP^5$ is not $G$-birationally rigid by
Lemma~\ref{lemma:P5-net}.
\end{proof}

\begin{lemma}
\label{lemma:n-5-PSp4}
Suppose that $\widehat{G}$ is isomorphic to $\mathrm{PSp}_4(\mathbf{F}_3)$ or $\mathrm{PSp}_4(\mathbf{F}_3)\rtimes\mumu_2$.
Then $\mathbb{P}^5$ contains a unique $G$-invariant quadric hypersurface $Q$ and a unique $G$-invariant quintic hypersurface~$X$.
Moreover, the complete intersection $Q\cap X$ is reduced and $G$-irreducible.
\end{lemma}

\begin{proof}
The group $\mathrm{PSp}_4(\mathbf{F}_3)$ has a unique irreducible $6$-dimensional representation.
A direct computation of characters of the symmetric powers shows that
$\mathbb{P}^5$ contains a unique $\mathrm{PSp}_4(\mathbf{F}_3)$-invariant quadric hypersurface $Q$ and a unique $G$-invariant quintic hypersurface~$X$,
and $\mathbb{P}^5$ does not contain other reduced $G$-invariant hypersurfaces of degree at most $5$.
This implies that the complete intersection $Q\cap X$ is reduced and $G$-irreducible.
\end{proof}

\begin{corollary}
\label{corollary:n-5-PSp4}
Suppose that $\widehat{G}$ is isomorphic to $\mathrm{PSp}_4(\mathbf{F}_3)$ or $\mathrm{PSp}_4(\mathbf{F}_3)\rtimes\mumu_2$.
Then $\mathbb{P}^5$ is not $G$-birationally rigid.
\end{corollary}

\begin{proof}
Apply Lemma~\ref{lemma:n-5-PSp4} together with Lemma~\ref{lemma:Pn-ci}.
\end{proof}

\begin{remark}
As Zhijia Zhang informed us, using Magma one can show that the complete intersection in Lemma~\ref{lemma:n-5-PSp4} is irreducible, and its singularities consists of~$330$ isolated ordinary double points.
\end{remark}

\begin{remark}
If $\widehat{G}$ is isomorphic to $\mathrm{PSp}_4(\mathbf{F}_3)$ or $\mathrm{PSp}_4(\mathbf{F}_3)\rtimes\mumu_2$, then
the linear system $|\mathcal{O}_{\mathbb{P}^5}(n)|$ does not contain $G$-invariant pencils for $n\leqslant 5$. Thus, one cannot deduce Corollary~\ref{corollary:n-5-PSp4}
using Lemma~\ref{lemma:Pn-pencil} instead of Lemma~\ref{lemma:Pn-ci}.
\end{remark}

Summarizing, we obtain the following corollary.

\begin{corollary}
\label{corollary:n-5-rigid}
Let $G$ be a~finite subgroup in $\mathrm{PGL}_6(\mathbb{C})$ such that $\mathbb{P}^5$ is $G$-birationally rigid. Then there~exists
a~finite primitive subgroup $\widehat{G}\subset\mathrm{SL}_6(\mathbb{C})$ such that $\widehat{G}$ is mapped to $G$ via the natural projection $\mathrm{SL}_6(\mathbb{C})\to\mathrm{PGL}_6(\mathbb{C})$, and $\widehat{G}$ is isomorphic to one of the following groups:
\begin{itemize}
\item[$\mathrm{(VII)}$] $3.\mathfrak{A}_7$,%
\item[$\mathrm{(VIII)}$] $6.\mathfrak{A}_7$,%
\item[$\mathrm{(X)}$]
\begin{itemize}
\item[$\mathrm{(ii)}$] $\mathrm{SL}_2(\mathbf{F}_7)\rtimes\mumu_2$%
\end{itemize}
\item[$\mathrm{(XI)}$] $\mathrm{SL}_2(\mathbf{F}_{11})$,%
\item[$\mathrm{(XII)}$] $\mathrm{SL}_2(\mathbf{F}_{13})$,%
\item[$\mathrm{(XIV)}$]
\begin{itemize}
\item[$\mathrm{(i)}$] $\mathrm{SU}_3(\mathbf{F}_3)$, %
\item[$\mathrm{(ii)}$] $\mathrm{SU}_3(\mathbf{F}_3)\rtimes\mumu_2$,%
\end{itemize}
\item[$\mathrm{(XV)}$]
\begin{itemize}
\item[$\mathrm{(i)}$] $6.\mathrm{PSU}_4(\mathbf{F}_3)$, %
\item[$\mathrm{(ii)}$] $6.\mathrm{PSU}_4(\mathbf{F}_3)\rtimes\mumu_2$,%
\end{itemize}
\item[$\mathrm{(XVI)}$] $2.\mathrm{HaJ}$, where $\mathrm{HaJ}$ is the~Hall--Janko simple group,%
\item[$\mathrm{(XVII)}$]
\begin{itemize}
\item[$\mathrm{(i)}$] $6.\mathrm{PSL}_3(\mathbf{F}_4)$, %
\item[$\mathrm{(ii)}$] $6.\mathrm{PSL}_3(\mathbf{F}_4)\rtimes\mumu_2$.%
\end{itemize}
\end{itemize}
\end{corollary}

\begin{proof}
By Theorem~A, the subgroup $G$ is primitive. Therefore, it follows from Remark~\ref{remark:n-5-Segre} and Corollaries~\ref{corollary:P5-pencil-cases}, \ref{corollary:n-5-6A6-SL27}, and~\ref{corollary:n-5-PSp4} that we can choose the lift $\widehat{G}$ to be isomorphic to one of the groups in the required list.
\end{proof}

\begin{remark}[{cf. \cite[Theorem~3.3]{CheltsovShramov2011-2}}]
\label{remark:n-5-not-real}
If $\widehat{G}$ is a primitive subgroup in $\mathrm{SL}_6(\mathbb{C})$ that is isomorphic to one of the groups listed in Corollary~\ref{corollary:n-5-rigid}, then direct computations show that there are no $G$-invariant quadric hypersurfaces in~$\PP^5$. Moreover, if $\widehat{G}\simeq 3.\mathfrak{A}_7$, then $|\mathcal{O}_{\mathbb{P}^5}(n)|$ does not contain $G$-invariant hypersurfaces for~\mbox{$n\in\{1,2,4,5\}$} (but  there exists a unique $G$-invariant cubic hypersurface in $\mathbb{P}^5$).
If~\mbox{$\widehat{G}\simeq\mathrm{SL}_2(\mathbf{F}_{11})$} or $\widehat{G}\simeq\mathrm{SL}_2(\mathbf{F}_{13})$,
then $|\mathcal{O}_{\mathbb{P}^5}(n)|$ does
not contain $G$-invariant divisors for $n\in\{1,2,3,5\}$
(but there exists a unique $G$-invariant quartic hypersurface in $\mathbb{P}^5$).
If~\mbox{$G\simeq\mathrm{SL}_2(\mathbf{F}_7)\rtimes\mumu_2$}, the same assertion follows from Lemma~\ref{lemma:n-5-SL27}.
If $\widehat{G}$ is isomorphic to one of the groups $6.\mathfrak{A}_7$, $\mathrm{SU}_3(\mathbf{F}_3)$, $6.\mathrm{PSU}_4(\mathbf{F}_3)$, $2.\mathrm{HaJ}$, or $6.\mathrm{PSL}_3(\mathbf{F}_4)$,
then $|\mathcal{O}_{\mathbb{P}^5}(n)|$ does not contain $G$-invariant
divisors for $n\leqslant 5$. Hence, if~$G$ is isomorphic to
$\mathrm{SU}_3(\mathbf{F}_3)\rtimes\mumu_2$, $6.\mathrm{PSU}_4(\mathbf{F}_3)\rtimes\mumu_2$, or $6.\mathrm{PSL}_3(\mathbf{F}_4)\rtimes\mumu_2$, then~$|\mathcal{O}_{\mathbb{P}^5}(n)|$ also does not contain $G$-invariant divisors
for~\mbox{$n\leqslant 5$}.
\end{remark}

Now, we prove Theorem~D. Suppose  $G\simeq\mathrm{PSU}_4(3)\rtimes\mumu_2$.
We can change the lift $\widehat{G}\subset\GL_6(\mathbb{C})$ such that $\widehat{G}$ is the exceptional complex reflection group number $34$ in \cite{OuraSekiguchi,ShephardTodd}, so $\widehat{G}\not\subset\mathrm{SL}_6(\mathbb{C})$ as we assumed earlier.
Let $V$ be the corresponding $6$-dimensional representation of $\widehat{G}$. Then
$$
\mathbb{C}[x_1,\ldots,x_6]^{\widehat{G}}=\mathbb{C}[f_6,f_{12},f_{18},f_{24},f_{30},f_{42}],
$$
where  each $f_i$ is a homogeneous polynomial of degree $i$. Note that 
$$
|\widehat{G}|=6|G|=39191040,
$$
and $\widehat{G}\simeq 6.(\mathrm{PSU}_4(3)\rtimes\mumu_2)$; we point out that it is not isomorphic to the lift of $G$ to $\mathrm{SL}_6(\mathbb{C})$ discussed in Corollary~\ref{corollary:n-5-rigid}. Let $\widehat{G}_0=\widehat{G}\cap\mathrm{SL}_6(\mathbb{C})$, and let $G_0$ be its image  in $\mathrm{PGL}_6(\mathbb{C})$.
Then $G_0\simeq \mathrm{PSU}_4(3)$, this is the only subgroup in $G$ of index less than~$112$,
and the restriction of $V$ to this subgroup is irreducible.

Let $U=\mathbb{P}^5/G$, and let $\pi\colon \mathbb{P}^5\to U$ be the quotient map. Then $U\simeq \mathbb{P}(1,2,3,4,5,7)$ and
$$
\mathcal{O}_{\mathbb{P}^5}(6)\sim_{\mathbb{Q}}\pi^*(\mathcal{O}_U(1)),
$$
Set $(a_0,a_1,a_2,a_3)=(22,36,57,36)$.

\begin{lemma}
\label{lemma:P5-degree-bounds}
Let $Z\subset {\mathbb{P}^5}$ be a $G$-irreducible subvariety of dimension $n\leqslant 3$. Then $\deg(Z)\geqslant a_n$.
\end{lemma}

\begin{proof}
Let $r$ be the number of irreducible components of $Z$, let $C$ be its irreducible component,
and let~$G_C$ be the stabilizer of $C$ in the group $G$. Then $r$ is the index of $G_C$ in $G$, so we may assume that $r\leqslant 56$,
since otherwise we would have $\mathrm{deg}(Z)=r\mathrm{deg}(C)\geqslant r\geqslant 57\geqslant a_{n}$.
Hence, either $r=1$ and $G_C=G$, or $r=2$ and $G_C=G_0$.
In particular, we have $n\ne 0$.

The action of $G_C$ on the subvariety $C$ induces an injective  homomorphism $G_C\to\mathrm{Aut}(C)$. Thus, as cycles in $U$, we have $\pi_*[Z]=|G|[\pi(Z)]$.
On the other hand, by Lemma~\ref{lemma:weighted-degree}, we get
$$
6^n\deg(Z)=\big(\mathcal{O}_{\mathbb{P}^5}(6)\big)^n\cdot Z=\big(\mathcal{O}_U(1)\big)^n\cdot\pi_*[Z]= |G|\,\big(\mathcal{O}_U(1)\big)^n\cdot\pi(Z)\geqslant\frac{|G|}{w_{5-n}w_{6-n}\ldots w_5},
$$
where $(w_0,w_1,w_2,w_3,w_4,w_5)=(1,2,3,4,5,7)$. This gives $\mathrm{deg}(Z)\geqslant a_n$.
\end{proof}

\begin{corollary}
\label{corollary:P5-lc-centers-away-from-Q}
Let $\mathcal{M}\subset|\mathcal{O}_{\mathbb{P}^5}(k)|$ be a mobile $G$-invariant linear system, and let $a=\frac{5-n}{4-n}$ for some non-negative integer $n\leqslant 3$.
Then $({\mathbb{P}^5},\frac{6a}{k}\mathcal{M})$ is log canonical in codimension $5-n$.
\end{corollary}

\begin{proof}
We prove the required assertion by induction on $n$ starting from $n=3$.
Suppose that the singularities of $({\mathbb{P}^5},\frac{6a}{k}\mathcal{M})$ are not log canonical in codimension $5-n$.
Then there are $\lambda<1$ and an irreducible subvariety $C\subset {\mathbb{P}^5}$ of dimension $n\leqslant 3$
such that $C$ is a center of log canonical singularities of $({\mathbb{P}^5},\frac{6a\lambda}{k}\mathcal{M})$.
Let $Z$ be the $G$-orbit of $C$. Then $\mathrm{deg}(Z)\geqslant a_n$ by Lemma~\ref{lemma:P5-degree-bounds}.

If $n=3$ (the base of induction), then it follows from Corti's inequality \cite{CheltsovShramov2015} that $\mathrm{deg}(Z)<36$,
which is impossible, since $\mathrm{deg}(Z)\geqslant a_3=36$.

Now, we suppose that $n<3$. By induction, the pair $({\mathbb{P}^5},\frac{6a}{k}\mathcal{M})$ has log canonical singularities in codimension $4-n$.
Let $Y$ be intersection of $n$ general hyperplanes in~${\mathbb{P}^5}$. Set $\mathcal{M}_Y=\mathcal{M}\vert_{Y}$.
Then $Y$ is a projective space of dimension $5-n$, the intersection $Y\cap Z$ is a finite subset consisting of $\mathrm{deg}(Z)$ points,
and every point of this intersection is an isolated component of the locus $\mathrm{Nklt}(Y,\frac{6a\lambda}{k}\mathcal{M}_Y)$.
Now, applying the Nadel's vanishing theorem \cite[Theorem 9.4.8]{Lazarsfeld} to $(Y,\frac{6a\lambda}{k}\mathcal{M}_Y)$, we get
$$
\deg(Z)=|Z\cap Y|\leqslant h^0\big(Y,\mathcal{O}_Y(m)\big)
$$
for
$$
m=\left\{\aligned
&5 \ \text{if $n=2$},\\
&3 \ \text{if $n=1$},\\
&2 \ \text{if $n=0$}.
\endaligned
\right.
$$
This gives a contradiction with $\mathrm{deg}(Z)\geqslant a_n$.
\end{proof}

Now, we are ready to prove Theorem~D.

\begin{proof}[Proof of Theorem~D]
Suppose ${\mathbb{P}^5}$ is not $G$-birationally superrigid.
Then, it follows from \cite{CheltsovShramov2015} that there exist an integer $k>0$ and a mobile $G$-invariant linear subsystem $\mathcal{M}\subset|\mathcal{O}_{\mathbb{P}^5}(k)|$ such that the singularities of the pair $({\mathbb{P}^5},\frac{6}{k}\mathcal{M})$ are not canonical.
Let $C$ be a center of non-canonical singularities of this  pair that has maximal dimension $n\leqslant 3$.
Then, by \cite[Remark~3.6]{CheltsovSarikyanZhuang2024}, $C$ is a center of non-log canonical singularities of the log pair $(\mathbb{P}^5,\frac{6}{k}\cdot\frac{5-n}{4-n}\mathcal{M})$,
which contradicts Corollary~\ref{corollary:P5-lc-centers-away-from-Q}.
\end{proof}

\section{Six-dimensional projective space}
\label{section:P6}

The goal of this section is to discuss primitive subgroups $G\subset \mathrm{PGL}_7(\mathbb{C})$
and show that~$\mathbb{P}^6$ is not $G$-birationally rigid for many of them.
First, we recall the classification of primitive finite subgroups of $\mathrm{PGL}_7(\mathbb{C})$ from \cite[Theorem~4.1]{Wa69} and \cite[Theorem~I]{Wa70}.
Let $\mathbb{H}_7$ be the~subgroup in~\mbox{$\mathrm{PGL}_{7}(\mathbb{C})$}
generated by the following projective transformations:
\begin{align*}
[x_1:x_2:x_3:x_4:x_5:x_6:x_7]&\mapsto [x_2:x_3:x_4:x_5:x_6:x_7:x_1],\\
[x_1:x_2:x_3:x_4:x_5:x_6:x_7]&\mapsto\\
\mapsto &[e^{\frac{2\pi \sqrt{-1}}{7}} x_1:e^{\frac{4\pi \sqrt{-1}}{7}} x_2:e^{\frac{6\pi \sqrt{-1}}{7}} x_3:e^{\frac{8\pi \sqrt{-1}}{7}} x_4:e^{\frac{10\pi \sqrt{-1}}{7}} x_5:e^{\frac{12\pi \sqrt{-1}}{7}} x_6:x_7],
\end{align*}
and let $\mathbb{N}_7$ be the~normalizer of $\mathbb{H}_7$ in $\mathrm{PGL}_7(\mathbb{C})$. Then \mbox{$\mathbb{H}_7\simeq\mumu_7^2$}, the subgroup $\mathbb{H}_7$ is transitive and imprimitive, and
$$
\mathbb{N}_7/\mathbb{H}_7\simeq\mathrm{SL}_2(\mathbf{F}_{7}).
$$
Moreover, we have $\mathbb{N}_7\simeq\mathbb{H}_7\rtimes\mathrm{SL}_2(\mathbf{F}_{7})$, and the subgroup $\mathbb{N}_7$ is primitive.

Now, let $G$ be a~finite primitive subgroup in $\mathrm{PGL}_7(\mathbb{C})$. Then either
\begin{itemize}
\item[(I)] $G$ is conjugate to a subgroup in $\mathbb{N}_7$ that contains $\mathbb{H}_7$,
\end{itemize}
or there~exists a~finite primitive subgroup $\widehat{G}\subset\mathrm{SL}_7(\mathbb{C})$ such that $\widehat{G}$ is mapped to $G$ via the natural projection $\mathrm{SL}_7(\mathbb{C})\to\mathrm{PGL}_7(\mathbb{C})$, and  $\widehat{G}$ is isomorphic to one of the following groups:
\begin{itemize}
\item[(II)] $\mathrm{PSL}_2(\mathbf{F}_{13})$, %
\item[(III)]
\begin{itemize}
\item[(i)] $\mathrm{SL}_2(\mathbf{F}_8)$,%
\item[(ii)] $\mathrm{SL}_2(\mathbf{F}_8)\rtimes \mumu_3$,%
\end{itemize}
\item[(IV)] $\mathfrak{A}_8$ or $\mathfrak{S}_8$,%
\item[(V)]
\begin{itemize}
\item[(i)] $\mathrm{PSL}_2(\mathbf{F}_7)$,%
\item[(ii)] $\mathrm{PGL}_2(\mathbf{F}_7)$,
\end{itemize}
\item[(VI)]
\begin{itemize}
\item[(i)] $\mathrm{SU}_3(\mathbf{F}_3)$, %
\item[(ii)] $\mathrm{SU}_3(\mathbf{F}_3)\rtimes \mumu_2$,%
\end{itemize}
\item[(VII)] $\mathrm{Sp}_6(\mathbf{F}_2)$.
\end{itemize}
Let us show that $\mathbb{P}^6$ can be $G$-birationally rigid only in case (I).

\begin{lemma}
\label{lemma:n-6-PSL-2-13}
Suppose that $\widehat{G}$ is isomorphic to $\mathrm{PSL}_2(\mathbf{F}_{13})$.
Then the linear system $|\mathcal{O}_{\mathbb{P}^6}(4)|$ contains a unique $G$-invariant
pencil.
\end{lemma}

\begin{proof}
The GAP ID of $\mathrm{PSL}_2(\mathbf{F}_{13})$ is [1092,25], and this group has two
irreducible $7$-dimensional representation.
A direct computation shows that
$|\mathcal{O}_{\mathbb{P}^6}(4)|$ contains a unique $G$-invariant pencil.
\end{proof}

\begin{lemma}
\label{lemma:n-6-PSL-2-8}
Suppose that $\widehat{G}$ is isomorphic to $\mathrm{SL}_2(\mathbf{F}_{8})$
or $\mathrm{SL}_2(\mathbf{F}_{8})\rtimes\mumu_3$.
Then either the linear system~\mbox{$|\mathcal{O}_{\mathbb{P}^6}(4)|$},
or the linear system~\mbox{$|\mathcal{O}_{\mathbb{P}^6}(6)|$} contains at least two $G$-invariant hypersurfaces. In particular, either the linear system~\mbox{$|\mathcal{O}_{\mathbb{P}^6}(4)|$},
or the linear system~\mbox{$|\mathcal{O}_{\mathbb{P}^6}(6)|$} contains a $G$-invariant pencil.
\end{lemma}

\begin{proof}
The GAP ID of $\mathrm{SL}_2(\mathbf{F}_{8})$ is [504,156], and this group has four
irreducible $7$-dimensional representations.
(Note that three of them give rise to the same
embedding of $\mathrm{SL}_2(\mathbf{F}_{8})$ to $\mathrm{SL}_7(\CC)$, and both of the resulting subgroups
$\mathrm{SL}_2(\mathbf{F}_{8})\subset\mathrm{SL}_7(\CC)$ are primitive, because $\mathrm{SL}_2(\mathbf{F}_{8})$ does not have subgroups of index~$7$.)
A direct computation shows that for one of these representations
the vector space of polynomials of degree $6$ contains a five-dimensional
vector subspace which consists of invariant polynomials, while
for each of the remaining three representations
the vector space of polynomials of degree $4$ contains a two-dimensional
vector subspace which consists of invariant polynomials (i.e., there is a unique
$G$-invariant pencil in $|\mathcal{O}_{\PP^6}(4)|$ in this case).
The action of the group $\mathrm{SL}_2(\mathbf{F}_{8})\rtimes\mumu_3$ on this subspace
factors through the action of the abelian group $\mumu_3$, and thus splits into a sum of one-dimensional representations, which correspond to $G$-invariant hypersurfaces.
\end{proof}

\begin{lemma}
\label{lemma:n-6-A8-S8}
Suppose that $\widehat{G}$ is isomorphic to $\mathfrak{A}_8$ or $\mathfrak{S}_8$.
Then the linear system $|\mathcal{O}_{\mathbb{P}^6}(4)|$ contains a unique $G$-invariant  pencil.
\end{lemma}

\begin{proof}
Each of these groups has a unique irreducible $7$-dimensional representation,
which is a summand of the $8$-dimensional permutation representation.
The well-known results about symmetric functions imply that $|\mathcal{O}_{\mathbb{P}^6}(4)|$ contains a unique $G$-invariant pencil.
\end{proof}

\begin{lemma}
\label{lemma:n-6-PSL27}
Suppose that $\widehat{G}$ is isomorphic to $\mathrm{PSL}_2(\mathbf{F}_7)$ or $\mathrm{PGL}_2(\mathbf{F}_7)$. Then the linear system~\mbox{$|\mathcal{O}_{\mathbb{P}^6}(5)|$}  contains a unique $G$-invariant pencil.
\end{lemma}

\begin{proof}
The GAP ID of $\mathrm{PSL}_2(\mathbf{F}_7)$ is $[168,42]$,
and this group has a unique irreducible $7$-dimensional representation.
A direct computation shows that $|\mathcal{O}_{\mathbb{P}^6}(5)|$ contains a unique $\mathrm{PSL}_2(\mathbf{F}_7)$-invariant pencil. If $\widehat{G}\simeq \mathrm{PGL}_2(\mathbf{F}_7)$, then the required assertion follows from the case $\widehat{G}\simeq \mathrm{PSL}_2(\mathbf{F}_7)$.
\end{proof}

\begin{lemma}
\label{lemma:n-6-PSU33}
Suppose that $\widehat{G}$ is isomorphic to $\mathrm{SU}_3(\mathbf{F}_3)$ or $\mathrm{SU}_3(\mathbf{F}_3)\rtimes \mumu_2$. Then the linear system~$|\mathcal{O}_{\mathbb{P}^6}(6)|$ contains a unique $G$-invariant pencil.
\end{lemma}

\begin{proof}
The group $\mathrm{SU}_3(\mathbf{F}_3)$ has three
irreducible $7$-dimensional representations.
A direct computation shows that $|\mathcal{O}_{\mathbb{P}^6}(6)|$ contains a unique $\mathrm{SU}_3(\mathbf{F}_3)$-invariant pencil.
This also implies that~\mbox{$|\mathcal{O}_{\mathbb{P}^6}(6)|$} contains a unique $G$-invariant pencil in the case when $\widehat{G}\simeq\mathrm{SU}_3(\mathbf{F}_3)\rtimes\mumu_2$.
\end{proof}

\begin{lemma}
\label{lemma:n-6-Sp62}
Suppose that $\widehat{G}$ is isomorphic to $\mathrm{Sp}_6(\mathbf{F}_2)$.
Then the linear system $|\mathcal{O}_{\mathbb{P}^6}(6)|$ contains a unique $G$-invariant pencil.
\end{lemma}

\begin{proof}
The group $\mathrm{Sp}_6(\mathbf{F}_2)$ has a unique
irreducible $7$-dimensional representation. A direct
computation shows that $|\mathcal{O}_{\mathbb{P}^6}(6)|$
contains a unique $G$-invariant pencil.
\end{proof}

Summarizing, we obtain the following corollary.

\begin{corollary}
\label{corollary:n-6}
Let $G$ be a~finite subgroup in $\mathrm{PGL}_7(\mathbb{C})$ such that $\mathbb{P}^6$ is $G$-birationally rigid.
Then~$G$ is conjugate to a subgroup in $\mathbb{N}_7$ that contains $\mathbb{H}_7$.
\end{corollary}

\begin{proof}
By Theorem~A, the subgroup $G$ is primitive.
Suppose that $G$ is not conjugate to a subgroup in $\mathbb{N}_7$ that contains $\mathbb{H}_7$. Then it follows from
Lemmas~\ref{lemma:n-6-PSL-2-13}, \ref{lemma:n-6-PSL-2-8}, \ref{lemma:n-6-A8-S8},
\ref{lemma:n-6-PSL27}, \ref{lemma:n-6-PSU33}, and~\ref{lemma:n-6-Sp62}
that the linear system $|\mathcal{O}_{\mathbb{P}^6}(n)|$ contains a $G$-invariant pencil for some $n\in\{4,5,6\}$. Therefore, $\PP^6$ is not $G$-birationally rigid
by Lemma~\ref{lemma:Pn-pencil}.
\end{proof}

Unfortunately, we do not know whether $\mathbb{P}^6$ is $G$-birationally rigid in the case when $G$ is conjugate to a subgroup in $\mathbb{N}_7$ that contains $\mathbb{H}_7$.

\begin{remark}[{cf. Remark~\ref{remark:P4-no-quadrics}}]
\label{remark:P6-no-quadrics}
If $G\subset\mathrm{PGL}_{7}(\mathbb{C})$ is conjugate to a subgroup in
$\mathbb{N}_7$ that contains~$\mathbb{H}_7$, then there are no $G$-invariant quadrics in $\PP^6$.
Indeed, let $\widehat{\mathbb{H}}_7$ denote
the preimage of the group~$\mathbb{H}_7$ in $\mathrm{SL}_7(\CC)$. Then
$\widehat{\mathbb{H}}_7$ is a transitive group of order $343$ with center $Z(\widehat{\mathbb{H}}_7)\simeq \mumu_7$.
Every irreducible representation of $\widehat{\mathbb{H}}_7$ where $Z(\widehat{\mathbb{H}}_7)$ acts non-trivially has dimension $7$. Hence there are no one-dimensional subrepresentations in the space of quadratic polynomials on $\CC^7$.
\end{remark}

\section{Seven-dimensional projective space}
\label{section:P7}

The goal of this section is to prove Theorem~E.
If $G$ is a finite primitive subgroups of $\mathrm{PGL}_8(\mathbb{C})$, then one of the following two cases holds:
\begin{enumerate}
\item either $G$ is one of the few groups described in \cite{Fe75,HuffmanWales},
\item or $G$ leaves invariant $Y\subset\mathbb{P}^7$ such that $Y$ is the image of the Segre embedding $\mathbb{P}^1\times\mathbb{P}^3\hookrightarrow\mathbb{P}^7$ or $Y$ is the image of the Segre embedding $(\mathbb{P}^1)^3\hookrightarrow\mathbb{P}^7$.
\end{enumerate}
In the second case, blowing up $Y$, we obtain a $G$-birational map from $\mathbb{P}^7$ to a $G$-Mori fibre space with a positive dimensional base,
so $\mathbb{P}^7$ is not $G$-birationally solid and is not $G$-birationally rigid.

Among finite primitive subgroups of $\mathrm{PGL}_8(\mathbb{C})$ described in \cite{Fe75,HuffmanWales}, one has the largest order.
If~$G$ is this largest subgroup, then $G\simeq\mathrm{O}_8^+(\mathbf{F}_2)\rtimes\mumu_2$, and Theorem~E says that $\mathbb{P}^7$ is $G$-birationally superrigid.
The proof of this theorem is similar to the proof of Theorem~D presented in Section~\ref{section:P5}.
To start with, let us describe this subgroup and its action on $\mathbb{P}^7$.

Let $\widehat{G}\subset\mathrm{GL}_8(\mathbb{C})$ be a subgroup isomorphic to the Weyl group $W(\mathbb{E}_8)$,
let $V$ be the corresponding  $8$-dimensional representation of $\widehat{G}$, let $G\subset\mathrm{PGL}_8(\mathbb{C})$ be the image of $\widehat{G}$ via the natural projection, let $U=\mathbb{P}^7/G$, and let $\pi\colon \mathbb{P}^7\to U$ be the quotient map.
Then $U\simeq \mathbb{P}(1,4,6,7,9,10,12,15)$ and
$$
\mathcal{O}_{\mathbb{P}^7}(2)\sim_{\mathbb{Q}}\pi^*(\mathcal{O}_U(1)),
$$
since
$$
\mathbb{C}[x_1,\ldots,x_8]^{\widehat{G}}=\mathbb{C}[f_2,f_8,f_{12},f_{14},f_{18},f_{20},f_{24},f_{30}],
$$
where  $f_i$ is a $\widehat{G}$-invariant homogeneous polynomial of degree $i$. Note that
$$
|\widehat{G}|=2|G|=696729600.
$$
Let $\widehat{G}_0=\widehat{G}\cap\mathrm{SL}_8(\mathbb{C})$, let $\widehat{G}_1\subset\widehat{G}$ be a subgroup such that $\widehat{G}_1\simeq W(\mathbb{E}_7)$,
let $\widehat{G}_2\subset\widehat{G}$ be a subgroup such that $\widehat{G}_2\simeq W(\mathbb{D}_8)$, let $G_0$, $G_1$, $G_2$ be the images of these subgroups in $\mathrm{PGL}_8(\mathbb{C})$, respectively.
Then the indices in $G$ of  $G_0$, $G_1$, $G_2$ are $2$, $120$, $135$, respectively. Note that
\begin{align*}
G_0&\simeq\mathrm{O}_8^+(\mathbf{F}_2),\\
G_1&\simeq\mumu_2\times\mathrm{Sp}_6(\mathbf{F}_2),\\
G_2&\simeq\mumu_2^6\rtimes\mathfrak{S}_8.
\end{align*}
and $G_0$, $G_1$, $G_2$ are the only proper subgroups of the group $G$ of index at most~$210$, up to conjugation. Note that
$G_1\cap G_0\simeq\mathrm{Sp}_6(\mathbf{F}_2)$ and $G_2\cap G_0\simeq\mumu_2^6\rtimes\mathfrak{A}_8$,
and these are all proper subgroups of $G_0$ of index at most~$210$, up to conjugation.

\begin{lemma}
\label{lemma:small-index}
Let $\Gamma$ be one of the subgroups $G$, $G_0$, $G_1$, $G_2$, let $Y\subset {\mathbb{P}^7}$ be a $\Gamma$-invariant irreducible positive dimensional subvariety, let $\mathfrak{K}_Y$ be the kernel of the induced homomorphism $\Gamma\to\mathrm{Aut}(Y)$. Then $\mathfrak{K}_Y$ is trivial if $\Gamma\ne G_1$, and $|\mathfrak{K}_Y|\leqslant 2$ if $\Gamma=G_1$.
\end{lemma}

\begin{proof}
The $\widehat{G}$-representation $V$ is irreducible, and its restrictions to $\widehat{G}_0$ and $\widehat{G}_2$ are irreducible.
Thus, if $\Gamma\ne G_1$, then $Y$ spans ${\mathbb{P}^7}$, so $\mathfrak{K}_Y$ is trivial.
On the other hand, the restriction of $V$ to $\widehat{G}_1$ splits as a sum of a $1$-dimensional representation and irreducible $7$-dimensional representations.
Thus, since  $G_1\cap G_0$ is simple, its action on $Y$ must be faithful, which gives $|\mathfrak{K}_Y|\leqslant 2$.
\end{proof}

Let $Q$ be the quadric defined by equation $f_2=0$. Set
$$
(a_0,a_1,a_2,a_3,a_4,a_5)=(36,78,106,92,50,16).
$$

\begin{lemma}
\label{lemma:Q-degree-bounds}
Let $Z\subset Q$ be a $G$-irreducible subvariety of dimension $n\leqslant 5$.
Then $\deg(Z)\geqslant a_n$.
\end{lemma}

\begin{proof}
Let $r$ be the number of irreducible components of $Z$, and let $C$ be one its irreducible component. If $r>1$ and $Z$ is $G_0$-irreducible, then $r\geqslant 120$,
so $\deg(Z)\geqslant r\geqslant 120>a_n$. Hence, we may assume that $n\geqslant 1$, $r\leqslant 2$, and $C$ is $G_0$-invariant.
Then, by Lemma~\ref{lemma:small-index}, the induced action of the group $G_0$ on $C$ is faithful, so a general point of $Z$ has trivial stabilizer in $G$.
Therefore, as cycles in $U$, we have $\pi_*[Z]=|G|[\pi(Z)]$, so by Lemma~\ref{lemma:weighted-degree}, we get
$$
2^n\deg(Z)=\big(\mathcal{O}_{\mathbb{P}^7}(2)\big)^n\cdot Z=\big(\mathcal{O}_U(1)\big)^n\cdot\pi_*[Z]=|G|\,\big(\mathcal{O}_U(1)\big)^n\cdot\pi(Z)\geqslant\frac{|G|}{w_{7-n}w_{8-n}\ldots w_7},
$$
where $(w_0,w_1,w_2,w_3,w_4,w_5,w_6,w_7)=(1,4,6,7,9,10,12,15)$. This gives $\deg(Z)\geqslant a_n$.
\end{proof}

Recall from \cite{CheltsovShramov2008} that
$$
\alpha_G(Q)=\inf\big\{\mathrm{lct}(Q;D)\ \vert\ \text{$D$ is an effective $G$-invariant $\mathbb{Q}$-divisor on $Q$ such that $D\sim_{\mathbb{Q}}-K_Q$}\big\},
$$
where $\mathrm{lct}(Q;D)$ is the log canonical threshold of $D$. Then Lemma~\ref{lemma:Q-degree-bounds} gives the following result.

\begin{corollary}
\label{corollary:alpha}
One has $\alpha_G(Q)=\frac{4}{3}$.
\end{corollary}

\begin{proof}
Let $R\in|\mathcal{O}_{Q}(8)|$ be the divisor $\{f_8=0\}\cap Q$. Then $R$ is $G$-invariant, which gives $\alpha_G(Q)\leqslant \frac{4}{3}$.

Suppose  $\alpha_G(Q)<4/3$. Then there is a $G$-invariant effective $\mathbb{Q}$-divisor $D\sim_{\mathbb{Q}}\mathcal{O}_Q(8)$ such that the log pair $(Q,D)$ is not log canonical.
Let $\lambda=\mathrm{lct}(Q;D)<1$, let $C$ be a center of log canonical singularities of $(Q,\lambda D)$ of maximal dimension $n\leqslant 5$,
and let $Z$ be its $G$-orbit.
Then $n\ne 5$, because $|\mathcal{O}_Q(k)|$ does not contain $G$-invariant divisors for $k\leqslant 7$, and $R$ is reduced.

Let $Y\subset {\mathbb{P}^7}$ be the intersection of $n$  general hyperplanes sections of $Q$. Then $Y$ is a smooth quadric of dimension $6-n$,
the intersection $Z\cap Y$ consists of $\deg(Z)$ distinct points, and each point of the intersection is an isolated component of the locus $\mathrm{Nklt}(Y,\lambda D\vert_{Y})$.
Applying the Nadel's vanishing theorem \cite[Theorem 9.4.8]{Lazarsfeld} to the multiplier ideal sheaf of the pair $(Y,\lambda D\vert_{Y})$, we get
$$
\deg(Z)=|Z\cap Y|\leqslant h^0\big(Y,\mathcal{O}_Y(n+2)\big)=\binom{9}{n+2}-\binom{7}{n}.
$$
But $\deg(Z)\geqslant a_n$ by Lemma~\ref{lemma:Q-degree-bounds}. This is a contradiction. Therefore, $\alpha_G(Q)=\frac{4}{3}$.
\end{proof}

Set $(b_0,b_1,b_2,b_3,b_4,b_5)=(37,85,127,211,166,64)$.

\begin{lemma}[{cf. the proof of Lemma~\ref{lemma:P5-degree-bounds}}]
\label{lemma:degree-bounds}
Let $Z\subset {\mathbb{P}^7}$ be a $G$-irreducible (possibly reducible) subvariety of dimension $n\leqslant 5$ such that $Z\not\subset Q$.
Then $\deg(Z)\geqslant b_n$.
\end{lemma}

\begin{proof}
If $n=0$, then $|Z|\geqslant 120$, so we may assume that $n\geqslant 1$. As in the proof of Lemma~\ref{lemma:Q-degree-bounds},
let $r$ be the number of irreducible components of $Z$, let $C$ be one its irreducible component,
and let~$G_C$ be the stabilizer of $C$ in the group $G$. Then $r$ is the index of $G_C$ in $G$, and we may assume that $r\leqslant 210$, since otherwise $\mathrm{deg}(Z)=r\mathrm{deg}(C)\geqslant r\geqslant 211\geqslant b_{n}$.
Hence, $G_C$ is conjugated to one of the subgroups $G$, $G_0$, $G_1$, $G_2$, so we may assume that $G_C$ is one of these subgroups.

The $G_C$-action on $C$ induces a group homomorphism $G_C\to\mathrm{Aut}(C)$.
Let $\mathfrak{K}_C$ be its kernel. Then, as cycles in $U$, we have
$$
\pi_*[Z]=\frac{|G|}{|\mathfrak{K}_C|}[\pi(Z)].
$$
On the other hand, by Lemma~\ref{lemma:weighted-degree}, we get
$$
2^n\deg(Z)=\big(\mathcal{O}_{\mathbb{P}^7}(2)\big)^n\cdot Z=\big(\mathcal{O}_U(1)\big)^n\cdot\pi_*[Z]= \frac{|G|}{|\mathfrak{K}_C|}\,\big(\mathcal{O}_U(1)\big)^n\cdot\pi(Z)\geqslant\frac{|G|}{|\mathfrak{K}_C|w_0w_{8-n}w_{9-n}\ldots w_7},
$$
where $(w_0,w_1,w_2,w_3,w_4,w_5,w_6,w_7)=(1,4,6,7,9,10,12,15)$. By Lemma~\ref{lemma:small-index}, this gives
$$
2\mathrm{deg}(Z)\geqslant\mathrm{deg}(Z)|\mathfrak{K}_C|\geqslant\left\{\aligned
&96\ \text{if $n=5$},\\
&1344\ \text{if $n=4$},\\
&24192\ \text{if $n=3$},\\
&483840\ \text{if $n=2$},\\
&11612160\ \text{if $n=1$}.
\endaligned
\right.
$$
Hence, $\mathrm{deg}(Z)>b_n$ unless $n=5$ and $|\mathfrak{K}_C|=2$.
In the remaining case, it follows from Lemma~\ref{lemma:small-index} that $G_C=G_1$, so that $r=120$, which gives $\mathrm{deg}(Z)\geqslant 120>b_5=64$.
\end{proof}

\begin{corollary}
\label{corollary:lc-centers-away-from-Q}
Let $\mathcal{M}\subset|\mathcal{O}_{\mathbb{P}^7}(k)|$ be a mobile $G$-invariant linear system, and let $a=\frac{7-n}{6-n}$ for some non-negative integer $n\leqslant 5$. Then the log pair $({\mathbb{P}^7},\frac{8a}{k}\mathcal{M})$ has log canonical singularities in codimension $7-n$ away from the quadric $Q$.
\end{corollary}

\begin{proof}
We prove the required assertion by ``induction'' on $n$ starting from $n=5$.
Suppose that the singularities of the pair $({\mathbb{P}^7},\frac{8a}{k}\mathcal{M})$ are not log canonical in codimension $7-n$ away from $Q$.
Then there exist $\lambda<1$ and an irreducible $n$-dimensional subvariety $C\subset {\mathbb{P}^7}$ such that $C\not\subset Q$,
and $C$ is a center of log canonical singularities of $({\mathbb{P}^7},\frac{8a\lambda}{k}\mathcal{M})$.
Let $Z$ be the $G$-orbit of $C$.
Then $\mathrm{deg}(Z)\geqslant b_n$ by Lemma~\ref{lemma:degree-bounds}. Let us show that this inequality leads to a contradiction.

We start with the ``base of induction'' $n=5$. In this case, $a=2$, and Corti's inequality \cite{CheltsovShramov2015} gives $\mathrm{deg}(Z)<64$,
which contradicts to $\mathrm{deg}(Z)\geqslant b_5=64$.

Now, we suppose that $n<5$. By induction, the pair $({\mathbb{P}^7},\frac{8a}{k}\mathcal{M})$ has log canonical singularities in codimension $6-n$ away from the quadric $Q$.
Let $Y$ be intersection of $n$ sufficiently general hyperplanes in~${\mathbb{P}^7}$, and let $\mathcal{M}_Y=\mathcal{M}\vert_{Y}$.
Then $Y$ is a projective space of dimension $7-n$, the intersection $Y\cap Z$ is a finite subset consisting of $\mathrm{deg}(Z)$ points,
and every point of this intersection is an isolated component of the locus $\mathrm{Nklt}(Y,\frac{8a\lambda}{k}\mathcal{M}_Y)$.
Now, as in the proof Corollary~\ref{corollary:alpha}, applying the Nadel's vanishing theorem \cite[Theorem 9.4.8]{Lazarsfeld} to $(Y,\frac{8a\lambda}{k}\mathcal{M}_Y)$, we get
$$
\deg(Z)=|Z\cap Y|\leqslant h^0\big(Y,\mathcal{O}_Y(m)\big)
$$
for
$$
m=\left\{\aligned
&8\ \text{if $n=4$},\\
&6 \ \text{if $n=3$},\\
&4 \ \text{if $n=2$},\\
&3 \ \text{if $n=1$},\\
&2 \ \text{if $n=0$}.
\endaligned
\right.
$$
This gives a contradiction with $\mathrm{deg}(Z)\geqslant b_n$.
\end{proof}

Now, we are ready to prove Theorem~E.

\begin{proof}[Proof of Theorem~E]
Suppose ${\mathbb{P}^7}$ is not $G$-birationally superrigid.
Then, it follows from \cite{CheltsovShramov2015} that there exist an integer $k>0$ and a mobile $G$-invariant linear subsystem $\mathcal{M}\subset|\mathcal{O}_{\mathbb{P}^7}(k)|$ such that the singularities of the pair $({\mathbb{P}^7},\frac{8}{k}\mathcal{M})$ are not canonical.
Let $C$ be a center of non-canonical singularities of this  pair that has maximal dimension $n\leqslant 5$.
If $C\subset Q$, the singularities of $(Q,\frac{8}{k}\mathcal{M}\vert_{Q})$ are not log canonical by the inversion of adjunction \cite[Theorem~5.50]{KollarMori},
which implies that $\alpha_G(Q)<\frac{4}{3}$, which contradicts Corollary~\ref{corollary:alpha}.
Hence, $C\not\subset Q$, and it follows from \cite[Remark~3.6]{CheltsovSarikyanZhuang2024} that
$C$ is a center of non-log canonical singularities of $({\mathbb{P}^7},\frac{8}{k}\frac{7-n}{6-n}\mathcal{M})$,
which contradicts Corollary~\ref{corollary:lc-centers-away-from-Q}.
\end{proof}

\section{Real projective spaces}
\label{section:real}

Now, we discuss finite groups acting on real projective spaces, and prove Theorems~F and G. We fix a finite subgroup $G\subset \mathrm{PGL}_{n+1}(\RR)$, and a finite subgroup~\mbox{$\widehat{G}\subset\mathrm{GL}_{n+1}(\RR)$} that is mapped surjectively onto $G$.

\begin{remark}
If $n$ is even, then the projection~\mbox{$\theta\colon \mathrm{SL}_{n+1}(\RR)\to  \mathrm{PGL}_{n+1}(\RR)$} is an isomorphism, so we can take $\widehat{G}\simeq G$.
On the other hand, if $n$ is odd, then it is possible that one cannot choose $\widehat{G}$ to be a subgroup of $\mathrm{SL}_{n+1}(\RR)$. For instance, if $n=1$ and
$$
g=
\left(
\begin{array}{cc}
1 & 0 \\
0 & -1
\end{array}
\right)\in \mathrm{PGL}_{2}(\RR),
$$
then $g$ is not contained in the image of $\theta$. In this case one can consider the group
$$
\mathrm{SL}_{n+1}^\pm(\RR)\subset \mathrm{GL}_{n+1}(\RR)
$$
of all matrices whose determinant equals either $1$ or $-1$, note that it maps surjectively
onto~$\mathrm{PGL}_{n+1}(\RR)$, and construct $\widehat{G}$ as the preimage of $G$
in $\mathrm{SL}_{n+1}^\pm(\RR)$.
\end{remark}

First, we present a very simple and well-known observation, which is nevertheless very useful.

\begin{lemma}\label{lemma:quadric}
The projective space $\PP^n_\RR$ contains a smooth pointless $G$-invariant quadric.
\end{lemma}

\begin{proof}
Starting from a positive definite quadratic form and averaging it over $\widehat{G}$, we produce a positive definite
$\widehat{G}$-invariant quadratic form $q$. Let $Q$ be the hypersurface defined by $q$, so that $Q$ is a $G$-invariant quadric.
Then $Q$ is smooth, because $q$ is non-degenerate, and it has no real points, because $q$ is positive definite.
\end{proof}

We will also need another nearly obvious fact.

\begin{lemma}
\label{lemma:R-pencil-from-hypersurfaces}
Suppose that for some positive integer $m$ the linear system $|\mathcal{O}_{\PP^n_{\CC}}(m)|$
contains at least two $G$-invariant hypersurfaces. Then the linear system
$|\mathcal{O}_{\PP^n_{\RR}}(m)|$ contains a $G$-invariant pencil.
\end{lemma}

\begin{proof}
If a $G$-invariant
hypersurface $F\in |\mathcal{O}_{\PP^n_{\CC}}(m)|$ is not invariant under the action
of the non-trivial element $\sigma$ of the Galois group $\mathrm{Gal}(\CC/\RR)$, then
$F$ and $\sigma(F)$ generate a $\mathrm{Gal}(\CC/\RR)$-invariant $G$-invariant pencil in
$|\mathcal{O}_{\PP^n_{\CC}}(m)|$, which gives
a $G$-invariant pencil in $|\mathcal{O}_{\PP^n_{\RR}}(m)|$. On the other hand, if each of
the $G$-invariant hypersurfaces in $|\mathcal{O}_{\PP^n_{\CC}}(m)|$ is also
$\mathrm{Gal}(\CC/\RR)$-invariant, then they are all defined over $\RR$, and so we can take any two of them to generate a $G$-invariant pencil in $|\mathcal{O}_{\PP^n_{\RR}}(m)|$.
\end{proof}

Now, we are going to prove Theorem~F. Suppose that $n\geqslant 2$. We aim to show that $G$ is a primitive subgroup in $\mathrm{PGL}_{n+1}(\mathbb{C})$ if $\mathbb{P}^n_\RR$ is $G$-birationally rigid. To do this, we need to prove several auxiliary lemmas, which are real counterparts of Lemmas~\ref{lemma:intransitive},  \ref{lemma:intransitive-2}, \ref{lemma:Pn-pencil}, \ref{lemma:Pn-ci} proved in Section~\ref{section:preliminaries}.

\begin{lemma}[cf. Lemmas~\ref{lemma:intransitive} and \ref{lemma:intransitive-2}]
\label{lemma:intransitive-auxiliary}
Suppose that there exists a decomposition
\begin{equation*}
\CC^{n+1}=V'\oplus V''
\end{equation*}
such that  the vector subspaces $V'$ and $V''$ are interchanged by the Galois group $\mathrm{Gal}(\CC/\RR)\simeq\mumu_2$, and at least one of the following two conditions holds:
\begin{itemize}
\item either both $V'$ and $V''$ are invariant with respect to the action of the group $\widehat{G}$,
\item or $V'$ and $V''$ are interchanged by $\widehat{G}$.
\end{itemize}
Then $\mathbb{P}^n_\RR$ is not $G$-birationally rigid.
\end{lemma}

\begin{proof}
By assumption, there exist two disjoint linear subspaces $\Lambda',\Lambda''\subset\PP^n_\CC$
such that the union~\mbox{$\Lambda'\cup\Lambda''$} spans $\PP^n_\CC$, the group
$\mathrm{Gal}(\CC/\RR)$ interchanges $\Lambda'$ and $\Lambda''$, and
for any~\mbox{$g\in\widehat{G}$} one has either~\mbox{$g(\Lambda')=\Lambda'$}, or $g(\Lambda')=\Lambda''$. Denote $k=\dim\Lambda'=\dim\Lambda''$, so that $n=2k+1$.
Let
\begin{align*}
\psi'\colon &\PP^n_\CC\dasharrow \Lambda'', \\
\psi''\colon &\PP^n_\CC\dasharrow \Lambda'
\end{align*}
be the linear projections from $\Lambda'$ and $\Lambda''$, respectively.
Set
$$
\psi=\psi'\times \psi''\colon \PP^n\dasharrow \Lambda''\times\Lambda'.
$$
Then $\psi$ fits into a
$G$-equivariant $\mathrm{Gal}(\CC/\RR)$-equivariant commutative diagram
\begin{equation}\label{eq:intransitive-auxiliary}
\xymatrix{
& Y\ar@{->}[ld]_{\pi_{\CC}}\ar@{->}[rd]^{\phi_{\CC}} & \\
\mathbb{P}^n_\CC\ar@{-->}[rr]^\psi  && \Lambda''\times\Lambda'
}
\end{equation}
where $\pi_{\CC}$ is a blow up of $\Lambda'$ and $\Lambda''$, and $\phi_{\CC}$ is a $\PP^1_{\CC}$-bundle.
Since $\Lambda'$ and $\Lambda''$ are interchanged by the Galois group
$\mathrm{Gal}(\CC/\RR)$, we see that
$\Lambda''\times\Lambda'$ is obtained by extension of scalars
from the Weil restriction of scalars~\mbox{$R_{\CC/\RR}\Lambda'$}.
Furthermore, the diagram~\eqref{eq:intransitive-auxiliary}
is obtained by extension of scalars from the $G$-equivariant
commutative diagram
$$
\xymatrix{
& Y\ar@{->}[ld]_\pi\ar@{->}[rd]^{\phi} & \\
\mathbb{P}^n_\RR\ar@{-->}[rr]  && R_{\CC/\RR}\Lambda'}
$$
where $\pi$ is a blow up of the subscheme whose geometrically irreducible components are $\Lambda'$ and~$\Lambda''$,
and $\phi$ is a real $G$-Mori fiber space whose fibers are $\mathbb{P}^1_{\RR}$. Hence, $\PP^n_\RR$ is not $G$-birationally rigid.
\end{proof}

\begin{lemma}[cf. Lemma~\ref{lemma:intransitive}]
\label{lemma:intransitive-R}
Suppose that $\mathbb{P}^n_{\RR}$ is $G$-birationally rigid. Then $G$ is a transitive subgroup of $\mathrm{PGL}_{n+1}(\CC)$.
\end{lemma}

\begin{proof}
Suppose that $G$ is not a transitive subgroup of $\mathrm{PGL}_{n+1}(\CC)$. Then there exists a $G$-invariant linear subspace $\Lambda\subset\mathbb{P}^n_{\CC}$ of dimension~\mbox{$k\leqslant \lfloor \frac{n-1}{2}\rfloor$}. Denote by $\sigma$ the non-trivial element of the Galois group $\mathrm{Gal}(\CC/\RR)\simeq\mumu_2$, and set~\mbox{$\Lambda^\sigma=\sigma(\Lambda)$}.
If the intersection $\Xi=\Lambda\cap\Lambda^\sigma$ is not empty, then
$\Xi$ is a $G$-invariant $\mathrm{Gal}(\CC/\RR)$-invariant linear subspace of $\PP^n_{\CC}$
of dimension at most $\dim\Lambda$. Thus, $\Xi$ is the complexification of a real
$G$-invariant linear subspace $\Xi_{\RR}\subset\PP^n_{\RR}$.
The linear projection from $\Xi_{\RR}$ provides a $G$-equivariant commutative diagram
$$
\xymatrix{
& Y\ar@{->}[ld]_\pi\ar@{->}[rd]^{\phi} & \\
\mathbb{P}^n_\RR\ar@{-->}[rr]^\psi  && \Xi_{\RR}^\perp}
$$
where $\phi$ is a $G$-Mori fiber space over the projective space $\Xi_{\RR}^\perp$
of dimension
$$
n-1-\dim\Xi\geqslant n-1-\left\lfloor\frac{n-1}{2}\right\rfloor\geqslant 1.
$$
This contradicts our assumption that $\mathbb{P}^n_{\RR}$ is $G$-birationally rigid.
Hence, one has~\mbox{$\Lambda\cap\Lambda^\sigma=\varnothing$}.
Moreover, if $2k+1=n$, then $\Lambda\cup\Lambda^\sigma$ spans $\PP^n_{\CC}$,
which contradicts Lemma~\ref{lemma:intransitive-auxiliary}.
Hence, we have~\mbox{$2k+1<n$}.

Let $\Theta\subset\PP^n_{\CC}$ be the linear span of the union $\Lambda\cup\Lambda^{\sigma}$.
Then $\Theta$ is a $G$-invariant $\mathrm{Gal}(\CC/\RR)$-invariant linear subspace of $\PP^n_{\CC}$ of dimension
$$
\dim\Lambda+\dim\Lambda^\sigma+1=2k+1.
$$
Thus, $\Theta$ is the complexification of a real
$G$-invariant subspace $\Theta_{\RR}\subset\PP^n_{\RR}$ of the same dimension.
Furthermore, there exists a $G$-invariant linear subspace $\Theta_{\RR}^\perp\subset\PP^n_{\RR}$
of dimension
$$
n-\dim\Theta-1=n-2k-2\geqslant 0
$$
such that the union $\Theta_{\RR}\cup\Theta_{\RR}^\perp$ spans $\PP^n_{\RR}$.
The linear projection from $\Theta_{\RR}^\perp$ provides a $G$-equivariant commutative diagram
$$
\xymatrix{
& Y\ar@{->}[ld]_\pi\ar@{->}[rd]^{\phi} & \\
\mathbb{P}^n_\RR\ar@{-->}[rr]  && \Theta_{\RR}}
$$
where $\phi$ is a $G$-Mori fiber space over the projective space  $\Theta_{\RR}$
of dimension $2k+1\geqslant 1$. This is impossible, since $\mathbb{P}^n_{\RR}$ is $G$-birationally rigid.
\end{proof}

Surprisingly, a naive analog of Lemma~\ref{lemma:intransitive-2} that would generalize it in the same way as Lemma~\ref{lemma:intransitive-R} generalizes Lemma~\ref{lemma:intransitive} does not hold.

\begin{example}
\label{example:funny}
Let $G=\mathfrak{A}_5$, and let $H\simeq \mathfrak{A}_4$ be its subgroup.
Denote by $I^\prime$ and $I^{\prime\prime}$ the two non-trivial complex conjugate one-dimensional
representations of the subgroup $H$. Set
\begin{equation}
\label{eq:Ind-omega}
V=\mathrm{Ind}_H^G(I^\prime).
\end{equation}
Then the restriction of $V$ to $H$ splits as $I^\prime\oplus I^{\prime\prime}\oplus W$, where $W$ is the unique three-dimensional (real) irreducible representation of the group $H$. Thus, $V$ is defined over~$\mathbb{R}$, and one has
\begin{equation}\label{eq:Ind-omega-2}
V\simeq\mathrm{Ind}_H^G(I^{\prime\prime}).
\end{equation}
In particular, \eqref{eq:Ind-omega} and~\eqref{eq:Ind-omega-2} define two systems of imprimitivity for the $G$-representation~$V$, which are swapped by the complex conjugation.
\end{example}

The next two lemmas are similar to Lemmas~\ref{lemma:Pn-pencil}
and~\ref{lemma:Pn-ci}.

\begin{lemma}[cf. Lemma~\ref{lemma:Pn-pencil}]
\label{lemma:Pn-pencil-R}
Suppose that $|\mathcal{O}_{\mathbb{P}^n_{\RR}}(d)|$ contains a $G$-invariant pencil $\mathcal{P}$ for some~\mbox{$d\leqslant n$}.
Then~$\mathbb{P}_{\RR}^n$ is not $G$-birationally rigid.
\end{lemma}

\begin{proof}
Identical to the proof of Lemma~\ref{lemma:Pn-pencil}.
\end{proof}

\begin{lemma}[cf. Lemma~\ref{lemma:Pn-ci}]
\label{lemma:Pn-ci-R}
Suppose that~$\mathbb{P}^n_{\RR}$ contains a $G$-irreducible (over $\mathbb{R}$) reduced complete intersection $X=F_{d_1}\cap F_{d_2}$ such that $X_{\mathbb{C}}$ has hypersurface singularities and~\mbox{$d_1<d_2\leqslant n$}, where~$F_{d_1}$ and $F_{d_2}$ are hypersurfaces in $\mathbb{P}^n_{\RR}$ of degree $d_1$ and $d_2$, respectively. Then~$\mathbb{P}^n_{\RR}$ is not $G$-birationally rigid.
\end{lemma}

\begin{proof}
Identical to the proof of Lemma~\ref{lemma:Pn-ci}.
\end{proof}

Now we are ready to prove Theorem~F in all dimensions $n \geqslant 4$.

\begin{proposition}
\label{proposition:D}
Suppose that $n\geqslant 4$ and
$\mathbb{P}^n_{\mathbb{R}}$ is \mbox{$G$-birationally} rigid.
Then $G$ is a primitive subgroup in $\mathrm{PGL}_{n+1}(\mathbb{C})$.
\end{proposition}

\begin{proof}
Suppose that $G$ is not a primitive subgroup of $\mathrm{PGL}_{n+1}(\CC)$.
Then $V=\CC^{n+1}$ is an irreducible representation of the group $\widehat{G}$ by Lemma~\ref{lemma:intransitive-R}. By assumption,  there is a non-trivial decomposition
\begin{equation}\label{eq:proposition-D-blocks}
V=\bigoplus_{i=1}^{s}V_i
\end{equation}
such that $\widehat{G}$ transitively permutes the vector subspaces $V_1,\ldots,V_s$.

Denote by $\sigma$ the non-trivial element of the Galois group
$\mathrm{Gal}(\CC/\RR)$. We emphasize that we do not assume that $\sigma$ permutes
$V_1,\ldots,V_s$, cf. Example~\ref{example:funny}. Instead, we construct a $\widehat{G}$-invariant quadratic form that
measures the position of the decomposition~\eqref{eq:proposition-D-blocks} with respect
to its complex conjugate. To do this, let us start by constructing a $\widehat{G}$-invariant Hermitian form on~$V$ (we always use the convention that Hermitian forms are linear in the first argument).

Let $\widehat{G}_1$ be the stabilizer of $V_1$ in $\widehat{G}$. Choose a positive
definite Hermitian form on $V_1$ and average it over $\widehat{G}_1$. This gives us a positive definite $\widehat{G}_1$-invariant Hermitian form $h_1$ on $V_1$.
For every $i$, choose $g_i\in\widehat{G}$ such that $g_i(V_1)=V_i$, and
transport $h_1$ to a Hermitian form $h_i$ on $V_i$ by setting
$$
h_i(g_i(u),g_i(v))=h_1(u,v)
$$
for all $u,v\in V_1$. This definition does not depend on the choice of $g_i$, because
$h_1$ is invariant under the stabilizer of $V_1$. Declare the vector subspaces
$V_1,\ldots,V_s$ to be pairwise orthogonal. We obtain a positive definite Hermitian
form $h$ on $V$. Moreover, the form $h$ is $\widehat{G}$-invariant. Indeed, suppose that $g(V_i)=V_j$. Then
$$
g_j^{-1}\circ g\circ g_i\in\widehat{G}_1,
$$
and the $\widehat{G}_1$-invariance of $h_1$ shows that the restriction of $g$ gives an
isometry from $(V_i,h_i)$ to $(V_j,h_j)$. Since distinct vector subspaces $V_i$ are
$h$-orthogonal, this proves the $\widehat{G}$-invariance of~$h$.

Now, we define another positive definite Hermitian form on $V$ by
$$
h^\sigma(v,w)=\overline{h(\sigma(v),\sigma(w))}.
$$
Since the actions of $\widehat{G}$ and $\sigma$ commute, the Hermitian form $h^\sigma$ is also $\widehat{G}$-invariant. Further, there exists a unique complex linear operator~\mbox{$A\colon V\to V$} such that
$$
h^\sigma(v,w)=h(A(v),w)
$$
for all $v,w\in V$. The $\widehat{G}$-invariance of $h$ and $h^\sigma$ implies that $A$ commutes with the action of $\widehat{G}$. Moreover, since $V$ is irreducible, it follows from Schur's lemma that $A=c\mathrm{Id}_V$ for some $c\in\CC$. On the other hand, the positive definiteness of $h$ and $h^\sigma$ implies that $c$ is a positive real number. Thus, we have $h^\sigma=c h$. Now, applying the same operation once again, we get
$$
h=(h^\sigma)^\sigma=c^2h,
$$
so that $c=1$. Hence, we have
\begin{equation}
\label{eq:proposition-D-real-Hermitian}
h(\sigma(v),\sigma(w))=\overline{h(v,w)}
\end{equation}
for all $v,w\in V$.

Now, we set $B(v,w)=h(v,\sigma(w))$. Then the form $B$ is complex bilinear, since $h$ is conjugate linear in its second argument and $\sigma$ is conjugate linear. Moreover, it follows from \eqref{eq:proposition-D-real-Hermitian} that
$$
B(w,v)=h(w,\sigma(v))=\overline{h(\sigma(v),w)}=h(v,\sigma(w))=B(v,w),
$$
so $B$ is symmetric. Moreover, it is non-degenerate, because $h$ is non-degenerate and $\sigma$ is bijective. Furthermore, for every $g\in\widehat{G}$, we have
$$
B\big(g(v),g(w)\big)=h(g(v),\sigma(g(w)))=h\big(g(v),g(\sigma(w))\big)=B(v,w),
$$
so that $B$ is $\widehat{G}$-invariant. Next, using~\eqref{eq:proposition-D-real-Hermitian}, we get
$$
B(\sigma(v),\sigma(w))=h(\sigma(v),w)=\overline{h(v,\sigma(w))}=\overline{B(v,w)}.
$$
Thus, $B$ is the complexification of a real symmetric bilinear form. Finally, if $v\in\RR^{n+1}$ is non-zero, then $B(v,v)=h(v,v)>0$.
In particular, the quadratic form $q(v)=B(v,v)$  is non-zero, and it is a~complexification of a positive definite real quadratic form.

For a vector $v\in V$, write $v=v_1+\ldots+v_s$ with $v_i\in V_i$. For $1\leqslant i\leqslant s$, we set $q_{ii}(v)=B(v_i,v_i)$. Similarly, for $1\leqslant i<j\leqslant s$, we set $q_{ij}(v)=2B(v_i,v_j)$. Then
$$
q=\sum_{1\leqslant i\leqslant j\leqslant s}q_{ij}.
$$
Let $\Omega=\big\{\{i,j\}\ \big|\ 1\leqslant i\leqslant j\leqslant s,\ q_{ij}\neq 0\big\}$.
Then $\Omega$ is non-empty, because $q\neq 0$. By construction, the group $\widehat{G}$ acts on $\Omega$. Namely, if $g\in\widehat{G}$ induces a permutation $\tau_g$ of the vector subspaces~$V_i$, then
$$
q_{\tau_g(i)\tau_g(j)}(g(v))=q_{ij}(v).
$$
Thus, the group $\widehat{G}$ permutes the non-zero quadratic forms $q_{ij}$ in exactly the same way as it permutes the corresponding unordered pairs of indices.

We claim that $\Omega$ is a single $\widehat{G}$-orbit. Indeed, let $\Omega=\Omega_1\sqcup\ldots\sqcup\Omega_t$ be the orbit decomposition. For every $a\in\{1,\ldots,t\}$, set
$$
Q_a=\sum_{\{i,j\}\in\Omega_a}q_{ij}.
$$
Then each $Q_a$ is non-zero and $\widehat{G}$-invariant. Moreover, the quadratic forms
$Q_1,\ldots,Q_t$ are linearly independent.
Indeed, if we choose coordinates on $V$ which agree with the decomposition~\eqref{eq:proposition-D-blocks} and expand our quadratic forms in these coordinates,
then each monomial in $Q_a$ is a product of two variables which never appear simultaneously
in a monomial in any $Q_b$ with~\mbox{$b\neq a$}. Thus, if $t>1$, then $|\mathcal{O}_{\mathbb{P}^n_{\mathbb{R}}}(2)|$ contains a $G$-invariant pencil by Lemma~\ref{lemma:R-pencil-from-hypersurfaces}. Now, using Lemma~\ref{lemma:Pn-pencil-R}, we conclude that $t=1$, so that $\Omega$ is a single $\widehat{G}$-orbit.
In particular, we have the following dichotomy:
\begin{itemize}
\item either all elements of $\Omega$ are of the form $\{i,i\}$,
\item or all elements of  $\Omega$ are of the form $\{i,j\}$ with~\mbox{$i\neq j$}.
\end{itemize}
In both cases, we set
$$
F=\sum_{\{i,j\}\in\Omega}q_{ij}^2.
$$
Then $F$ is a $\widehat{G}$-invariant quartic form.
Recall that $q^2$ is also a $\widehat{G}$-invariant quartic form.
Thus, if $F$ and $q^2$ are linearly independent, then $|\mathcal{O}_{\mathbb{P}^n_{\mathbb{R}}}(4)|$ contains a $G$-invariant pencil by Lemma~\ref{lemma:R-pencil-from-hypersurfaces}. Hence, it follows from Lemma~\ref{lemma:Pn-pencil-R} that $F$ and $q^2$ are linearly dependent, because $n\geqslant 4$.
On the other hand, if $|\Omega|\geqslant 2$,  then choosing coordinates on $V$ as above, we~see that $q^2$ contains monomials depending on collections of variables which never appear in the~monomials of~the quartic form $F$. Hence, we conclude that $|\Omega|=1$.

Since the group $\widehat{G}$ acts transitively on the set of vector subspaces $V_1,\ldots,V_s$ and $s\ne 1$, the~unique element of the set $\Omega$ cannot be of the form $\{i,i\}$. Hence, possibly after renumbering the vector subspaces, we may assume
$$
\Omega=\big\{\{1,2\}\big\}.
$$
Now, using the transitivity of the $\widehat{G}$-action on the set of vector subspaces $V_1,\ldots,V_s$ again, we see that $s=2$, so we have $V=V_1\oplus V_2$. Further, $B(v,v)=0$ for every $v\in V_1$ and for every $v\in V_2$. Thus, since the bilinear form $B$ is symmetric, we conclude that both restrictions $B\vert_{V_1}$ and $B\vert_{V_2}$ vanish.

Let $w\in V_1$. Then for every $u\in V_1$ we have
$$
h(u,\sigma(w))=B(u,w)=0.
$$
Thus, since $V_1$ and $V_2$ are $h$-orthogonal, we have $\sigma(w)\in V_2$. This gives $\sigma(V_1)=V_2$ and $\sigma(V_2)=V_1$,
which contradicts Lemma~\ref{lemma:intransitive-auxiliary}.
\end{proof}

To finish the proof of Theorem F, let us describe all finite subgroups $G\subset\mathrm{PGL}_{n+1}(\RR)$ such that~$\PP^n_\RR$ is $G$-birationally rigid in the cases when $n=2$ and $n=3$.

\begin{lemma}
\label{lemma:R-dim-2}
Suppose that $n=2$. Then $\PP^2_\RR$ is $G$-birationally rigid if and only if $G\simeq\mathfrak{A}_5$.
\end{lemma}

\begin{proof}
We know from Lemma~\ref{lemma:quadric} that
$\mathbb{P}^2_{\mathbb{R}}$ contains a $G$-invariant smooth pointless conic~$C$. This gives a~monomorphism
$$
G\hookrightarrow\mathrm{Aut}(C)\simeq\mathrm{SO}_3(\mathbb{R}).
$$
On the other hand, we know that every finite subgroup in $\mathrm{SO}_3(\mathbb{R})$ is isomorphic to one of the following groups:
\begin{itemize}
\item the cyclic group $\mumu_m$,
\item the Klein four group $\mumu_2^2$,
\item the dihedral group $\mathfrak{D}_m$ of order $2m\geqslant 6$,
\item the alternating group $\mathfrak{A}_4$,
\item the symmetric group $\mathfrak{S}_4$,
\item the alternating group $\mathfrak{A}_5$.
\end{itemize}
Moreover, finite subgroups of the group $\mathrm{SO}_3(\mathbb{R})$ are conjugate if and only if they are isomorphic.
Hence, arguing as in the proof of Theorem~\ref{theorem:Sakovich} given in~\cite{Sakovics2019}, we see that $\mathbb{P}^2_{\mathbb{R}}$ is $G$-birationally rigid if and only if~\mbox{$G\simeq\mathfrak{A}_5$}.
\end{proof}

Up to conjugation, $\mathrm{PGL}_3(\RR)$ contains is unique subgroup isomorphic to $\mathfrak{A}_5$, and this subgroup is primitive.
Hence, Lemma~\ref{lemma:R-dim-2} implies Theorem~F in dimension $2$.

\begin{remark}
Note also that $\PP^2_\RR$ is not $\mathfrak{A}_5$-birationally superrigid, see e.g.~\mbox{\cite[Lemma B.15]{Cheltsov2014}}. Indeed, $\PP^2_\RR$ contains a unique $\mathfrak{A}_5$-orbit $\Sigma_6$ of length $6$. Blowing up this orbit, we obtain a real form of the famous Clebsch cubic. Then blowing down the strict transforms of six conics passing through quintuples of points of $\Sigma_6$, we obtain a $\mathfrak{A}_5$-equivariant birational map $\PP^2_\RR\dasharrow\PP^2_\RR$.
\end{remark}

In dimension $3$, Theorem F is implied by the following lemmas.

\begin{lemma}
\label{lemma:R-dim-3-aux-2-2}
Suppose that $n=3$, and  $\PP^3_{\mathbb{R}}$ contains a real $G$-invariant curve $Z$ of degree at most~$2$.
Then $\PP^3_\RR$ is not $G$-birationally rigid.
\end{lemma}

\begin{proof}
Follows from Lemmas~\ref{lemma:intransitive-auxiliary} and \ref{lemma:intransitive-R}.
\end{proof}

\begin{lemma}
\label{lemma:R-dim-3-quadric-4-lines}
Suppose that $n=3$, and $\PP^3_{\mathbb{R}}$ contains a real $G$-invariant curve $Z$ of degree $4$ such that its geometric model consists of four disjoint lines, and $Z$ is not contained in a quadric surface. Then $\PP^3_\RR$ is not $G$-birationally rigid.
\end{lemma}

\begin{proof}
Write $Z_\mathbb{C}=L_1+L_2+L_3+L_4$, where $L_1$, $L_2$, $L_3$, $L_4$ are disjoint lines in $\mathbb{P}^3_{\mathbb{C}}$.
Since $L_1$, $L_2$, $L_3$, $L_4$ are not contained in a quadric surface, they have either one or two common transversals. The union of these transversals is a real $G$-invariant curve of degree at most $2$, so the assertion follows from Lemma~\ref{lemma:R-dim-3-aux-2-2}.
\end{proof}

\begin{lemma}
\label{lemma:R-dim-3-aux-1-1-1-1}
Suppose that $n=3$, and  $\PP^3_{\mathbb{R}}$ contains a real zero-dimensional $G$-invariant reduced subscheme $Z$ of length $\leqslant 4$.
Then $\PP^3_\RR$ is not $G$-birationally rigid.
\end{lemma}

\begin{proof}
By Lemma~\ref{lemma:intransitive-R}, we may assume that the length of $Z$ is $4$, and the points of $Z_{\mathbb{C}}$ are in general linear position.
Now, blowing up $Z$ and arguing as in Example~\ref{example:X24}, we obtain a~$G$-birational map from $\PP^3_\RR$ to a real form of the terminal toric Fano threefold $\mathscr{X}$ that is described in Example~\ref{example:X24}. Using this and Lemma~\ref{lemma:R-dim-3-aux-2-2}, we see that $\PP^3_{\RR}$ is not $G$-birationally rigid.
\end{proof}

\begin{lemma}[Theorem F in dimension $3$]
\label{lemma:R-dim-3-short}
Suppose that $n=3$ and $\mathbb{P}^3_{\mathbb{R}}$ is \mbox{$G$-birationally} rigid.
Then $G$ is a primitive subgroup in $\mathrm{PGL}_{4}(\mathbb{C})$.
\end{lemma}

\begin{proof}
Suppose that $G$ is not a primitive subgroup of $\mathrm{PGL}_{4}(\CC)$.
Then $V=\CC^{4}$ is an irreducible representation of the group $\widehat{G}$ by Lemma~\ref{lemma:intransitive-R}. By assumption,  there is a non-trivial decomposition
$$
V=\bigoplus_{i=1}^{s}V_i
$$
such that $\widehat{G}$ transitively permutes the vector subspaces $V_1,\ldots,V_s$. Then either $s=2$ or $s=4$. Let us seek for a contradiction.
As in the proof of Proposition~\ref{proposition:D}, we denote by $\sigma$ the non-trivial element of the Galois group $\mathrm{Gal}(\CC/\RR)$.

Suppose $s=2$. Set $L_1=\mathbb P(V_1)$, $L_2=\mathbb P(V_2)$, $L_1^\prime=\sigma(L_1)$, $L_2^\prime=\sigma(L_2)$.
If \mbox{$\{L_1,L_2\}=\{L_1^\prime,L_2^\prime\}$}, then $L_1+L_2$ is a real $G$-invariant curve, which contradicts Lemma~\ref{lemma:R-dim-3-aux-2-2}. Thus, the lines $L_1$, $L_2$, $L_1^\prime$, $L_2^\prime$ are distinct. Set $Z=L_1+L_2+L_1^\prime+L_2^\prime$. Then $Z$ is a real $G$-invariant curve.
If $Z$ is singular, its singular locus is a real $G$-invariant subscheme of length $\leqslant 4$, which contradicts Lemma~\ref{lemma:R-dim-3-aux-1-1-1-1}. Thus, $Z$ is smooth, so the lines $L_1$, $L_2$, $L_1^\prime$, $L_2^\prime$ are disjoint.
By Lemma~\ref{lemma:R-dim-3-quadric-4-lines}, they are contained in a smooth real $G$-invariant quadric surface $Q$.
Analyzing the $G$-action on $Q$, we see that $Q$ contains a real $G$-invariant curve $Z$ of degree $\leqslant 2$, which contradicts Lemma~\ref{lemma:R-dim-3-aux-2-2}.

Now, we deal with the case $s=4$. As above, set $P_1=\mathbb P(V_1)$, $P_2=\mathbb P(V_2)$, $P_3=\mathbb P(V_3)$, $P_4=\mathbb P(V_4)$. Then it follows from Lemma~\ref{lemma:R-dim-3-aux-1-1-1-1} that
$$
\big\{P_1,P_2,P_3,P_4\big\}\ne\big\{\sigma(P_1),\sigma(P_2),\sigma(P_3),\sigma(P_4)\big\},
$$
which implies that $\{P_1,P_2,P_3,P_4\}\cap\{\sigma(P_1),\sigma(P_2),\sigma(P_3),\sigma(P_4)\}=\varnothing$. For every $i$, let $\ell_i$ be the line passing through the points $P_i$ and $\sigma(P_i)$. Then each line $\ell_i$ is defined over $\mathbb{R}$. Set $C=\ell_1+\ell_2+\ell_3+\ell_4$.
Then $C$ is a $G$-invariant real curve, and
$$
\mathrm{deg}(C)\in\{1,2,4\}.
$$
By Lemma~\ref{lemma:R-dim-3-aux-2-2}, we have $\mathrm{deg}(C)=4$, so the lines $\ell_1$, $\ell_2$, $\ell_3$, $\ell_4$ are distinct. Note also that they are not contained in one plane by Lemma~\ref{lemma:intransitive-R}. This implies that $\mathrm{Sing}(C)$ consists of at most $4$ points, so the lines $\ell_1$, $\ell_2$, $\ell_3$, $\ell_4$ are disjoint by Lemma~\ref{lemma:R-dim-3-aux-1-1-1-1}. Then, by Lemma~\ref{lemma:R-dim-3-quadric-4-lines}, these four lines are contained in a real $G$-invariant quadric surface~$S$.
Since $\ell_1$, $\ell_2$, $\ell_3$, $\ell_4$ are real, $S\simeq\mathbb P^1_{\mathbb R}\times \mathbb P^1_{\mathbb R}$,
which contradicts Lemma~\ref{lemma:Pn-pencil-R}, because $\mathbb{P}^3_{\mathbb{R}}$ contains a $G$-invariant pointless quadric by
Lemma~\ref{lemma:quadric}.
\end{proof}

Actually, the assertion of Lemma~\ref{lemma:R-dim-3-short} can be made more precise.

\begin{proposition}\label{proposition:R-dim-3}
Suppose that $n=3$. Then the following conditions are equivalent:
\begin{enumerate}
\item[$\mathrm{(1)}$] $\mathbb{P}^3_{\mathbb{R}}$ is $G$-birationally rigid;
\item[$\mathrm{(2)}$] $\mathbb{P}^3_{\mathbb{R}}$ is $G$-birationally superrigid;
\item[$\mathrm{(3)}$] $G$ contains a subgroup isomorphic to $\mathfrak{A}_4\times \mathfrak{A}_4$.
\end{enumerate}
\end{proposition}

\begin{proof}
We know from Lemma~\ref{lemma:quadric} that $\mathbb{P}^3_{\mathbb{R}}$ contains a $G$-invariant smooth pointless quadric surface~$Q$.
Observe that $Q\simeq C\times C$ for a pointless conic $C$, because all other real quadrics
have points. Thus, we have a~monomorphism
$$
G\hookrightarrow\mathrm{Aut}(C\times C)\simeq\big(\mathrm{SO}_3(\mathbb{R})\times\mathrm{SO}_3(\mathbb{R})\big)\rtimes\mumu_2.
$$

Suppose $\mathbb{P}^3_{\mathbb{R}}$ is $G$-birationally rigid. Then
$G$ is a primitive subgroup of $\mathrm{PGL}_4(\CC)$ by Lemma~\ref{lemma:R-dim-3-short}.
Now, arguing as in the proof of \cite[Theorem~1.3]{CheltsovShramov2019} (cf.~Theorem~\ref{theorem:CheltsovShramov}), we see that $G$ contains a~subgroup isomorphic to $\mathfrak{A}_4\times \mathfrak{A}_4$, and $\mathbb{P}^3_{\mathbb{R}}$ is $G$-birationally superrigid. This proves $\mathrm{(1)}\Rightarrow\mathrm{(3)}$.

The implication $\mathrm{(2)}\Rightarrow\mathrm{(1)}$ is obvious. Hence, to complete the proof, we must establish the implication~\mbox{$\mathrm{(3)}\Rightarrow\mathrm{(2)}$}.
First,  using the classification of finite subgroups of $\mathrm{SO}_3(\mathbb{R})$, we see that the group~\mbox{$(\mathrm{SO}_3(\mathbb{R})\times\mathrm{SO}_3(\mathbb{R}))\rtimes\mumu_2$}  contains a unique subgroup isomorphic to $\mathfrak{A}_4\times \mathfrak{A}_4$ up to conjugation.
Hence, if $G$ contains a subgroup isomorphic to $\mathfrak{A}_4\times \mathfrak{A}_4$, then it follows from the proof of~\mbox{\cite[Theorem~1.3]{CheltsovShramov2019}}
that $\mathbb{P}^3_{\mathbb{R}}$ is $G$-birationally superrigid.
This gives the implication~\mbox{$\mathrm{(3)}\Rightarrow\mathrm{(2)}$}.
\end{proof}

As we mentioned in the proof of Proposition~\ref{proposition:R-dim-3}, the group $\mathrm{PGL}_4(\RR)$  contains a unique subgroup isomorphic to $\mathfrak{A}_4\times \mathfrak{A}_4$ up to conjugation, and it is a primitive subgroup of $\mathrm{PGL}_4(\CC)$. Anyway, by means of Proposition~\ref{proposition:R-dim-3} or Lemma~\ref{lemma:R-dim-3-short}, we obtain Theorem~F in dimension $3$.
Therefore, Theorem~F is completely proved.

\begin{remark}
It is not difficult to list all finite subgroups in $\mathrm{PGL}_4(\RR)$ that contain a subgroup isomorphic to $\mathfrak{A}_4\times \mathfrak{A}_4$.
Implicitly, this is done in \cite{CheltsovShramov2019}.
\end{remark}

Observe that Theorem F follows from Proposition~\ref{proposition:D} and Lemmas~\ref{lemma:R-dim-2} and~\ref{lemma:R-dim-3-short}.

\medskip
We conclude this section by proving Theorem~G.

\begin{proof}[Proof of Theorem~G]
We suppose that $n\in\{4,5,6\}$. We have to prove that $\mathbb{P}^n_{\mathbb{R}}$ is not $G$-birationally rigid. By Theorem~F, we may assume that  $G$ is a primitive subgroup of $\mathrm{PGL}_{n+1}(\mathbb{C})$. Recall from Lemma~\ref{lemma:quadric} that $G$ leaves invariant a smooth pointless quadric~\mbox{$Q\subset\mathbb{P}^n_{\mathbb{R}}$}.
Thus, if~\mbox{$n=4$}, then, keeping in mind
Remark~\ref{remark:P4-no-quadrics} and using the description of primitive subgroups in $\mathrm{PGL}_{5}(\mathbb{C})$ presented in Section~\ref{section:P4}, we see that $G$ is isomorphic to one of the following $3$ groups:
\begin{center}
$\mathfrak{S}_6$, $\mathfrak{A}_6$, $\mathfrak{S}_5$.
\end{center}
By Lemma~\ref{lemma:S6-A6-S5nst} the linear system $|\mathcal{O}_{\mathbb{P}^4_{\mathbb{C}}}(4)|$ contains a unique $G$-invariant pencil. Since this pencil is unique,
it is preserved by the Galois group
$\mathrm{Gal}(\CC/\RR)$, and hence is defined over $\RR$. Thus, $\mathbb{P}^4_{\mathbb{R}}$ is not $G$-birationally rigid by~Lemma~\ref{lemma:Pn-pencil-R}.

Now, we consider the case $n=5$. Then it follows from the classification of primitive subgroups of~\mbox{$\mathrm{PGL}_6(\CC)$} presented in Section~\ref{section:P5}
together with Lemmas~\ref{lemma:n-5-SL25}, \ref{lemma:n-5-2S5}, \ref{lemma:n-5-3A6}, \ref{lemma:n-5-6A6}, and~\ref{lemma:n-5-SL27},
and Remark~\ref{remark:n-5-not-real}
that either
\begin{itemize}
\item[(I)] $G$ leaves invariant a (possibly complex) Segre cubic scroll in $\mathbb{P}^5_{\mathbb{C}}$,
\end{itemize}
or $G$ is isomorphic to one of the following groups:
\begin{itemize}
\item[(VI)] $\mathfrak{A}_7$ or $\mathfrak{S}_7$,%
\item[(IX)]%
\begin{itemize}
\item[(i)] $\mathrm{PSL}_2(\mathbf{F}_7)$, %
\item[(ii)] $\mathrm{PGL}_2(\mathbf{F}_7)$,%
\end{itemize}
\item[$\mathrm{(XIII)}$]
\begin{itemize}
\item[(i)] $\mathrm{PSp}_4(\mathbf{F}_3)$,%
\item[(ii)] $\mathrm{PSp}_4(\mathbf{F}_3)\rtimes\mumu_2$.
\end{itemize}
\end{itemize}

If $G$ leaves invariant a real Segre cubic scroll in $\mathbb{P}^5_\RR$, then it follows from Remark~\ref{remark:n-5-Segre} that $\mathbb{P}^5_{\mathbb{R}}$ is not $G$-birationally rigid.
If $G$ leaves invariant a non-real Segre cubic scroll $Y\subset \mathbb{P}^5_\CC$, then the complex conjugate Segre scroll $Y^\prime$ is also $G$-invariant. Thus, it follows from Remark~\ref{remark:n-5-Segre} that we have two $G$-equivariant (complex) commutative diagrams:
$$
\xymatrix{
&X\ar@{->}[dl]_{\pi}\ar@{->}[dr]^{\eta}\\%
\mathbb{P}^5_{\mathbb{C}}\ar@{-->}[rr]^{\chi}&&\mathbb{P}^2_{\mathbb{C}}}
$$
and
$$
\xymatrix{
&X^\prime\ar@{->}[dl]_{\pi^\prime}\ar@{->}[dr]^{\eta^\prime}\\%
\mathbb{P}^5_{\mathbb{C}}\ar@{-->}[rr]^{\chi^\prime}&&\mathbb{P}^2_{\mathbb{C}}}
$$
where $\pi$ and $\pi^\prime$ are blow ups of the scrolls $Y$ and $Y^\prime$, respectively,
both $\eta$ and $\eta^\prime$ are $\mathbb{P}^3_{\mathbb{C}}$-bundles, and the maps $\chi$ and $\chi^\prime$ are given by the linear systems of all quadrics passing through $Y$ and $Y^\prime$, respectively. Note that the fibers of $\chi$ and $\chi^\prime$ are three-dimensional linear subspaces in $\mathbb{P}^5_{\mathbb{C}}$.  Thus, taking the product of the maps $\chi$ and $\chi^\prime$, we obtain a real rational map $\rho\colon\mathbb{P}^5_{\mathbb{R}}\dasharrow W$, where~$W$ is a real form of $\mathbb{P}^2_{\mathbb{R}}\times \mathbb{P}^2_{\mathbb{R}}$
obtained as the Weil restriction of scalars~\mbox{$R_{\CC/\RR}(\mathbb{P}^2_\CC$)}. By construction, a general fiber of $\rho$ is an intersection of two (distinct) three-dimensional linear subspaces, so it is either a line or a plane. Thus, equivariantly resolving the indeterminacy of the map $\rho$ and applying relative real $G$-equivariant Minimal Model Program, we obtain a $G$-birational map from $\mathbb{P}^5_{\mathbb{R}}$ to a real $G$-Mori fibre space with a positive dimensional base, which implies that $\mathbb{P}^5_{\mathbb{R}}$ is not $G$-birationally rigid in this case.

If $G$ is isomorphic to $\mathfrak{A}_7$ or $\mathfrak{S}_7$, then by
Lemma~\ref{lemma:n-5-A7-S7} the linear system $|\mathcal{O}_{\mathbb{P}^5_{\mathbb{C}}}(4)|$ contains a unique $G$-invariant pencil. Since this pencil is unique, it is preserved by the Galois group
$\mathrm{Gal}(\CC/\RR)$, and hence is defined over $\RR$. Thus, $\mathbb{P}^5_{\mathbb{R}}$ is not $G$-birationally rigid by~Lemma~\ref{lemma:Pn-pencil-R}.
Similarly, if $G$ is isomorphic to  $\mathrm{PSL}_2(\mathbf{F}_7)$ or to  $\mathrm{PGL}_2(\mathbf{F}_7)$, then it follows from Lemma~\ref{lemma:n-5-PSL27} that $|\mathcal{O}_{\mathbb{P}^5_{\mathbb{R}}}(3)|$ contains a $G$-invariant pencil,  so $\mathbb{P}^5_{\mathbb{R}}$ is again not $G$-birationally rigid by~Lemma~\ref{lemma:Pn-pencil-R}.
Finally, if~$G$~is~isomorphic to~\mbox{$\mathrm{PSp}_4(\mathbf{F}_3)$} or $\mathrm{PSp}_4(\mathbf{F}_3)\rtimes\mumu_2$, then it follows from Lemma~\ref{lemma:n-5-PSp4} that $\mathbb{P}^5_{\mathbb{R}}$ contains a $G$-irreducible reduced complete intersection of a (smooth) quadric and a quintic, so $\mathbb{P}^5_{\mathbb{R}}$ is not $G$-birationally rigid by Lemma~\ref{lemma:Pn-ci-R}.

Thus, to complete the proof of Theorem~G, we may assume that $n=6$. Then it follows from the classification of primitive subgroups of $\mathrm{PGL}_7(\CC)$ presented in Section~\ref{section:P6} together with Remark~\ref{remark:P6-no-quadrics}
that $G$ is isomorphic to one of the following groups:
\begin{itemize}
\item[(II)] $\mathrm{PSL}_2(\mathbf{F}_{13})$, %
\item[(III)]
\begin{itemize}
\item[(i)] $\mathrm{SL}_2(\mathbf{F}_8)$,%
\item[(ii)] $\mathrm{SL}_2(\mathbf{F}_8)\rtimes\mumu_3$,%
\end{itemize}
\item[(IV)] $\mathfrak{A}_8$ or $\mathfrak{S}_8$,%
\item[(V)]
\begin{itemize}
\item[(i)] $\mathrm{PSL}_2(\mathbf{F}_7)$,%
\item[(ii)] $\mathrm{PGL}_2(\mathbf{F}_7)$,
\end{itemize}
\item[(VI)]
\begin{itemize}
\item[(i)] $\mathrm{SU}_3(\mathbf{F}_3)$,
\item[(ii)] $\mathrm{SU}_3(\mathbf{F}_3)\rtimes\mumu_2$,%
\end{itemize}
\item[(VII)] $\mathrm{Sp}_6(\mathbf{F}_2)$.
\end{itemize}
In each of these cases the linear system $|\mathcal{O}_{\mathbb{P}^6_\RR}(n)|$ contains a $G$-invariant pencil for some $n\in\{4,5,6\}$.
This follows in a straightforward way from
Lemmas~\ref{lemma:n-6-PSL-2-13}, \ref{lemma:n-6-A8-S8},
\ref{lemma:n-6-PSL27}, \ref{lemma:n-6-PSU33}, and~\ref{lemma:n-6-Sp62}
in all the cases except for
$G\simeq \mathrm{SL}_2(\mathbf{F}_{8})$
and $G\simeq\mathrm{SL}_2(\mathbf{F}_{8})\rtimes\mumu_3$, while for the latter two groups the same assertion is implied by Lemmas~\ref{lemma:n-6-PSL-2-8}
and~\ref{lemma:R-pencil-from-hypersurfaces}.
Now we can apply Lemma~\ref{lemma:Pn-pencil-R} to show that $\mathbb{P}^6_{\mathbb{R}}$ is not $G$-birationally rigid.
\end{proof}

\begin{remark}
In all the cases one comes across in the proof of Theorem~G, the conjugacy class of
the subgroup $G\subset\mathrm{PGL}_{n+1}(\mathbb{R})$
is uniquely determined by the conjugacy class of $G$
in~$\mathrm{PGL}_{n+1}(\mathbb{C})$. This can be deduced
from~\mbox{\cite[Proposition 2.2]{BeauvillePGL}}.
However, we do not use this fact in the proof.
\end{remark}

\appendix

\section{Combinatorics}
\label{section:combinatorics}

In the sequel, we will always interpret the symmetric group $\mathfrak{S}_r$ as the group of permutations
of the set $\Sigma_r=\{1,\ldots,r\}$.
Recall the following widely-used definition.

\begin{definition}[{cf. Definitions~\ref{definition:primitive-linear-group} and~\ref{definition:transitive-imprimitive}}]
\label{definition:primitive-transitive-permutation-group}
Let $G\subset\mathfrak{S}_r$ be a subgroup. One says that $G$ is \emph{transitive}
if it acts transitively on the set $\Sigma_r$, i.e. $\Sigma_r$ is a single $G$-orbit.
One says that $G$ is primitive if it is transitive, and for any partition
$$
\Sigma_r=\Sigma_r^1\sqcup \ldots\sqcup \Sigma_r^k
$$
into non-empty subsets such that $\Sigma_r^i$ are permuted by $G$,
one has either $k=1$ and $\Sigma_r^1=\Sigma_r$, or~\mbox{$k=r$} and $|\Sigma_r^i|=1$.
\end{definition}

A cyclic subgroup $\mumu_r\subset\mathfrak{S}_r$ is said to be \emph{regular}, if
it is conjugate to the group whose generator acts by the cyclic permutation $(12\ldots r)$.
It is easy to see that $\mumu_r$ is regular if and only if it is transitive.
Thus, for instance, the subgroup of $\mathfrak{S}_6$ generated by the permutation
$(12)(345)$ is a non-regular subgroup isomorphic to~$\mumu_6$.
Similarly, if~\mbox{$r\geqslant 3$}, we say that a dihedral subgroup~\mbox{$\mathfrak{D}_{r}\subset \mathfrak{S}_r$} of order $2r$ is regular,
if it is conjugate to the group generated by the cyclic permutation
$(12\ldots r)$ and the involution which swaps $1$ with $r$, $2$ with $r-1$, etc.
We will denote by $\Sigma_r^{(k)}$ the set of non-ordered $k$-tuples of
distinct elements of $\Sigma_r$.

\begin{lemma}\label{lemma:Sr-orbit-r}
Let $G$ be a primitive subgroup of $\mathfrak{S}_r$. The following assertions hold.

\begin{itemize}
\item[(i)]
Suppose that $G$ has an orbit of length $s$ in $\Sigma_r^{(2)}$.
Then $s\geqslant r$.

\item[(ii)] Suppose that $G$ has an orbit of length~$r$ in $\Sigma_r^{(2)}$.
Then $G$ is a regular subgroup of~$\mathfrak{S}_r$ isomorphic either to $\mumu_r$ or $\mathfrak{D}_{r}$.
\end{itemize}
\end{lemma}

\begin{proof}
Let $\Xi=\{(a_1,b_1),\ldots, (a_s,b_s)\}$ be a $G$-orbit of length $s\leqslant r$ in~$\Sigma_r^{(2)}$.
Since a primitive group is also transitive, the $G$-orbit of $a_1$ coincides with the whole set $\Sigma_r$.
On the other hand, this orbit is contained in the set $\Theta=\{a_1,\ldots,a_s,b_1,\ldots,b_s\}$, so that
$\Theta=\Sigma_r$. Furthermore, it follows from the transitivity of $G$ that each element of $\Sigma_r$
appears among $a_1,\ldots,a_s,b_1,\ldots,b_s$ the same number~$t$ of times. This gives $tr=2s$.

Suppose that $s<r$. Then~$tr=2s$ gives $t=1$. In other words, the pairs $(a_1,b_1),\ldots, (a_s,b_s)$ provide a partition of $\Sigma_r$ into
a union of subsets of cardinality~$2$ permuted by $G$. This is impossible for a primitive group. The obtained contradiction proves assertion~(i).

Suppose that $s=r$. If $r=2$, there is nothing to prove. So we assume $r\geqslant 3$.
From~$tr=2s$, one gets $t=2$. In other words,
each $c\in\Sigma_r$ appears in exactly two pairs $(a_i,b_i)$ and $(a_j,b_j)$. Consider the minimal equivalence relation
under which $a_i$ is equivalent to $b_i$ for all $1\leqslant i\leqslant r$. The equivalence classes under this relation form a partition
of $\Sigma_r$ into subsets permuted by $G$. Since $G$ is primitive, we conclude that the whole $\Sigma_r$ is
the unique equivalence class. This means that after relabelling the elements of $\Sigma_r$ if necessary, one has
$$
\Xi=\{(1,2), (2,3), \ldots, (r,1)\}.
$$

Let $\Delta\subset\mathfrak{S}_r$ be the subgroup which consists of all permutations preserving the $G$-orbit~$\Xi$.
Then~$\Delta$ is a regular dihedral group $\mathfrak{D}_{r}$. Indeed, such a group is obviously contained in~$\Delta$. On the other hand,
one has
$$
|\Delta|=r\cdot |\Delta_{(1,2)}|,
$$
where $\Delta_{(1,2)}$ is the stabilizer of the non-ordered pair $(1,2)$ in $\Delta$.
On the other hand, the stabilizer~$\Delta_{(1,2)}'$ of the \emph{ordered} pair consisting of $1$ and $2$ in $\Delta$
fixes $1$ and $2$, hence fixes $3$ as the unique element different from $1$ which is contained together with $2$ in a pair from~$\Xi$, etc;
in other words, the stabilizer $\Delta_{(1,2)}'$ is trivial. This gives
$$
|\Delta_{(1,2)}|\leqslant 2|\Delta_{(1,2)}'|=2,
$$
and so $\Delta\simeq\mathfrak{D}_{r}$.

Thus, $G$ is a subgroup of a regular dihedral group $\Delta\simeq\mathfrak{D}_{r}$. Since $G$ has an orbit of length $r$,
its order cannot be smaller than $r$. In other words, $G$ either coincides with~$\Delta$, or is a subgroup of index $2$ therein.
However, any subgroup of index $2$ in $\mathfrak{D}_{r}$ is either the cyclic group of order $r$ (which is regular provided that
$\mathfrak{D}_{r}$ is regular), or a dihedral group
of order $r$ if $r$ is even. The latter subgroup of~$\mathfrak{S}_r$ is not transitive (and in particular not primitive), which means that $G$ is always a regular subgroup of $\mathfrak{S}_r$.
This proves assertion~(ii).
\end{proof}

\begin{remark}
\label{remark:prime}
One can strengthen the assertion of Lemma~\ref{lemma:Sr-orbit-r}(ii) by observing that
a regular  dihedral subgroup of $\mathfrak{S}_r$ is primitive if and only if $r$ is a prime number.
In other words, under the assumptions of Lemma~\ref{lemma:Sr-orbit-r}, the group $G$ is
a regular subgroup of $\mathfrak{S}_r$ isomorphic either to $\mumu_r$ or~$\mathfrak{D}_{r}$,
and $r$ is prime.
\end{remark}

Lemma~\ref{lemma:Sr-orbit-r} easily implies the following.

\begin{corollary}\label{corollary:Sr-orbit-r}
Let $G$ be a primitive subgroup of $\mathfrak{S}_r$. Suppose that $G$ is neither
a regular subgroup~$\mumu_r$ nor a regular subgroup $\mathfrak{D}_{r}$ of $\mathfrak{S}_r$.
Then the minimal length of the $G$-orbit in $\Sigma_r^{(r-2)}$ is greater than $r$.
\end{corollary}

\begin{proof}
The lengths of the $G$-orbits in~$\Sigma_r^{(r-2)}$ are the same as
the lengths of the $G$-orbits in~$\Sigma_r^{(2)}$.
Now the assertion follows from
Lemma~\ref{lemma:Sr-orbit-r}.
\end{proof}

\section{Luminy toric Fano varieties}
\label{app:toric-fano}

Let us use notations and assumptions of Remark~\ref{remark:do-nor-understand-M}.
Suppose that $n\geqslant 3$. For $i<j$, put $\Lambda_{ij}=\{x_i=x_j=0\}$, and let $E_{ij}$ be the exceptional divisor of the blow up of $\mathbb{P}^n$ along $\Lambda_{ij}$.

\begin{theorem}
\label{thm:appendix-toric-fano}
The map $\psi$ is birational, and $\mathscr X$ is a normal toric Fano variety
with terminal singularities such that
$$
(-K_{\mathscr X})^n=\frac{n(n+1)^n}{2^{n-1}(n-1)}.
$$
Its torus invariant prime divisors are in bijection with the subspaces
$\Lambda_{ij}$. If $D_{ij}$ corresponds to $\Lambda_{ij}$, then
$\operatorname{ord}_{D_{ij}}=\operatorname{ord}_{E_{ij}}$ under the identification
of function fields induced by $\psi$.
\end{theorem}

\begin{remark}\label{rem:appendix-toric-action}
Since $\mathcal{M}$ is $\mathbb W$-invariant, $\psi$ is
$\mathbb W$-equivariant, and the correspondence $D_{ij}\leftrightarrow\Lambda_{ij}$ in Theorem~\ref{thm:appendix-toric-fano}
is also equivariant. In particular, the action of $\mathfrak S_{n+1}$ on
the toric divisors of $\mathscr X$ is its natural action on the non-ordered
pairs of distinct indices, or, equivalently, on the codimension two toric
strata of~$\mathbb{P}^n$.
\end{remark}

Let us prove Theorem~\ref{thm:appendix-toric-fano}. Put $m=n+1$ and
$p=(1,\ldots,1)\in\mathbb Z^m$. Let $e_1,\ldots,e_m$ be the standard basis,
let $N=\mathbb Z^m/\mathbb Zp$, and let
$$
M=\{u\in\mathbb Z^m \mid \sum_{i=1}^{m} u_i=0\}
$$ 
be the dual lattice.
Write $v_i=[e_i]\in N$ and $v_{ij}=v_i+v_j$.
The exponent set of the monomials spanning $\mathcal M$ is
$$
\mathcal A=
\left\{p+b \mid b\in\mathbb Z_{\ge0}^m,\ \sum_{i=1}^m b_i=m-2\right\}
\cup\left\{2p-2e_i \mid 1\leqslant i\leqslant m\right\}.
$$
Set $A=\operatorname{conv}(\mathcal A)$ and
$P=\operatorname{conv}\{v_{ij} \mid i<j\}$. We use the inward normal fan convention
and the polar inequalities $\langle u,v\rangle\ge-1$.
The singularity criteria used below can be found in~\cite{CoxLittleSchenckToric}.

\begin{lemma}
\label{lem:appendix-toric-polytopes}
The origin is an interior point of $P$. If $Q\subset M_{\mathbb R}$ is its
polar, then
\begin{equation}\label{eq:appendix-toric-homothety}
A=\frac{2m-2}{m}p+\frac{2(m-2)}m Q
=\left\{a\in\mathbb R^m \mid \sum_{i=1}^m a_i=2m-2,\ a_i+a_j\ge2\text{ for }i<j\right\}.
\end{equation}
The primitive inward facet normals of $A$ are precisely the $v_{ij}$.
\end{lemma}
\begin{proof}
The projection $\mathbb R^m\to N_{\mathbb R}$ identifies the hyperplane
$\sum_i z_i=2$ with $N_{\mathbb R}$. Under this identification, $P$ is the image
of 
$$
\{z\in[0,1]^m \mid \sum_{i=1}^m z_i=2\}.
$$
The vertices of this polytope have zero and one coordinates, with exactly two ones. Its barycenter $(2/m)p$ maps
to the origin, which is therefore interior to $P$.

Put $\ell_i=e_i-p/m$. The facets of this polytope are $z_i=0$ and $z_i=1$;
both have dimension $m-2$. Their polar vertices are, respectively,
$\alpha_i=(m/2)\ell_i$ and $\beta_i=-m\ell_i/(m-2)$.
Thus, we have
$$
Q=\operatorname{conv}\{\alpha_i,\beta_i \mid 1\leqslant i\leqslant m\}=\{u\in M_{\mathbb R} \mid u_i+u_j\ge-1\text{ for }i<j\}.
$$
The affine transformation in~\eqref{eq:appendix-toric-homothety} sends
$\alpha_i$ to $p+(m-2)e_i$ and $\beta_i$ to $2p-2e_i$.
These are exactly the vertices of $\operatorname{conv}(\mathcal A)$, which
proves the equality and the displayed inequalities. Each pair inequality
defines a facet: take $a_i=a_j=1$ and $a_k=2$ for $k\notin\{i,j\}$ to make
just that inequality an equality. Its inward normal is $v_{ij}$, which is
primitive since $\langle e_i-e_k,v_{ij}\rangle=1$ for $k\notin\{i,j\}$.
\end{proof}

\begin{lemma}\label{lem:appendix-toric-normality}
The variety $\mathscr X$ is projectively normal, the map $\psi$ is birational,
and the fan of $\mathscr X$ consists of the cones over the proper faces of $P$.
\end{lemma}
\begin{proof}
Let $k\ge1$ and $a\in kA\cap\mathbb Z^m$, and put $b_i=a_i-k$.
Then 
$$
\sum_{i=1}^m b_i=k(m-2)
$$ and $b_i+b_j\geqslant 0$ for $i<j$. We also have $a_i\geqslant 0$, since summing the inequalities involving a fixed $i$ gives
$$
(m-2)a_i+\sum_{j=1}^m a_j\geqslant 2k(m-1).
$$
Consequently, at most one $b_i$ is negative.

If all $b_i\geqslant 0$, split the $k(m-2)$ integral units among $k$ nonnegative
integral vectors of coordinate sum $m-2$. Adding $p$ to each vector expresses
$a$ as a sum of $k$ elements of $\mathcal A$.
Otherwise, let $b_i=-t<0$. Then $1\leqslant t\leqslant k$ and $b_j\geqslant t$ for $j\ne i$.
Subtract $t$ copies of $2p-2e_i$ from $a$, and then subtract $(k-t)p$.
The remaining vector has coordinate zero at $i$, coordinates $b_j-t\geqslant 0$
elsewhere, and coordinate sum $(k-t)(m-2)$. Splitting these units into
$k-t$ vectors as before gives the required decomposition. This also covers
$t=k$, when the remaining vector is zero. Taking $k=1$ shows that
$A\cap\mathbb Z^m=\mathcal A$.

Choose a monomial $h$ of degree $m-3$, and set $s=x_1\cdots x_m$. Then both $shx_i$ and $shx_j$ belong to
$\mathcal{M}$, and their ratio is $x_i/x_j$. Hence the differences of elements
of $\mathcal A$ generate $M$, and $\psi$ induces an isomorphism of the dense
tori. Moreover, the group generated by the pairs $(a,1)$ with
$a\in\mathcal A$ is
$$
\{(a,k)\in\mathbb Z^m\times\mathbb Z \mid \sum_{i=1}^m a_i=(2m-2)k\}.
$$
The decomposition just proved says that their semigroup is the intersection
of this group with its real cone. It is saturated, so the homogeneous
coordinate ring $\mathbb C[x^a\,\mid\,a\in\mathcal A]$ is normal.

The affine chart at a vertex $a_0$ has the semigroup generated by
$a-a_0$, with $a\in\mathcal A$; its real cone is the tangent cone of $A$
at $a_0$. Normality identifies this chart with the affine toric variety
of the dual cone. Thus the fan of $\mathscr X$ is the normal fan of $A$.
By~\eqref{eq:appendix-toric-homothety}, it is the normal fan of $Q$, or,
equivalently, the fan over the proper faces of $P$.
\end{proof}

\begin{lemma}
\label{lem:appendix-toric-terminal}
The variety $\mathscr X$ is Fano variety that has terminal singularities. 
\end{lemma}

\begin{proof}
The toric divisors $D_{ij}$ correspond to the primitive rays $v_{ij}$, and
$-K_{\mathscr X}=\sum_{i<j}D_{ij}$. The associated rational divisor polytope
is $Q$. Since the fan of $\mathscr X$ is the normal fan of $Q$, the local
Cartier data are its rational vertices and its support function is strictly
convex. Hence $-K_{\mathscr X}$ is $\mathbb Q$-Cartier and ample.
The coordinates of $\alpha_i$ are $(m-1)/2$ at $i$ and $-1/2$ elsewhere;
those of $\beta_i$ are $-(m-1)/(m-2)$ at $i$ and $1/(m-2)$ elsewhere.

We claim that $P\cap N=\{0\}\cup\{v_{ij} \mid i<j\}$.
Represent a lattice class in $P$ by $z\in\mathbb Z^m$. Its representative
in $\sum_i z_i=2$ has the form $z+tp$ and belongs to $[0,1]^m$.
Thus the coordinates of $z$ differ by at most one. If they are all equal,
the class is zero. Otherwise their translated minimum and maximum are
zero and one, so all coordinates of $z+tp$ are zero or one. Exactly two
are one, giving a vertex $v_{ij}$.

To conclude, let $\sigma$ be the cone over a facet $F$ of $P$, and let
$h_\sigma$ be the rational linear function equal to one on $F$.
A new primitive ray $w\in\sigma$ on a smooth toric resolution has
discrepancy $\langle h_\sigma,w\rangle-1$. If this were nonpositive,
$w$ would be a nonzero lattice point of $\operatorname{conv}(0,F)\subset P$
other than a vertex, contrary to the claim. Thus every exceptional divisor
on a smooth toric resolution has positive discrepancy, and $\mathscr X$
is terminal.
\end{proof}

\begin{lemma}\label{lem:appendix-toric-degree}
Let $H$ be the hyperplane class of the given embedding of $\mathscr X$.
Then 
$$
H\sim_{\mathbb Q}\frac{2(n-1)}{n+1}(-K_{\mathscr X})
$$
and $H^n=2n(n-1)^{n-1}$.
\end{lemma}
\begin{proof}
By Lemma~\ref{lem:appendix-toric-normality}, the degree-$k$ monomials in the
homogeneous coordinate ring are indexed by $kA\cap\mathbb Z^m$.
In the notation of its proof, the points with all $b_i\ge0$ number
$\binom{(m-2)k+m-1}{m-1}$. For a fixed index $i$ with $b_i=-t<0$,
the remaining $m-1$ nonnegative coordinates have sum $(k-t)(m-2)$.
Therefore, the Hilbert function is
$$
\binom{(m-2)k+m-1}{m-1}+m\sum_{q=0}^{k-1}\binom{(m-2)q+m-2}{m-2},
$$
which gives $H^n=2(m-1)(m-2)^{m-2}=2n(n-1)^{n-1}$.

To compare the divisor classes, choose $a_0\in\mathcal A$ and translate
$A$ by $-a_0$. Equation~\eqref{eq:appendix-toric-homothety} identifies this
polytope with $\frac{2(m-2)}m Q$ up to translation by an element of
$M_{\mathbb Q}$. Such a translation changes a toric divisor by a
$\mathbb Q$-principal divisor. This gives the asserted
$\mathbb Q$-linear equivalence.
\end{proof}

\begin{proof}[Proof of Theorem~\ref{thm:appendix-toric-fano}]
By Lemmas~\ref{lem:appendix-toric-normality} and~\ref{lem:appendix-toric-terminal},
$\mathscr{X}$ is a Fano variety with terminal singularity. 
The degree formula follows from Lemma~\ref{lem:appendix-toric-degree}.
Finally, $\Lambda_{ij}$ corresponds to the cone generated by $v_i,v_j$
in the fan of $\mathbb P^n$. Its ordinary blow up inserts the primitive
ray $v_i+v_j$. These are exactly the rays of $\mathscr X$, by
Lemma~\ref{lem:appendix-toric-normality}. Consequently, $D_{ij}$ and $E_{ij}$
define the same normalized divisorial valuation, as required.
\end{proof}

The variety $\mathscr{X}$ is isomorphic to the toric Fano variety $X$ of the first construction in Section~\ref{section:technical}. Indeed, its toric boundary has the same valuations, hence the same primitive rays in $N$. Its anticanonical polytope is therefore $Q$, and ampleness forces its fan to be the normal fan of $Q$.

\end{document}